\documentclass[10pt,letterpaper,reqno]{amsart}
\usepackage{amssymb,a4wide,amsthm,amsmath,amsfonts,xcolor}
\usepackage{mathrsfs}
\usepackage{amsthm}
\usepackage{hyperref}
\usepackage{color}
\usepackage{enumitem}
\usepackage{fouridx}
\usepackage[utf8]{inputenc}
\usepackage[T1]{fontenc}

\numberwithin{equation}{section}

\newtheorem{theorem}{Theorem}[section]
\newtheorem{lemma}[theorem]{Lemma}
\newtheorem{proposition}[theorem]{Proposition}

\newtheorem{corollary}[theorem]{Corollary}
\newtheorem{definition}[theorem]{Definition}
\newtheorem{main result}{Main Result}

\newtheorem{remark}[theorem]{Remark}

\begin{document}
\title[Martingale solution of Nematic Liquid Crystals driven by Pure Jump Noise]{Existence of weak martingale solution of Nematic Liquid Crystals driven by Pure Jump Noise}

\author[Z. Brze\'zniak]{Zdzis{\l}aw Brze\'zniak}
	\address{Department of Mathematics\\
		The University of York\\
		Heslington, York YO10 5DD, UK} \email{zb500@york.ac.uk}
	\author[U. Manna]{Utpal Manna}
	\address{School of Mathematics, Indian Institute of Science Education and Research Thiruvananthapuram\\ Trivandrum 695016, INDIA}
	\email{manna.utpal@iisertvm.ac.in}
	\author[A. A. Panda]{Akash Ashirbad Panda}
	\address{School of Mathematics, Indian Institute of Science Education and Research Thiruvananthapuram\\ Trivandrum 695016, INDIA}
	\email{akash.panda13@iisertvm.ac.in}


\begin{abstract}
In this work we consider a stochastic evolution equation which describes the system governing the nematic liquid crystals driven by a pure jump noise. The existence of a martingale solution is proved for both 2D and 3D cases. The construction of the solution is based on the classical Faedo-Galerkin approximation, compactness method and the Jakubowski's version of the Skorokhod representation theorem for non-metric spaces. We prove the solution is pathwise unique and further establish the existence of a strong solution in the 2-D case. 
\end{abstract}
\date{\today}
\maketitle
\tableofcontents

\section{Introduction}

The obvious states of matter are the solid, the liquid and the gaseous state. The liquid crystal is an intermediate state of a matter, in between the liquid and the crystalline solid, i.e. it must possess some typical properties of a liquid as well as some crystalline properties. The nematic liquid crystal phase is characterized by long-range orientational order, i.e. the molecules have no positional order but tend to align along a preferred direction. Much of the interesting phenomenology of liquid crystals involves the geometry and dynamics of the preferred axis, which is defined by a vector $\mathbf{d}.$ This vector is called a director. Since the sign as well as the magnitude of the director has no physical significance, it is taken to be unity.

 One can observe the flow of nematic liquid crystals as slowly moving particles where the alignment of the particles and the velocity of the fluid sway each other. The all-inclusive description of the physical relevance of liquid crystals has been illustrated in Chandrasekhar \cite{Sc}, Warner and Terentjev \cite{WT} and Gennes and Prost \cite{GP}. In the 1960's,  Ericksen \cite{Eri} and Leslie \cite{Les} demonstrated the hydrodynamic theory of liquid crystals. Moreover, they expanded the continuum theory which has been widely used  by most researchers to design the dynamics of the nematic liquid crystals. Inspired by this theory, the most fundamental form of dynamical system representing the motion of nematic liquid crystals has been procured by Lin and Liu \cite{LL}. This system can be derived as
\begin{align}\label{el}
& \frac{\partial \mathbf{u}}{\partial t} + (\mathbf{u} \cdot \nabla)\mathbf{u} - \mu \Delta \mathbf u + \nabla p = - \lambda \nabla \cdot \big( \nabla \mathbf{d} \odot \nabla \mathbf{d} \big),
\\ & \qquad \qquad \qquad \qquad \quad \nabla \cdot \mathbf{u} = 0,
\\ &\qquad \qquad \frac{\partial \mathbf{d}}{\partial t} + (\mathbf{u} \cdot \nabla)\mathbf{d} = \gamma \left( \Delta \mathbf{d} + |\nabla \mathbf{d}|^2 \mathbf{d} \right),
\\ &\qquad \qquad \qquad \qquad \quad |\mathbf{d}|^2 = 1. \label{eq4}
\end{align}
This holds in $\mathbb{O}_{T} := (0, T] \times \mathbb{O},$  where $\mathbb{O} \subset \mathbb{R}^\mathbf{n}, \mathbf{n} = 2, 3.$ Where the vector field $\mathbf{u} := \mathbf{u}(x, t)$ denotes the velocity of the fluid, $\mathbf{d} := \mathbf{d}(x, t)$ is the director field, $p$ denoting the scalar pressure. The symbol $\nabla \mathbf{d} \odot \nabla \mathbf{d}$ is the $\mathbf{n} \times \mathbf{n}$-matrix with the entries
\[\big[ \nabla \mathbf{d} \odot \nabla \mathbf{d} \big]_{i, j} = \sum_{k=1}^{\mathbf{n}} \partial_{x_i} \mathbf{d}^{(k)} \partial_{x_j} \mathbf{d}^{(k)}, \quad i, j = 1, \cdots, \mathbf{n}. \]

We equip the system with the initial and boundary conditions respectively as follows
\begin{align}
\mathbf{u}(0) = \mathbf{u}_0 \ \ \text{and} \ \  \mathbf{d}(0) = \mathbf{d}_0,
\\ \mathbf{u} = 0 \ \ \text{and} \ \ \frac{\partial \mathbf{d}}{\partial \textsl{n}} = 0 \ \ \text{on} \ \  \partial \mathbb{O}. \label{inboc}
\end{align}
Where the vector $\textsl{n}(x)$ is the outward unit normal vector at each point $x$ of $\mathbb{O}$.

It is the most simple mathematical model one can acquire without disrupting the basic nonlinear structure. Though \eqref{el}-\eqref{eq4} is a much simplified version of the equations used in Ericksen-Leslie theory, it preserves many crucial physical attributes of the nematic liquid crystals. Since
\begin{align*}
\Delta \mathbf{d} + |\nabla \mathbf{d}|^2 \mathbf{d} = \mathbf{d} \times \big( \Delta \mathbf{d} \times \mathbf{d}\big),
\end{align*}
we obtain non-parabolicity in \eqref{eq4}. Also we have high nonlinearity in \eqref{el} due to the term $\nabla \cdot \big( \nabla \mathbf{d} \odot \nabla \mathbf{d} \big)$. These two frame the main mathematical difficulties while studying the above model. So the problem \eqref{el}-\eqref{inboc} form a fully nonlinear system of Partial Differential Equations with constraint.  Since the system \eqref{el}-\eqref{inboc} comprise of the Navier-Stokes equations as a subsystem, in general one can not expect any superior results than those for the Navier-Stokes equations.

To overcome the difficulty, we have a closely related system of \eqref{el}-\eqref{inboc}, which eases the constraint $|\mathbf{d}|^2 = 1$ and the gradient nonlinearity $|\nabla \mathbf{d}|^2 \mathbf{d},$ due to the suggestion of Lin and Liu \cite{LL} in 1995. They have worked on the following model
\begin{align}\label{el1}
& \frac{\partial \mathbf{u}}{\partial t} + (\mathbf{u} \cdot \nabla)\mathbf{u} - \mu \Delta \mathbf u + \nabla p = - \lambda \nabla \cdot \big( \nabla \mathbf{d} \odot \nabla \mathbf{d} \big) \ \ \text{in} \ \ (0, T] \times \mathbb{O},
\\ & \qquad \qquad \qquad \qquad \quad \nabla \cdot \mathbf{u} = 0 \ \ \text{in} \ \ [0, T] \times \mathbb{O},
\\ &\qquad \qquad \frac{\partial \mathbf{d}}{\partial t} + (\mathbf{u} \cdot \nabla)\mathbf{d} = \gamma \left( \Delta \mathbf{d} - \frac{1}{\epsilon^{2}} \big( |\mathbf{d}|^2 -1 \big) \mathbf{d} \right) \ \ \text{in} \ \ (0, T] \times \mathbb{O},
\\ &\qquad \qquad \qquad \quad \mathbf{u}(0) = \mathbf{u}_0 \ \ \text{and} \ \  \mathbf{d}(0) = \mathbf{d}_0 \ \ \text{in} \ \ \mathbb{O}. \label{el1eq4}
\end{align}
 Where $\epsilon > 0$ is an arbitrary constant. Though it is a much simpler version of the previous system  \eqref{el}-\eqref{inboc}, still it is a captivating as well as a toilsome problem. Many have done meticulous work on the systems \eqref{el}-\eqref{inboc} and \eqref{el1}-\eqref{el1eq4} (e.g. see \cite{GR, LiL, LL, LW, LLW, Mh, MdM, Ss}, to name a few).

In this paper, we analyse the stochastic version of the problem \eqref{el1}-\eqref{el1eq4}. For that, we instigate L\'evy noise in the equation \eqref{el1}. Moreover, we set $ \mu = \lambda = \gamma = 1,$ as well as, we supersede the Ginzburg-Landau bounded function $\chi_{|\mathbf{d}| \leq 1}( |\mathbf{d}|^2 -1 ) \mathbf{d}$ by a general polynomial function $f(\mathbf{d}).$ We consider $\tilde{\eta}$ to be the compensated time homogeneous Poisson random measure on a certain measurable space $(Y, \mathscr{B}(Y))$ and $F$ to be a measurable function. This system of stochastic evolution equations is given by
 \begin{align}\label{1st}
& d\mathbf{u}(t) + \big[ (\mathbf{u}(t) \cdot \nabla)\mathbf{u}(t) - \mu \Delta \mathbf u(t) + \nabla p \big] dt  \nonumber
\\ &\quad\quad \ = - \lambda \nabla \cdot \big( \nabla \mathbf{d}(t) \odot \nabla \mathbf{d}(t) \big) dt + \int_{Y} F(t, \mathbf{u}(t); y) \,\tilde{\eta}(dt, dy),
\\ & \qquad \qquad \qquad \qquad \qquad \qquad \nabla \cdot \mathbf{u}(t) = 0,
\\ &\qquad d\mathbf{d}(t) + \big[(\mathbf{u}(t) \cdot \nabla)\mathbf{d}(t)\big] dt = \gamma \left( \Delta \mathbf{d}(t) - \frac{1}{\epsilon^{2}} f(\mathbf{d}(t)) \right)dt. \label{2nd}
\end{align}
This holds in $\mathbb{O}_{T} := (0, T] \times \mathbb{O},$ with the initial and boundary conditions as \eqref{el1eq4} and \eqref{inboc} respectively. In the coming sections we will provide more information about the polynomial $f$ and the measurable function $F$. 

Most of the physical systems confront dynamical instabilities. The instability befalls at some critical value of the control parameter (which is in our case some random external noise) of the system. In our predicament the dynamics are quite intricate because the evolution of the director field $\mathbf{d}(x, t)$ is coupled to the velocity field $\mathbf{u}(x, t).$ San Miguel \cite{SM}, has studied the stationary orientational correlations of the director field of a nematic liquid crystal near the Fr\'eedericksz transition. In this transition the molecules tend to reorient due to some random external perturbations. It has been studied by Sagu\'es and Miguel \cite{FS} that the decay time, required for the system is shortened by the field fluctuations to leave an unstable state, whcih is built by switching on the field to a value beyond instability point. See also Horsthemke and Lefever\cite{WR} and references there in, for more details. A nematic drifts very much like a typical organic liquid with molecules of indistinguishable size. Since, the transitional motions are coupled to inner, orientational motions of the molecules, in most cases the flow muddles the alignment. Conversely, by implimentation of an external field, a change in the alignment will generate a flow in the nematic. So we are interested in the study of \eqref{1st}-\eqref{2nd}, which characterize the flow of nematic liquid crystals, effected by altering external forces.

There are few notable works available on the stochastic version of the problem \eqref{el1}-\eqref{el1eq4}. The authors in the unpublished manuscript \cite{BHRa} have studied the Ginzburg-Landau approximation of the above system governing the nematic liquid crystals under the influence of fluctuating external forces. In their paper they have proved the existence and uniqueness of local maximal solution for both 2-D and 3-D case using fixed point argument. Also they have proved the existence of global strong solution to the problem in 2-D. Later, the same authors in another unpublished manuscript \cite{BHR} have considered the same model as in \cite{BHRa} with multiplicaive Gaussian noise and replaced the Ginzburg-Landau function by a general polynomial, under suitable assumptions on it. In their paper they have proved the existence of global weak solution and showed pathwise uniqueness of the solution in 2-D. Also they have established the existence and uniqueness of local maximal strong solution and showed this solution is global in 2-D.  

In our paper we consider the same model which describes the dynamics of the nematic liquid crystal, but we have replaced the multiplicative Gaussian noise with L\'evy noise (Time homogeneous Poisson random measure). It is worthwhile to note that there are not many literatures available on this particular stochastic version of the problem. Though in our coupled system we have the L\'evy noise in the first equation only, it still makes a quite difficult and interesting problem as it invloves tedious calculations. We are interested in showing the existence (both in 2-D and 3-D) and pathwise uniqueness (only in 2-D) of the weak martingale solution of our problem. Some of the arguments used in Sections \ref{tolas} and \ref{eoms} of our paper are closely related to the paper \cite{EM}, where the author has proved the existence of martingale solutions of the stochastic Navier-Stokes equations driven by L\'evy noise.

This paper is organized as follows. In Section \ref{fsm} we define variuos functional spaces and some useful operators which are used throughout in our paper. Also we have stated various inequalities and listed all the assumptions at the end of this particular Section. In Section \ref{somr} we define the martingale solution of our model in the view of operators defined in Section \ref{fsm}. In Section \ref{catc} we state compactness results and tightness criterion for both $\mathbf{u}$ and $\mathbf{d}$. In Section \ref{EE} we derive several important a priori energy estimates  of the approximating sequences $(\mathbf{u}_n, \mathbf{d}_n),$ obtained by the Faedo-Galerkin method. Then the Sections \ref{tolas} and \ref{eoms} are devoted to the proof of tightness of the above approximating solutions and the existence of martingale solution respectively. For the existence of such solution we use Skorokhod embedding theorems. As a consequence of this theorem, finally in the end of the Section \ref{eoms} we show the convergence of the new processes to the corresponding limiting processes. Then the proof of the main result defined in Section \ref{somr}. In Section \ref{pu2} we show the pathhwise uniqueness of the weak solution, but only in two dimensions. Finally in Appendix we prove various results and estimates to use them in the derivation of a priori estimates as well as in the proof of existence of weak martingale solution. 

\section{Functional Setting of the Model}\label{fsm}

\subsection{Basic Definitions and Functional Spaces}\label{bd}
Let $\mathbb{O} \subset \mathbb{R}^\mathbf{n}$ be a bounded domain with smooth boundary $\partial \mathbb{O}, \mathbf{n} = 2, 3$. For any $p \in [1, \infty)$ and $k \in \mathbb{N},$ $L^p(\mathbb{O})$ and $W^{k, p}(\mathbb{O})$ are well-known Lebesgue and Sobolev spaces of $\mathbb{R}^{\mathbf{n}}$-valued functions respectively. For $p=2,$ put $W^{k, 2}= H^k.$ For instance, $H^1(\mathbb{O}; \mathbb{R}^{\mathbf{n}})$ is the Sobolev space of all $\mathbf{u} \in L^2(\mathbb{O}; \mathbb{R}^{\mathbf{n}}),$ for which there exist weak derivatives $\frac{\partial  \mathbf{u}}{\partial x_i} \in L^2(\mathbb{O}; \mathbb{R}^{\mathbf{n}}), i=1, 2, \cdots, \mathbf{n}.$ It is a Hilbert space with the scalar product given by
\[(\mathbf{u}, \mathbf{v})_{H^1} := (\mathbf{u}, \mathbf{v})_{L^2} + (\nabla \mathbf{u}, \nabla \mathbf{v})_{L^2}, \quad \mathbf{u}, \mathbf{v} \in H^1(\mathbb{O}, \mathbb{R}^{\mathbf{n}}).\]

 Let us define the following spaces
\begin{align*}
\mathcal{V} := \{ \mathbf{u} \in \mathcal{C}_{c}^{\infty}(\mathbb{O}; \mathbb{R}^\mathbf{n}) : div \ \mathbf{u} = 0\},\\
\mathbb{H} := \mathrm{the \ closure \ of \ }\mathcal{V} \ \mathrm{in} \ L^2(\mathbb{O}; \mathbb{R}^\mathbf{n}),\\
\mathbb{V} := \mathrm{the \ closure \ of \ }\mathcal{V} \ \mathrm{in} \ H^{1}_{0}(\mathbb{O}; \mathbb{R}^\mathbf{n}).
\end{align*}

In the space $\mathbb{H}$ we consider the scalar product and the norm inherited from $L^2(\mathbb{O}; \mathbb{R}^\mathbf{n})$ and denote them by $(\cdot, \cdot)_{\mathbb{H}}$ and $|\cdot|_{\mathbb{H}},$ respectively, i.e.,
\begin{align}
(\mathbf{u}, \mathbf{v})_{\mathbb{H}} := (\mathbf{u}, \mathbf{v})_{L^2}, \ \ \ \ \ |\mathbf{u}|_{\mathbb{H}} := |\mathbf{u}|_{L^2} := |\mathbf{u}|, \ \ \ \ \ \ \mathbf{u}, \mathbf{v} \in \mathbb{H}.
\end{align}

In the space $\mathbb{V}$ we consider the scalar product inherited from the Sobolev space $H^1(\mathbb{O}; \mathbb{R}^\mathbf{n})$ i.e.,
\begin{align}
(\mathbf{u}, \mathbf{v})_{\mathbb{V}} := (\mathbf{u}, \mathbf{v})_{L^2} + ((\mathbf{u}, \mathbf{v})),
\end{align}
where 
\begin{align}\label{Gd}
((\mathbf{u}, \mathbf{v})) := (\nabla \mathbf{u}, \nabla \mathbf{v})_{L^2} = \sum_{i=1}^{\mathbf{n}} \int_{\mathbb{O}} \frac{\partial  \mathbf{u}}{\partial x_i} \cdot \frac{\partial  \mathbf{v}}{\partial x_i} \,dx, \ \ \ \ \ \mathbf{u}, \mathbf{v} \in \mathbb{V}.
\end{align}

and the norm
\begin{align}
\|\mathbf{u}\|^{2}_{\mathbb{V}} := |\mathbf{u}|^{2}_{\mathbb{H}} + \|\mathbf{u}\|^{2},
\end{align}
where $\|\mathbf{u}\|^2 := \| \nabla \mathbf{u}\|_{L^2}^{2}.$

\subsection{Bilinear Operators}

Let us consider the following trilinear form, see Temam\cite{Te}
\[b(\mathbf{u}, \mathbf{v}, \mathbf{w}) = \sum_{i, j = 1}^{n} \int_{\mathbb{O}} \mathbf{u}^{(i)} \partial_{x_{i}}\mathbf{v}^{(j)} \mathbf{w}^{j} \,dx, \quad \mathbf{u} \in L^p, \mathbf{v} \in W^{1, q}, \mathbf{w} \in L^r,\]
where $p, q, r \in [1, \infty]$ satisfying
\begin{align}\label{pqr1}
\frac{1}{p} + \frac{1}{q} + \frac{1}{r} \leq 1.
\end{align}

We will recall the fundamental properties of the form $b$ that are valid for both bounded and unbounded domains. By the Sobolev embedding Theorem, see Adams\cite{Ad}, and H\"older's inequality, we obtain
\begin{align*}
|b(\mathbf{u}, \mathbf{w}, \mathbf{v})| \leq c \|\mathbf{u}\|_{\mathbb{V}}\|\mathbf{w}\|_{\mathbb{V}}\|\mathbf{v}\|_{\mathbb{V}}, \qquad \mathbf{u}, \mathbf{w}, \mathbf{v} \in \mathbb{V}
\end{align*}
for some positive constant c. Thus the form $b$ is continuous on $\mathbb{V}.$ Now we define a bilinear map $B$ by $B(\mathbf{u}, \mathbf{w}) := b(\mathbf{u}, \mathbf{w}, \cdot),$ then we infer that $B(\mathbf{u}, \mathbf{w}) \in \mathbb{V}'$ for all $\mathbf{u}, \mathbf{w} \in \mathbb{V}$ and the following inequality holds
\begin{align}\label{buw}
\|B(\mathbf{u}, \mathbf{w})\|_{\mathbb{V}'} \leq c \ \|\mathbf{u}\|_{\mathbb{V}}\|\mathbf{w}\|_{\mathbb{V}}, \qquad \mathbf{u}, \mathbf{w} \in \mathbb{V}.
\end{align}

Moreover, the mapping $B : \mathbb{V} \times \mathbb{V} \to \mathbb{V}'$ is bilinear and continuous. The form $b$ also has the following properties, see \cite{Te},
\[ b(\mathbf{u}, \mathbf{w}, \mathbf{v}) = -b(\mathbf{u}, \mathbf{v}, \mathbf{w}), \qquad \mathbf{u}, \mathbf{w}, \mathbf{v} \in \mathbb{V}.\] 

In particular,
\[ b(\mathbf{u}, \mathbf{v}, \mathbf{v}) = 0, \qquad  \mathbf{u}, \mathbf{v} \in \mathbb{V}.\]

Hence
\[\langle B(\mathbf{u}, \mathbf{w}), \mathbf{v}\rangle = -\langle B(\mathbf{u}, \mathbf{v}), \mathbf{w}\rangle, \quad \mathbf{u}, \mathbf{w}, \mathbf{v} \in \mathbb{V}\]
and
\begin{align}
\langle B(\mathbf{u}, \mathbf{v}), \mathbf{v}\rangle = 0, \qquad \mathbf{u}, \mathbf{v} \in \mathbb{V}.
\end{align}

Moreover, for all $\mathbf{u} \in \mathbb{V}, \mathbf{v} \in H^1,$ we have
\begin{align}\label{buvvd}
\|B(\mathbf{u}, \mathbf{v})\|_{\mathbb{V}'} \leq c \ |\mathbf{u}|^{1-\frac{\mathbf{n}}{4}} \| \mathbf{u}\|^{\frac{\mathbf{n}}{4}} |\mathbf{v}|^{1-\frac{\mathbf{n}}{4}} \| \mathbf{v}\|^{\frac{\mathbf{n}}{4}}, \quad \mathbf{n} \in \{ 2, 3\}
\end{align}

\noindent For the proof, we refer to Section 1.2 of Temam\cite{Te}. The operator $B$ can be uniquely extended to a bounded linear operator
\[B : \mathbb{H} \times \mathbb{H} \to \mathbb{V}'.\]
And it satisfies the follwing estimate
\begin{align}\label{BHH}
\|B(\mathbf{u}, \mathbf{v})\|_{\mathbb{V}'} \leq c \ |\mathbf{u}|_{\mathbb{H}} |\mathbf{v}|_{\mathbb{H}}.
\end{align}

We will use the following notation, $B(\mathbf{u}) := B(\mathbf{u}, \mathbf{u}).$ Also note that the map $B : \mathbb{V} \to \mathbb{V}'$ is locally Lipschitz continuous.
\\ One can define a bilinear map $\tilde{B}$ defined on $H^1 \times H^1$ with values in $(H^1)'$ such that 
\begin{align*}
\langle \tilde{B}(\mathbf{u}, \mathbf{v}), \mathbf{w} \rangle = b(\mathbf{u}, \mathbf{v}, \mathbf{w}) \qquad \mathbf{u}, \mathbf{v}, \mathbf{w} \in H^1
\end{align*}

With an abuse of notation, we again denote by $\tilde{B}(\cdot, \cdot)$ the restriction of $\tilde{B}(\cdot, \cdot)$ to $\mathbb{V} \times H^2,$ which maps continuously $\mathbb{V} \times H^2$ into $L^2$. Using the Gagliardo-Nirenberg inequalities one can show there exists a positive constant $C$ such that for $\mathbf{n} \in \{ 2, 3\},$
\begin{align}\label{btil}
\big| \tilde{B}(\mathbf{u}, \mathbf{d}) \big| \leq C \ |\mathbf{u}|^{1-\frac{\mathbf{n}}{4}} \|\mathbf{u}\|^{\frac{\mathbf{n}}{4}} \|\mathbf{d}\|^{1-\frac{\mathbf{n}}{4}} |\Delta \mathbf{d}|^{\frac{\mathbf{n}}{4}}, \qquad \mathbf{u} \in \mathbb{V}, \mathbf{d} \in H^2.
\end{align}
Moreover, using Young's inequality one can get
\begin{align}
\big| \tilde{B}(\mathbf{u}, \mathbf{d}) \big| \leq C \ \|\mathbf{u}\|_{\mathbb{V}} \|\mathbf{d}\|_{H^2}.
\end{align}
We also have
\begin{align}
\langle \tilde{B}(\mathbf{u}, \mathbf{d}), \mathbf{d} \rangle = 0, \qquad \mathbf{u} \in \mathbb{V}, \mathbf{d} \in H^2.
\end{align}
 For the proof, we refer to Section 1.2 of Temam\cite{Te}.

Let $m$ be the trilinear form defined by
\[m(\mathbf{d}_1, \mathbf{d}_2, \mathbf{u}) = - \sum_{i,j,k=1}^{\mathbf{n}} \int_{\mathbb{O}} \partial_{x_i}\mathbf{d}_{1}^{(k)} \partial_{x_j}\mathbf{d}_{2}^{(k)} \partial_{x_j}\mathbf{u}^{(i)} \,dx\]
for any $\mathbf{d}_{1} \in W^{1, p}, \mathbf{d}_{2} \in W^{1, q}$ and $\mathbf{u} \in W^{1, r}$ with $p, q, r \in  (1, \infty)$ satisfying condition (\ref{pqr1}). Since $\mathbf{n} \in \{2, 3\},$ the above integral is well defined, when $\mathbf{d}_{1}, \mathbf{d}_{2} \in H^2$ and $\mathbf{u} \in \mathbb{V}.$ We have the following Lemma.
\begin{lemma}
There exists a constant $C > 0$ such that
\begin{align*}
|m(\mathbf{d}_1, \mathbf{d}_2, \mathbf{u})| \leq C \ \|\mathbf{d}_1\|^{1-\frac{\mathbf{n}}{4}} |\Delta \mathbf{d}_1|^{\frac{\mathbf{n}}{4}} \|\mathbf{d}_2\|^{1-\frac{\mathbf{n}}{4}} |\Delta \mathbf{d}_2|^{\frac{\mathbf{n}}{4}} \|\mathbf{u}\|
\end{align*}
for any $\mathbf{d}_{1}, \mathbf{d}_{2} \in H^2$ and $\mathbf{u} \in \mathbb{V}.$
\end{lemma}

\noindent For proof see \cite{BHR}. Now we state the following Lemma.

\begin{lemma}
There exists a bilinear operator $M$ defined on $H^2 \times H^2$ taking values on $\mathbb{V}'$ such that for any $\mathbf{d}_{1}, \mathbf{d}_{2} \in H^2$
\[\big\langle M(\mathbf{d}_{1}, \mathbf{d}_{2}), \mathbf{u} \big\rangle = m(\mathbf{d}_{1}, \mathbf{d}_{2}, \mathbf{u}), \quad \mathbf{u} \in \mathbb{V}.\]
Furthermore, we have
\begin{align}\label{md1d2}
\|M(\mathbf{d}_{1}, \mathbf{d}_{2})\|_{\mathbb{V}'} \leq C \ \|\mathbf{d}_1\|^{1-\frac{\mathbf{n}}{4}} |\Delta \mathbf{d}_1|^{\frac{\mathbf{n}}{4}} \|\mathbf{d}_2\|^{1-\frac{\mathbf{n}}{4}} |\Delta \mathbf{d}_2|^{\frac{\mathbf{n}}{4}},
\end{align}
for any $\mathbf{d}_{1}, \mathbf{d}_{2} \in H^2.$
\end{lemma}

\noindent For proof we refer \cite{BHR}. We will use the following notation, $M(\mathbf{d}) := M(\mathbf{d}, \mathbf{d}).$

\subsection{Linear Operators, Its Properties and Important Embeddings}\label{loU}

Now we will recall operators and their properties used in \cite{BrM}. Consider the natural embedding $j : \mathbb{V} \hookrightarrow \mathbb{H}$ and its adjoint $j' : \mathbb{H} \hookrightarrow \mathbb{V}.$ Since the range of $j$ is dense in $\mathbb{H},$ the map $j'$ is one-to-one. Let us put 
\begin{align}\label{scra}
\mathscr{A}\mathbf{u} := (( \mathbf{u}, \cdot)), \qquad \mathbf{u} \in \mathbb{V},
\end{align}
where $((\cdot, \cdot))$ is defined in \eqref{Gd}. If $\mathbf{u} \in \mathbb{V},$ then $\mathscr{A}\mathbf{u} \in \mathbb{V}'$. Since we have the following inequalities
\begin{align*}
|(( \mathbf{u}, \mathbf{v}))| \leq \|\mathbf{u}\| \cdot \|\mathbf{v}\| \leq \|\mathbf{u}\| (\|\mathbf{v}\|^2 + |\mathbf{v}|_{\mathbb{H}}^2)^{\frac{1}{2}} = \|\mathbf{u}\| \cdot \|\mathbf{v}\|_{\mathbb{V}}, \quad \mathbf{v} \in \mathbb{V}. 
\end{align*}
we infer that
\begin{align}
|\mathscr{A}\mathbf{u}|_{\mathbb{V}'} \leq \|\mathbf{u}\|, \qquad \mathbf{u} \in \mathbb{V}.
\end{align}

From \cite{Ouh}, we can define a self-adjoint operator $\mathcal{A} : H^1 \to \big( H^1 \big)'$ by
\begin{align}\label{cala}
\big\langle \mathcal{A} \mathbf{d}, \mathbf{w}\big\rangle := ((\mathbf{d}, \mathbf{w})) := \int_{\mathbb{O}} \nabla\mathbf{d} \cdot \nabla \mathbf{w} \,dx, \quad \mathbf{d}, \mathbf{w} \in H^1.
\end{align}
The Neumann Laplacian acting on $\mathbb{R}^{\mathbf{n}}$-valued function will be denoted by $\mathcal{A},$ i.e.,
\begin{align*}
D(&\mathcal{A}) := \bigg\{ \mathbf{d} \in H^2 : \frac{\partial \mathbf{d}}{\partial \textsl{n}} = 0 \ on \ \partial\mathbb{O} \bigg\}, 
\\ &\mathcal{A}\mathbf{d} := -\sum_{i = 1}^{\mathbf{n}} \frac{\partial^{2} \mathbf{d}}{\partial x_{i}^{2}}, \quad \mathbf{d} \in D(\mathcal{A}).
\end{align*}
It can be shown that $D(\mathcal{A}) = \{ \mathbf{d} \in H^1 : \mathcal{A} \mathbf{d} \in L^2 \}.$ We also have
\begin{align}\label{Ah1}
\|\mathcal{A}\mathbf{d}\|_{(H^1)'} \leq \|\mathbf{d}\|_{H^1}.
\end{align}

As we are working on a bounded domain, it is clear that $\mathbb{V}$ is dense in $\mathbb{H}$ and the embedding is continuous. We have
\begin{align}\label{emb}
\mathbb{V} \underset{j_1}{\hookrightarrow} \mathbb{H} \cong \mathbb{H}' \underset{j'_1}{\hookrightarrow} \mathbb{V}' .
\end{align}

Also we have the embeddings
\begin{align}\label{emb2}
H^2 \underset{j_2}{\hookrightarrow} H^1 \hookrightarrow L^2 \cong L^2 \hookrightarrow (H^1)' \underset{j'_2}{\hookrightarrow} (H^2)'.
\end{align}

\subsection{Assumptions}\label{assu}

\begin{enumerate}[label=(\Alph*)]
\item
$\tilde{\eta}$ is a compensated time homogeneous Poisson random measure on a measurable space $(Y, \mathscr{B}(Y))$ over $(\Omega, \mathcal{F}, \mathbb{F}, \mathbb{P})$ with a $\sigma$-finite intensity measure $\nu.$ See Appendix for definitions and more details.
\item
Let $F : [0, T] \times \mathbb{H} \times Y \to \mathbb{H}$ is a measurable function and there exists a constant $L$ such that 
\begin{align}\label{lipf}
\int_{Y} \big| F(t, \mathbf{u}_1; y) - F(t, \mathbf{u}_2; y) \big|^{2}_{\mathbb{H}} \,\nu(dy) \leq L \big|\mathbf{u}_1 - \mathbf{u}_2 \big|^{2}_{\mathbb{H}}, \ \ \mathbf{u}_1, \mathbf{u}_2 \in \mathbb{H}, t \in [0, T].
\end{align}
and for each $p \geq 1$ there exists a constant $C_p$ such that
\begin{align}\label{lgf}
\int_{Y} \big| F(t, \mathbf{u}; y) \big|^{p}_{\mathbb{H}} \,\nu(dy) \leq C_p \big( 1 + |\mathbf{u}|^{p}_{\mathbb{H}}\big), \quad \mathbf{u} \in \mathbb{H}, \quad t \in [0, T], 
\end{align}

\item
Let $\mathbb{I}_\mathbf{n}$ be the set defined by 
\[ \mathbb{I}_\mathbf{n} = \begin{cases}
                                 \mathbb{N} := \{1, 2, 3, \cdots\} &\quad \text{if } \mathbf{n} = 2,\\
                                 \{1\}, &\quad \text{if } \mathbf{n} = 3.
                               \end{cases}
\]
For $N \in \mathbb{I}_\mathbf{n}$ and numbers $b_j, j = 0, \cdots, N,$ with $b_N > 0$ we define a function $\tilde{f} \colon [0, \infty) \to \mathbb{R}$ by
\begin{align*}
\tilde{f}(r) = \sum_{j = 0}^{N} b_j r^j, \text{for any } r \in \mathbb{R}_{+}.
\end{align*}
We define a map ${f} \colon \mathbb{R}^\mathbf{n} \to \mathbb{R}^\mathbf{n}$ by 
\begin{align}
f(\mathbf{d}) = \tilde{f}(|\mathbf{d}|^2)\mathbf{d}.
\end{align}
Let $\mathrm{F} : \mathbb{R}^\mathbf{n} \to \mathbb{R}$ be a Fr\'ech\'et differentiable map such that for any $\mathbf{d} \in \mathbb{R}^\mathbf{n}$ and $\mathbf{g} \in \mathbb{R}^\mathbf{n}$
\begin{align*}
\mathrm{F}'(\mathbf{d})[\mathbf{g}] = f(\mathbf{d}) \cdot \mathbf{g}. 
\end{align*}
Let also $\tilde{F}$ be an antiderivative of $\tilde{f}$ such that $\tilde{F}(0) = 0.$ We have 
\[\tilde{F}(r) = a_{N+1} r^{N+1} + \mathcal{U}(r),\]
where $\mathcal{U}$ is a polynomial function of at most degree $N$ and $a_{N+1} > 0.$

\begin{remark}\label{fes}
We have the following results
\begin{align*}
|\tilde{f}(r)| \leq l_1(1 + r^N) \ \text{ and } \ |\tilde{f}'(r)| \leq l_2(1 + r^{N_1}), \quad r > 0.  
\end{align*}
for some $l_1, l_2 > 0$. And there exist positive constants $c, \tilde{c}$ such that
\begin{align*}
|f(\mathbf{d})|_{\mathbb{R}^{\mathbf{n}}} \leq c \ \big(1 + \big|\mathbf{d} \big|_{\mathbb{R}^{\mathbf{n}}}^{2N+1} \big) \ \text{ and } \ |f'(\mathbf{d})|_{\mathbb{R}^{\mathbf{n}}} \leq \tilde{c}\ \big(1 + \big|\mathbf{d} \big|_{\mathbb{R}^{\mathbf{n}}}^{2N} \big), \quad \mathbf{d} \in \mathbb{R}^\mathbf{n}.  
\end{align*}
\end{remark}

\end{enumerate}

We rewrite the equations \eqref{1st}-\eqref{2nd}, for $\epsilon = 1,$ as

\begin{align}\label{r1st}
 d\mathbf{u}(t) + \big[ \mathscr{A}\mathbf{u}(t) &+ B(\mathbf{u}(t)) + M(\mathbf{d}(t)) \big] \,dt  =  \int_{Y} F(t, \mathbf{u}(t); y) \,\tilde{\eta}(dt, dy),
\\& d\mathbf{d}(t) = - \big[ \mathcal{A} \mathbf{d}(t) + \tilde{B}(\mathbf{u}(t), \mathbf{d}(t)) + f(\mathbf{d}(t)) \big] dt. \label{r2nd}
\end{align}

\section{Statement of the Main Result}\label{somr}

Let us recall the definition of a martingale solution.

\begin{definition}\label{mdefi}
A {\bf martingale solution} of the problem\eqref{r1st}-\eqref{r2nd} is a system \\$\big( \bar{\Omega}, \bar{\mathcal{F}}, \bar{\mathbb{F}}, \bar{\mathbb{P}}, \bar{\mathbf{u}}, \bar{\mathbf{d}}, \bar{\eta}\big),$ where
\begin{enumerate}
\item
$\big( \bar{\Omega}, \bar{\mathcal{F}}, \bar{\mathbb{F}}, \bar{\mathbb{P}}\big)$ is a filtered probability space with a filtration $\bar{\mathbb{F}}= \big\{ \bar{\mathcal{F}}_t\big\}_{t \geq 0},$
\item
$\bar{\eta}$ is a time homogeneous Poisson random measure on $(Y, \mathscr{B}(Y))$ over \\$\big( \bar{\Omega}, \bar{\mathcal{F}}, \bar{\mathbb{F}}, \bar{\mathbb{P}}\big)$ with the intensity measure $\nu,$
\item
$\bar{\mathbf{u}} : [0, T] \times \bar{\Omega} \to \mathbb{V}$ is a progressively measurable process with $\bar{\mathbb{P}}$-a.e. paths
\begin{align}
\bar{\mathbf{u}}(\cdot, \omega) \in \mathbb{D}([0, T]; \mathbb{H}_w) \cap L^2(0, T; \mathbb{V})
\end{align}
such that for all $t \in [0, T]$ and all $v \in \mathbb{V}$ the following identity holds  $\bar{\mathbb{P}}$-a.s.
\begin{align}
&(\bar{\mathbf{u}}(t), v )_{\mathbb{H}} + \int_{0}^{t} \big\langle \mathscr{A}\bar{\mathbf{u}}(s), v \big\rangle \,ds +  \int_{0}^{t} \big\langle B(\bar{\mathbf{u}}(s)), v \big\rangle \,ds \nonumber
\\ & \ \ +  \int_{0}^{t} \big\langle M(\bar{\mathbf{d}}(s)), v \big\rangle \,ds = (\mathbf{u}_{0}, v )_{\mathbb{H}} + \int_{0}^{t} \int_{Y} \big( F(s, \bar{\mathbf{u}}(s); y), v\big)_{\mathbb{H}} \tilde{\bar{\eta}}(ds, dy).
\end{align}
\item
$\bar{\mathbf{d}} : [0, T] \times \bar{\Omega} \to H^2$ is a progressively measurable process with $\bar{\mathbb{P}}$-a.e. paths
\begin{align}
\bar{\mathbf{d}}(\cdot, \omega) \in \mathbb{C}([0, T]; H^{1}_{w}) \cap L^2(0, T; H^2)
\end{align}
such that for all $t \in [0, T]$ and all $v \in H^2$ the following identity holds  $\bar{\mathbb{P}}$-a.s.
\begin{align}
&(\bar{\mathbf{d}}(t), v ) + \int_{0}^{t} \big\langle \mathcal{A}\bar{\mathbf{d}}(s), v \big\rangle \,ds +  \int_{0}^{t} \big\langle \tilde{B}(\bar{\mathbf{u}}(s), \bar{\mathbf{d}}(s)), v \big\rangle \,ds \nonumber
\\ &= (\mathbf{d}_{0}, v ) - \int_{0}^{t} \big\langle f(\bar{\mathbf{d}}(s)), v \big\rangle \,ds.
\end{align}
\end{enumerate}  
\end{definition}

\begin{definition}
The problem \eqref{r1st}-\eqref{r2nd} has a {\bf strong solution} if and only if for every stochastic basis $\big( \Omega, \mathcal{F}, \mathbb{F}, \mathbb{P}\big)$ with a filtration $\mathbb{F} = \big\{ \mathcal{F}_t\big\}_{t \geq 0}$ and a time homogeneous Poisson random measure $\eta$ on $(Y, \mathscr{B}(Y))$ over the stochastic basis with intensity measure $\nu$, there exist progressively measurable process $\mathbf{u} : [0, T] \times \bar{\Omega} \to \mathbb{V}$ with $\mathbb{P}$-a.e. paths
\begin{align}
\mathbf{u}(\cdot, \omega) \in \mathbb{D}([0, T]; \mathbb{H}) \cap L^2(0, T; \mathbb{V})
\end{align}
and progressively measurable process $\mathbf{d} : [0, T] \times \bar{\Omega} \to H^2$ with $\mathbb{P}$-a.e. paths
\begin{align}
\mathbf{d}(\cdot, \omega) \in \mathbb{C}([0, T]; H^{1}) \cap L^2(0, T; H^2)
\end{align}
such that for all $t \in [0, T]$ and $v \in \mathbb{V}$ the following identity holds $\mathbb{P}$-a.s.
\begin{align}
&(\mathbf{u}(t), v )_{\mathbb{H}} + \int_{0}^{t} \big\langle \mathscr{A}\mathbf{u}(s), v \big\rangle \,ds +  \int_{0}^{t} \big\langle B(\mathbf{u}(s)), v \big\rangle \,ds \nonumber
\\ & \ \ +  \int_{0}^{t} \big\langle M(\mathbf{d}(s)), v \big\rangle \,ds = (\mathbf{u}_{0}, v )_{\mathbb{H}} + \int_{0}^{t} \int_{Y} \big( F(s, \mathbf{u}(s); y), v\big)_{\mathbb{H}} \tilde{\eta}(ds, dy).
\end{align}
and for all $v \in H^2$ the following identity holds $\mathbb{P}$-a.s.
\begin{align}
(\mathbf{d}(t), v ) &+ \int_{0}^{t} \big\langle \mathcal{A} \mathbf{d}(s), v \big\rangle \,ds +  \int_{0}^{t} \big\langle \tilde{B}(\mathbf{u}(s), \mathbf{d}(s)), v \big\rangle \,ds \nonumber
\\ &= (\mathbf{d}_{0}, v ) - \int_{0}^{t} \big\langle f(\mathbf{d}(s)), v \big\rangle \,ds.
\end{align}
\end{definition}

\noindent The main result we are going to prove is as follows 
\begin{theorem}
Let the Assumptions in the Section \ref{assu} hold. Let $\mathbf{n} =2, 3$ and $(\mathbf{u}_0, \mathbf{d}_0) \in \mathbb{H} \times H^1.$ Then the system \eqref{r1st}-\eqref{r2nd} has a weak martingale solution in the sense of Definition \ref{mdefi}. Also the solution satisfies the following inequalities
\begin{align}\label{usv2-}
\mathbb{E} \bigg[ \sup_{t \in [0, T]} \big| \mathbf{u}(t) \big|^2_{\mathbb{H}} + \int_{0}^{T} \big\| \mathbf{u}(t) \big\|^{2}_{\mathbb{V}} \,ds \bigg] < \infty,
\end{align}
and
\begin{align}\label{dsp-}
\mathbb{E} \bigg[ \sup_{t \in [0, T]} \big\| \mathbf{d}(t) \big\|_{H^1}^{2} + \int_{0}^{T} \big\| \mathbf{d}(t) \big\|^{2}_{H^2} \bigg] < \infty.
\end{align}
Moreover, the pathwise uniqueness and the existence of a strong solution holds in the two dimensions.
\end{theorem}

\section{Compactness and Tightness Criterion}\label{catc}

\subsection{Compactness Results}

Let $(\mathbb{M}, \rho)$ be a complete separable metric space. Let $\mathbb{D}([0, T]; \mathbb{M})$ be the space of all $\mathbb{M}$-valued c\`adl\`ag functions defined on $[0, T],$ i.e. the functions which are right continuous and have left limits at every $t \in [0, T].$ This space is endwoed with the Skorokhod topology.

A sequence $(\mathbf{u}_n) \subset \mathbb{D}([0, T]; \mathbb{M})$ converges to $\mathbf{u} \in \mathbb{D}([0, T]; \mathbb{M})$ iff there exists a sequence $(\mu_n)$ of homeomorphisms of $[0, T]$ such that $\mu_n$ tends to the identity uniformly on $[0, T]$ and $\mathbf{u}_n \circ \mu_n$ tends to $\mathbf{u}$ uniformly on $[0, T].$

The topology is metrizable by the following metric $\vartheta_{T}$
\[\vartheta_{T}(\mathbf{u}, \mathbf{v}) := \inf_{\mu \in \sigma_T} \bigg[ \sup_{t \in [0, T]} \rho\big( \mathbf{u}(t), \mathbf{v} \circ \mu(t)\big) + \sup_{t \in [0, T]} |t - \mu(t)| + \sup_{s \neq t} \bigg| \log \frac{\mu(t) - \mu(s)}{t-s}\bigg|\bigg],\]

where $\sigma_{T}$ is the set of increasing homeomorphisms of $[0, T].$ \\Moreover, $\big( \mathbb{D}([0, T]; \mathbb{M}), \vartheta_{T}\big)$ is a complete metric space.

\begin{definition}
Let $\mathbf{u} \in \mathbb{D}([0, T]; \mathbb{M})$ and let $\delta > 0$ be given. A \textbf{modulus} of $\mathbf{u}$ is defined by 
\begin{align}\label{matw}
\mathcal{W}_{[0, T], \mathbb{M}}(\mathbf{u}; \delta) := \inf_{\Pi_{\delta}} \max_{t_i \in \tilde{\omega}} \sup_{t_i \leq s < t < t_{i+1}\leq T} \rho\big( \mathbf{u}(t), \mathbf{u}(s)\big),
\end{align}
where $\Pi_{\delta}$ is the set of all increasing sequences $\tilde{\omega}= \{0=t_0 < t_1 < \cdots < t_n =T\}$ with the following property
\[t_{i+1}-t_i \geq \delta, \quad i=0, 1, \cdots, n-1.\]
\end{definition}
 
Analogous to the Arzel\`a-Ascoli Theorem for the space of continuous functions, we have the following criterion for the relative compactness of a subset of the space $\mathbb{D}([0, T]; \mathbb{M}).$

\begin{theorem}
A set $B \subset \mathbb{D}([0, T]; \mathbb{M})$ has precompact iff it satisfies the following two conditions:
\begin{enumerate}
\item
there exists a dense subset $J \subset [0, T]$ such that for every $t \in J$ the set $\{ \mathbf{u}(t), \mathbf{u} \in B\}$ has compact closure in $\mathbb{M}.$
\item
$\lim_{\delta \to 0} \sup_{\mathbf{u} \in B} \mathcal{W}_{[0, T]}(\mathbf{u}; \delta) = 0$.
\end{enumerate}
\end{theorem}

\begin{proof}
 For details see \cite{MMe}.
\end{proof}

\noindent Let us consider the following functional spaces.
\begin{align*}
&\mathbb{D}([0, T]; \mathbb{V}') \ \ := \ \text{the space of c\`adl\`ag functions} \ \mathbf{u} : [0, T] \to \mathbb{V}' \ \text{with the topology}
\\ &\qquad \qquad \qquad \quad \ \ \mathcal{T}_1 \ \text{induced by the Skorokhod metric} \ \delta_{T}, 
\\ &\mathbb{C}([0, T]; (H^2)') \ \ := \ \text{the space of continuous functions} \ \mathbf{d} : [0, T] \to (H^2)' \ \text{with the topology} \  \mathcal{T}'_1,
\\ &L^{2}_{w}(0, T; \mathbb{V}) \ \ := \ \text{the space} \ L^2(0, T; \mathbb{V}) \ \text{with the weak topology} \ \mathcal{T}_2,
\\ &L^{2}_{w}(0, T; H^2) \ \ := \ \text{the space} \ L^2(0, T; H^2) \ \text{with the weak topology} \ \mathcal{T}'_2,
\\ &L^{2}(0, T; \mathbb{H}) \ \ := \ \text{the space of measurable functions} \ \mathbf{u} : [0, T] \to \mathbb{H} \ \text{with the}
\\ &\qquad \qquad \qquad \quad \ \text{topology} \ \mathcal{T}_3,
\\ &L^{2}(0, T; H^1) \ \ := \ \text{the space of measurable functions} \ \mathbf{d} : [0, T] \to H^1 \ \text{with the}
\\ &\qquad \qquad \qquad \quad \ \text{topology} \ \mathcal{T}'_3.
\end{align*}

Let $\mathbb{H}_{w}$ denote the Hilbert space $\mathbb{H}$ endowed with the weak topology. Let us consider the space
\begin{align*}
&\mathbb{D}([0, T]; \mathbb{H}_w) \ \ := \ \text{the space of weakly c\`adl\`ag functions} \ \mathbf{u} : [0, T] \to \mathbb{H} \ \text{with the}
\\ &\qquad \qquad \qquad \qquad \text{weakest topology} \ \mathcal{T}_4 \ \text{such that for all} \ h \in \mathbb{H} \ \text{the mappings}
\\ &\qquad \qquad \mathbb{D}([0, T]; \mathbb{H}_w) \ni \mathbf{u} \mapsto (\mathbf{u}(\cdot), h)_{\mathbb{H}} \in \mathbb{D}([0, T]; \mathbb{R}) \ \text{are continuous.}
\end{align*} 

In particular, $\mathbf{u}_n \to \mathbf{u}$ in $\mathbb{D}([0, T]; \mathbb{H}_w)$ iff for all $h \in \mathbb{H}:$
\[(\mathbf{u}_n (\cdot), h)_{\mathbb{H}} \to (\mathbf{u}(\cdot), h)_{\mathbb{H}} \ \text{in the space} \ \mathbb{D}([0, T]; \mathbb{R}).\]

Similarly we define $\mathbb{C}([0, T]; H^1_w)$ with the topology $\mathcal{T}'_4.$\\

\noindent The following two results are due to \cite{EM}, where the details of the proof can be found.

\begin{theorem}\label{compau}
$($Compactness Criterion for $\mathbf{u}$$)$ Let us consider the space
\begin{align}
&\mathcal{Z}_{T, 1} = L^{2}_{w}(0, T; \mathbb{V}) \cap L^2(0, T; \mathbb{H}) \cap \mathbb{D}([0, T]; \mathbb{V}') \cap \mathbb{D}([0, T]; \mathbb{H}_w)
\end{align}
and $\mathscr{T}^1$ be the supremum of the corresponding topologies. Then a set $\mathcal{K}_1 \subset \mathcal{Z}_{T, 1}$ is $\mathscr{T}^1$-relatively compact if the following three conditions hold
\begin{enumerate}
\item
$\sup_{\mathbf{u} \in \mathcal{K}_1} \sup_{s \in [0, T]} |\mathbf{u}(s)|_{\mathbb{H}} < \infty,$
\item
$\sup_{\mathbf{u} \in \mathcal{K}_1} \int_{0}^{T} \|\mathbf{u}(s)\|^{2}_{\mathbb{V}} \,ds < \infty,$ i.e., $\mathcal{K}_1$ is bounded in $L^{2}(0, T; \mathbb{V}),$
\item
$\lim_{\delta \to 0} \sup_{\mathbf{u} \in \mathcal{K}_1} \mathcal{W}_{[0, T], \mathbb{V}'}(\mathbf{u}; \delta) = 0.$ 
\end{enumerate}
\end{theorem}

\begin{theorem}\label{compad}
$($Compactness Criterion for $\mathbf{d}$$)$ Let us consider the space
\begin{align}
&\mathcal{Z}_{T, 2} = L^{2}_{w}(0, T; H^2) \cap L^2(0, T; H^1) \cap \mathbb{C}([0, T]; (H^2)') \cap \mathbb{C}([0, T]; H^{1}_{w})
\end{align}
and let $\mathscr{T}^2$ be the supremum of the corresponding topologies. Then a set $\mathcal{K}_2 \subset \mathcal{Z}_{T, 2}$ is $\mathscr{T}^2$-relatively compact if the following three conditions hold
\begin{enumerate}
\item
$\sup_{\mathbf{d} \in \mathcal{K}_2} \sup_{s \in [0, T]} \|\mathbf{d}(s)\|_{H^1} < \infty,$
\item
$\sup_{\mathbf{d} \in \mathcal{K}_2} \int_{0}^{T} \|\mathbf{d}(s)\|^{2}_{H^2} \,ds < \infty,$ i.e., $\mathcal{K}_2$ is bounded in $L^{2}(0, T; H^2),$
\item
$\lim_{\delta \to 0} \sup_{\mathbf{d} \in \mathcal{K}_2} \sup_{s, t \in [0, T], |t-s| \leq \delta} \rho(\mathbf{d}(t), \mathbf{d}(s))= 0.$ 
\end{enumerate}
\end{theorem}

\noindent Note that $\mathcal{Z}_{T, 1}$ and $\mathcal{Z}_{T, 2}$ are not Polish spaces.

\subsection{The Aldous Condition}

Here $(\mathbb{M}, \rho)$ is a complete, separable metric space. Let $(\Omega, \mathcal{F}, \mathbb{F}, \mathbb{P})$ be a probability space with ususal hypotheses. Let $(X_{n})_{n \in \mathbb{N}}$ be a sequence of $\mathbb{F}$-adapted and $\mathbb{M}$-valued processes.

We have the following definition due to \cite{AM}.
\begin{definition}
Let $(X_n)$ be a sequence of $\mathbb{M}$-valued random variables. The sequence of laws of these processes form a {\bf Tight sequence} if and only if
\begin{enumerate}
\item[{\bf [T]}]
$ \ \ \ \forall \,\epsilon > 0 \ \ \forall \,\eta > 0 \ \ \exists \,\delta > 0:$
\[\sup_{n \in \mathbb{N}} \mathbb{P} \big\{ \mathcal{W}_{[0, T], \mathbb{M}} (X_n, \delta) > \eta \big\} \leq \epsilon.\]
\end{enumerate}
\end{definition}
where $\mathcal{W}_{[0, T], \mathbb{M}}$ is defined in \eqref{matw}.

\begin{definition}
A sequence $(X_{n})_{n \in \mathbb{N}}$ satisfies the {\bf Aldous condtion} in the space $\mathbb{M}$ if and only if
\begin{enumerate}
\item[{\bf [A]}]
$ \ \ \ \forall \,\epsilon > 0 \ \ \forall \,\eta > 0 \ \ \exists \,\delta > 0$ such that for every sequence $(\tau_{n})_{n \in \mathbb{N}}$ of $\mathbb{F}$-stopping times with $\tau_{n} \leq T$ one has
\[\sup_{n \in \mathbb{N}} \ \sup_{0 \leq \theta \leq \delta} \mathbb{P} \big\{\big |X_{n}(\tau_{n} + \theta) - X_{n}(\tau_{n})|_{\mathbb{M}} \geq \eta \big\} \leq \epsilon.\]
\end{enumerate}
\end{definition}

\begin{lemma}
Condition {\bf [A]} implies condition {\bf [T]}.
\end{lemma}
\begin{proof}
See Theorem 2.2.2 of \cite{AM}.
\end{proof}

\begin{lemma}\label{aco}
Let $(E, \| \cdot\|_{E})$ be a separable Banach space and let $(X_{n})_{n \in \mathbb{N}}$ be a sequence of $E$-valued random variables. Assume that for every sequence $(\tau_{n})_{n \in \mathbb{N}}$ of $\mathbb{F}$-stopping times with $\tau_{n} \leq T$ and for every $n \in \mathbb{N}$ and $\theta \geq 0$ the following condition holds
\begin{align}\label{aco1}
\mathbb{E}\big[ \| X_n(\tau_n + \theta) - X_n(\tau_n)\|^{\alpha}_{E}\big] \leq C \theta^{\beta}
\end{align}
for some $\alpha, \beta > 0$ and some constant $C >0.$ Then the sequence $(X_{n})_{n \in \mathbb{N}}$ satisfies the {\bf Aldous condtion} in the space E.
\end{lemma}
\begin{proof}
See \cite{EM} for the proof.
\end{proof}

In the view of Theorem \ref{compau}, to show the law of $\mathbf{u}_n$ is tight, we need the following result
\begin{corollary}\label{tc1}
Let $(\mathbf{u}_{n})_{n \in \mathbb{N}}$ be a sequence of c\`adl\`ag $\mathbb{F}$-adapted, $\mathbb{V}'$-valued processes such that
\begin{enumerate}
\item[$(a')$]
there exists a positive constant $\tilde{C}_1$ such that
\[\sup_{n \in \mathbb{N}} \mathbb{E} \bigg[ \sup_{s \in [0, T]} |\mathbf{u}_n(s)|_{\mathbb{H}} \bigg] \leq \tilde{C}_1,\]
\item[$(b')$]
there exists a positive constant $\bar{C}_2$ such that
\[\sup_{n \in \mathbb{N}} \mathbb{E} \bigg[ \int_{0}^{T} \|\mathbf{u}_n(s)\|^2_{\mathbb{V}} \,ds \bigg] \leq \bar{C}_2,\]
\item[$(c')$]
$(\mathbf{u}_{n})_{n \in \mathbb{N}}$ satisfies the Aldous condition in $\mathbb{V}'.$
\end{enumerate}
Let $\mathbb{P}^1_n$ be the law of $\mathbf{u}_n$ on $\mathcal{Z}_{T, 1}$. Then for every $\epsilon > 0$ there exists a compact subset $K^1_{\epsilon}$ of $\mathcal{Z}_{T, 1}$ such that 
\[\mathbb{P}^1_n(K^1_{\epsilon}) \geq 1-\epsilon.\] 
\end{corollary}

Similarly in the view of Theorem \ref{compad}, to show the law of $\mathbf{d}_n$ is tight, we need the following result
\begin{corollary}\label{tc2}
Let $(\mathbf{d}_{n})_{n \in \mathbb{N}}$ be a sequence of continuous $\mathbb{F}$-adapted, $(H^2)'$-valued processes such that
\begin{enumerate}
\item[$(a'')$]
there exists a positive constant $\tilde{C}'_1$ such that
\[\sup_{n \in \mathbb{N}} \mathbb{E} \bigg[ \sup_{s \in [0, T]} |\mathbf{d}_n(s)|_{H^1} \bigg] \leq \tilde{C}'_1,\]
\item[$(b'')$]
there exists a positive constant $\bar{C}'_2$ such that
\[\sup_{n \in \mathbb{N}} \mathbb{E} \bigg[ \int_{0}^{T} \|\mathbf{d}_n(s)\|^2_{H^2} \,ds \bigg] \leq \bar{C}'_2,\]
\item[$(c'')$]
$(\mathbf{d}_{n})_{n \in \mathbb{N}}$ satisfies the Aldous condition in $(H^2)'.$
\end{enumerate}
Let $\mathbb{P}^2_n$ be the law of $\mathbf{d}_n$ on $\mathcal{Z}_{T, 2}$. Then for every $\epsilon > 0$ there exists a compact subset $K^2_{\epsilon}$ of $\mathcal{Z}_{T, 2}$ such that 
\[\mathbb{P}^2_n(K^2_{\epsilon}) \geq 1-\epsilon.\] 
\end{corollary}

\subsection{Skorokhod Embedding Theorems}

We have the following Jakubowski's version of the Skorokhod Theorem due to \cite{Jak}.
\begin{theorem}
Let $(\mathscr{G}, \tau)$  be a topological space such that there exists a sequence $(g_m)$ of continuous functions $g_m : \mathscr{G} \to \mathbb{R}$ that separates points of $\mathscr{G}.$ Let $(Z_n)$ be a sequence of $\mathscr{G}$-valued random variables. Suppose that for every $\epsilon >0$ there exists a compact subset $G_{\epsilon} \subset \mathscr{G}$ such that
\[\sup_{n \in \mathbb{N}} \mathbb{P}(\{ Z_n \in G_{\epsilon}\}) > 1-\epsilon.\]
Then there exists a subsequence $(Z_{n_{k}})_{k \in \mathbb{N}},$ a sequence $(X_k)_{k \in \mathbb{N}}$ of $\mathscr{G}$-valued random variables and an $\mathscr{G}$-valued random variable $X$ defined on some probability space $(\Omega, \mathcal{F}, \mathbb{P})$ such that 
\[\mathscr{L}aw(Z_{n_{k}}) = \mathscr{L}aw(X_k), \qquad k =1, 2, \cdots\]
and forall $\omega \in \Omega,$
\[X_k(\omega) \xrightarrow{\tau} X(\omega) \quad k \to \infty.\] 
\end{theorem}

We have the following version of Skorokhod Theorem due to \cite{EM} and \cite{BHa}.

\begin{theorem}
Let $E_1, E_2$ be two separable Banach spaces and let $\pi_i : E_1 \times E_2 \to E_i, i=1, 2,$ be the projection into $E_i,$ i.e.
\[E_1 \times E_2 \ni \chi = (\chi_1, \chi_2) \to \pi_i(\chi) \in E_i.\]
Let $(\Omega, \mathcal{F}, \mathbb{P})$ be a probability space and let $\chi_n : \Omega \to E_1 \times E_2, n \in \mathbb{N},$ be a family of random variables such that the sequence $\{ \mathscr{L}aw(\chi_n), n \in \mathbb{N}\}$ is weakly convergent on $E_1 \times E_2.$ Finally let us assume that there exists a random variable $\rho: \Omega \to E_1$ such that $\mathscr{L}aw( \pi_1 \circ \chi_n) = \mathscr{L}aw(\rho), \forall n \in \mathbb{N}.$

Then there exists a probability space $(\bar{\Omega}, \bar{\mathcal{F}}, \bar{\mathbb{P}}),$ a family of $E_1 \times E_2$-valued random variables $\{\bar{\chi}_n, n \in \mathbb{N}\}$ on $(\bar{\Omega}, \bar{\mathcal{F}}, \bar{\mathbb{P}})$ and a random variable $\chi_{\ast} : \bar{\Omega} \to E_1 \times E_2$ such that
\begin{enumerate}
\item
$\mathscr{L}aw(\bar{\chi}_n) = \mathscr{L}aw(\chi_n), \forall n \in \mathbb{N};$
\item
$\bar{\chi}_n \to \chi_{\ast}$ in $E_1 \times E_2$ a.s.;
\item
$\pi_1 \circ \bar{\chi}_n(\bar{\omega}) = \pi_1 \circ \chi_{\ast}(\bar{\omega})$ forall $\bar{\omega} \in \bar{\Omega}.$
\end{enumerate}
\end{theorem}

\noindent We will use the following version of the Skorokhod embedding theorem in our paper. See Corollary 5.3 of \cite{ElM}.

\begin{theorem}\label{set3} 
Let $\mathscr{X}_1$ be a separable complete metric space and let $\mathscr{X}_2$ be a topological space such that there exists a sequence $\{ f_l\}_{l \in \mathbb{N}}$ of continuous functions  $f_l : \mathscr{X}_2 \to \mathbb{R}$ separating points of $\mathscr{X}_2.$ Let $\mathscr{X} := \mathscr{X}_1 \times \mathscr{X}_2$ with the Tychonoff topology induced by the projections
\begin{align*}
\pi_i : \mathscr{X}_1 \times \mathscr{X}_2 \to \mathscr{X}_i, \qquad i = 1, 2.
\end{align*}
Let $(\Omega, \mathcal{F}, \mathbb{P})$ be a probability space and let $\mathcal{X}_n : \Omega \to \mathscr{X}_1 \times \mathscr{X}_2,$ $n \in \mathbb{N},$ be a family of random variables such that the sequence $\{ \mathscr{L}(\mathcal{X}_n), n \in \mathbb{N}\}$ is tight on $\mathscr{X}_1 \times \mathscr{X}_2.$ Finally let us assume that there exists a random variable $\rho : \Omega \to \mathscr{X}_1$ such that $\mathscr{L}( \pi_1 \circ \mathcal{X}_n) = \mathscr{L}(\rho)$ for all $n \in \mathbb{N}.$

Then there exists a subsequence $\big(\mathcal{X}_{n_k} \big)_{k \in \mathbb{N}},$ a probability space $(\bar{\Omega}, \bar{\mathcal{F}}, \bar{\mathbb{P}}),$ a family of $\mathscr{X}_1 \times \mathscr{X}_2$-valued random variables $\{ \bar{\mathcal{X}_k}, k \in \mathbb{N}\}$ on $(\bar{\Omega}, \bar{\mathcal{F}}, \bar{\mathbb{P}})$ and a random variable $\mathcal{X}_{\ast} : \bar{\Omega} \to \mathscr{X}_1 \times \mathscr{X}_2$ such that 
\begin{enumerate}
\item
$\mathscr{L} (\bar{\mathcal{X}_k}) = \mathscr{L} (\mathcal{X}_{n_k})$ for all $k \in \mathbb{N};$
\item
$\bar{\mathcal{X}_k} \to \mathcal{X}_{\ast}$ in $\mathscr{X}_1 \times \mathscr{X}_2$ a.s. as $k \to \infty;$
\item
$\pi_1 \circ \bar{\mathcal{X}_k}(\bar{\omega}) = \pi_1 \circ \bar{\mathcal{X}_{\ast}}(\bar{\omega})$ for all $\bar{\omega} \in \bar{\Omega}.$
\end{enumerate}
\end{theorem}
\begin{proof}
For proof see Appendix B of \cite{EM}.
\end{proof}

\section{Energy Estimates}\label{EE}

Our aim is to prove the existence of a martingale solution of \eqref{r1st}-\eqref{r2nd}. We approach by Galerkin method. The Galerkin approximations generate a sequence of probability measures on appropiate functional space. We will show this sequences of laws are tight.

\subsection{The Faedo-Galerkin Approximation}

Our proof of existence of weak martingale solution depends on the Galerkin approximation. Let $\{ \varrho_i\}_{i = 1}^{\infty}$ be the orthonormal basis of $\mathbb{H}$ composed of eigenfunctions of the stokes operator $\mathscr{A}$. Let $\{ \varsigma_i\}_{i = 1}^{\infty}$ be the orthonormal basis of $L^2$ consisting of the eigenfunctions of the Neumann Laplacian $\mathcal{A}$. Let us define the following finite dimensional spaces for any $n \in \mathbb{N}$ 
\begin{align*}
&\mathbb{H}_{n} := Linspan\{ \varrho_1, \dots, \varrho_n\},
\\ &\mathbb{L}_n := Linspan\{ \varsigma_1, \dots, \varsigma_n\}.
\end{align*}

Our aim is to derive uniform estimates for the solution of the projection of \eqref{r1st}-\eqref{r2nd} onto the finite dimensional space $\mathbb{H}_{n} \times \mathbb{L}_n$, i.e., its Galerkin approximation. For this let us denote by $P_n$ the projection from $\mathbb{H}$ onto $\mathbb{H}_{n}$ and $\tilde{P}_n$ be the projection from $L^2$ onto $\mathbb{L}_n.$ We consider the following locally Lipschitz mappings:
\begin{align*}
&B_n : \mathbb{H}_{n} \ni \mathbf{u} \mapsto P_n B(\mathbf{u}, \mathbf{u}) \in \mathbb{H}_{n},
\\ &M_n : \mathbb{L}_n \ni \mathbf{d} \mapsto P_n M(\mathbf{d}) \in \mathbb{H}_{n},
\\ &f_n : \mathbb{L}_n \ni \mathbf{d} \mapsto \tilde{P}_n f(\mathbf{d}) \in \mathbb{L}_n,
\\ &\tilde{B}_n : \mathbb{H}_{n} \times \mathbb{L}_n \ni (\mathbf{u}, \mathbf{d}) \mapsto \tilde{P}_n \tilde{B}(\mathbf{u}, \mathbf{d}) \in \mathbb{L}_{n}.
\end{align*}

Let $P_n \mathbf{u}_{0} = \mathbf{u}_{0n}$ and $\tilde{P}_n \mathbf{d}_{0} = \mathbf{d}_{0n}$. So the Galerkin approximation of the problem \eqref{r1st}-\eqref{r2nd} is
\begin{align}\label{g1st}
 d\mathbf{u}_{n}(t) + \big[ \mathscr{A}\mathbf{u}_{n}(t) &+ B_{n}(\mathbf{u}_{n}(t)) + M_{n}(\mathbf{d}_{n}(t)) \big] dt  =  \int_{Y} P_{n}F(t, \mathbf{u}_{n}(t^{-}); y) \, \tilde{\eta}(dt, dy),
\end{align}
\begin{align}\label{g2nd}
 d\mathbf{d}_{n}(t) + \big[ \mathcal{A} \mathbf{d}_{n}(t) + \tilde{B}_{n}(\mathbf{u}_{n}(t), \mathbf{d}_{n}(t)) + f_{n}(\mathbf{d}_{n}(t)) \big] dt = 0.
\end{align}

The equations \eqref{g1st}-\eqref{g2nd} with initial conditions $\mathbf{u}_{n}(0) = \mathbf{u}_{0n}$ \ \ and \ \ $\mathbf{d}_{n}(0) = \mathbf{d}_{0n}$, form a system of Stochastic Differential Equations with locally Lipschitz coefficients. See \cite{BHR} for details.

Here we have followed the standard finite dimensional convention (see \cite{Da, IW, PZ}) of using left limit $t^-$ in the stochastic integral with respect to time homogeneous compensated Poisson random measure.

Now consider the following mappings
\begin{align*}
&B'_{n}(\mathbf{u}) := P_n B(\chi_{n}^{1} (\mathbf{u}), \mathbf{u}), \quad \mathbf{u} \in \mathbb{H}_n,
\\ &M'_{n}(\mathbf{d}) := P_n M(\chi_{n}^{2} (\mathbf{d}), \mathbf{d}), \quad \mathbf{d} \in \mathbb{L}_n,
\\\ &\tilde{B}'_{n}(\mathbf{u}, \mathbf{d}) := P_n \tilde{B}(\chi_{n}^{1} (\mathbf{u}), \mathbf{d}), \quad \mathbf{u} \in \mathbb{H}_n, \mathbf{d} \in \mathbb{L}_n,
\end{align*}
where $\chi_{n}^{1} : \mathbb{H} \to \mathbb{H}$ is defined by $\chi_{n}^{1}(\mathbf{u}) = \theta_n(|\mathbf{u}|_{\mathbb{V}'}) \mathbf{u},$ and $\chi_{n}^{2} : L^2 \to L^2$ is defined by $\chi_{n}^{2}(\mathbf{d}) = \theta_n(|\mathbf{d}|_{(H^2)'}) \mathbf{d},$ where $\theta_n : \mathbb{R} \to [0, 1]$ of class $\mathscr{C}^{\infty}$ such that
\[\theta_n(r) = 1 \quad if \quad r \leq n \quad and \quad \theta_n(r) = 0 \quad if \quad r \geq n+1.\]

Since $\mathbb{H}_n \subset \mathbb{H}$ and $\mathbb{L}_n \subset L^2,$ the mappings $B'_{n}, M'_{n}$ and $\tilde{B}'_{n}$ are well defined. Moreover, $B'_{n} : \mathbb{H}_n \to \mathbb{H}_n, M'_{n} : \mathbb{L}_n \to \mathbb{H}_{n},$ and $\tilde{B}'_{n} : \mathbb{H}_{n} \times \mathbb{L}_n \to \mathbb{L}_n$ is globally Lipschitz continuous.

Let us consider the Faedo-Galerkin approximation in the space $\mathbb{H}_n$ and $\mathbb{L}_n,$

\begin{align}\label{ga1st}
 d\mathbf{u}_{n}(t) + \big[ \mathscr{A}\mathbf{u}_{n}(t) &+ B'_{n}(\mathbf{u}_{n}(t)) + M'_{n}(\mathbf{d}_{n}(t)) \big] dt  =  \int_{Y} P_{n}F(t, \mathbf{u}_{n}(t^{-}); y) \tilde{\eta}(dt, dy),
\end{align}
\begin{align}\label{ga2nd}
 d\mathbf{d}_{n}(t) + \big[ \mathcal{A} \mathbf{d}_{n}(t) + \tilde{B}'_{n}(\mathbf{u}_{n}(t), \mathbf{d}_{n}(t)) + f_{n}(\mathbf{d}_{n}(t)) \big] dt = 0.
\end{align}

\noindent We refer to \cite{ABW} for the proof of the following two results.

\begin{lemma}
For each $n \in \mathbb{N},$ there exists a unique global, $\mathbb{F}$-adapted, $\mathbb{H}_n \times \mathbb{L}_n$ -valued processes $(\mathbf{u}_{n}, \mathbf{d}_{n})$ satisfying the Galerkin equations \eqref{ga1st}-\eqref{ga2nd}. 
\end{lemma}

\begin{corollary}
For each $n \in \mathbb{N},$ the equations \eqref{g1st}-\eqref{g2nd} has a unique local maximal solution. 
\end{corollary}

\begin{lemma}\label{Pnu}
In particular, for every $(\mathbf{u}, \mathbf{d}) \in \mathbb{V} \times H^1$ we have
\begin{enumerate}
\item
$\lim_{n \to \infty} \| P_n \mathbf{u} - \mathbf{u}\|_{\mathbb{V}} = 0.$
\item
$\lim_{n \to \infty} \| \tilde{P}_n \mathbf{d} - \mathbf{d}\|_{H^1} = 0.$
\end{enumerate}
\end{lemma}
\noindent For details see \cite{BrM}.

\subsection{A Priori Estimates}

The processes $\big(\mathbf{u}_{n}\big)_{n \in \mathbb{N}}$ and $\big(\mathbf{d}_{n}\big)_{n \in \mathbb{N}}$ satisfy the following estimates.

\begin{proposition}\label{dgrdn}
For any $p \geq 2,$ there exists a positive constant $C,$ depending on $p$ such that
\begin{align*}
\sup_{n \in \mathbb{N}}\bigg( \mathbb{E} \bigg[ \sup_{t \in [0,T]} \big| \mathbf{d}_{n}(t) \big|_{L^2}^p +  \int_{0}^{T} 
\big| \mathbf{d}_n(s) \big|_{L^2}^{p-2} \big( \| \mathbf{d}_{n}(s)\|^2 &+ \|\mathbf{d}_{n}(s)\|^{2N+2}_{L^{2N+2}}\big) \, 
ds \bigg] \bigg) \leq \mathbb{E} \ \mathfrak{C}_{0}(p, T).
\end{align*} 
where $\mathfrak{C}_{0}(p, T) := |\mathbf{d}_{0n}|^p \big( 1+CTe^{CT}\big),$ which is independent of $n \in \mathbb{N}$ 
and $t \in [0, T]$.
\end{proposition}
\begin{proof}
Consider the equation \eqref{g2nd} of the Faedo-Galerkin approximation  as follows
\begin{align}\label{gg2nd}
 d\mathbf{d}_{n}(t) + \big[ \mathcal{A} \mathbf{d}_{n}(t) + \tilde{B}_{n}(\mathbf{u}_{n}(t), \mathbf{d}_{n}(t)) + f_{n}(\mathbf{d}_{n}(t)) \big] dt = 0.
\end{align}

For $p \geq 2,$ let $\psi(.)$ be the mapping defined by 
\begin{align}\label{psi}
\psi(\mathbf{d}(t)) := \frac{1}{p} |\mathbf{d}(t)|^p, \ \ \ \ \mathbf d \in L^2.
\end{align}

The first Fr\'{e}chet derivative is 
\begin{align}
\psi '(\mathbf{d})[h] = |\mathbf{d}|^{p-2} \langle \mathbf{d}, h \rangle.
\end{align}

So applying this to the sequence $\mathbf{d}_{n}(t)$ we get,
\begin{align}\label{fre}
 d\psi(\mathbf{d}_{n}(t)) &= \psi '(\mathbf{d}_n)[d\mathbf{d}_n(t)] =  |\mathbf{d}_n|^{p-2} \langle d\mathbf{d}_n(t), \mathbf{d}_n(t)\rangle.
\end{align}

From \eqref{gg2nd} we have 
\begin{align}\label{inn}
\langle d\mathbf{d}_n(t), \mathbf{d}_n(t)\rangle &= \big \langle - \mathcal{A} \mathbf{d}_{n}(t) - \tilde{B}_{n}(\mathbf{u}_{n}(t), \mathbf{d}_{n}(t)) - f_{n}(\mathbf{d}_{n}(t)),  \mathbf{d}_{n}(t) \big \rangle \nonumber
\\ & = \| \mathbf{d}_{n}(t)\|^2 - \big \langle  f_{n}(\mathbf{d}_{n}(t)),  \mathbf{d}_{n}(t) \big \rangle.
\end{align}

From \eqref{inn} we further write \eqref{fre} as
\begin{align}\label{Psest}
\psi(\mathbf{d}_{n}(t)) = \psi(\mathbf{d}_{0}) &-  \int_{0}^{t} |\mathbf{d}_n|^{p-2} \| \mathbf{d}_{n}\|^2 \, ds -  \int_{0}^{t} |\mathbf{d}_n|^{p-2} \big \langle f_{n}(\mathbf{d}_{n}(s)),  \mathbf{d}_{n}(s) \big \rangle \, ds.
\end{align}

Then from \eqref{psi}, \eqref{Psest} and taking integration over all $t \in [0, T]$ we have,
\begin{align}\label{ex}
|\mathbf{d}_{n}(t)|^p  &+ \int_{0}^{t} |\mathbf{d}_n(s)|^{p-2} \| \mathbf{d}_{n}(s)\|^2 + \int_{0}^{t} |\mathbf{d}_n(s)|^{p-2} \big \langle f_{n}(\mathbf{d}_{n}(s)),  \mathbf{d}_{n}(s) \big \rangle \, ds \leq |\mathbf{d}_{0}|^p .
\end{align}

From Assumption (C) of section \ref{assu} we observe,
\begin{align}\label{pol}
\big \langle f(\mathbf{d}), \mathbf{d} \big \rangle &= \big \langle \tilde{f}(|\mathbf{d}|^2)\mathbf{d}, \mathbf{d}\big \rangle = \int_{\mathbb{O}} \tilde{f}(|\mathbf{d}(x)|^2) |\mathbf{d}(x)|^2 \, dx \nonumber
\\ &= \int_{\mathbb{O}} \sum_{k=0}^{N} a_k(|\mathbf{d}(x)|^2)^{k+1} \, dx = \int_{\mathbb{O}} \sum_{l=1}^{N+1} a_{l-1}(|\mathbf{d}(x)|^2)^{l} \, dx.
\end{align}

But from Lemma 8.7 of \cite{BM} we have,
\begin{align}
\int_{\mathbb{O}} |\mathbf{d}(x)|^{2N+2} \, dx - C \int_{\mathbb{O}}  |\mathbf{d}(x)|^{2} \,dx \leq \big \langle f(\mathbf{d}), \mathbf{d} \big \rangle.
\end{align}

Using this result in \eqref{ex} we get
\begin{align}
|\mathbf{d}_{n}(t)|^p &+ \int_{0}^{t} |\mathbf{d}_n(s)|^{p-2} \|\mathbf{d}_{n}(s)\|^2 \, ds + \int_{0}^{t} |\mathbf{d}_n(s)|^{p-2} \big(  \|\mathbf{d}_{n}(s)\|^{2N+2}_{L^{2N+2}} - C | \mathbf{d}_{n}(s) |^2 \big) \, ds \leq |\mathbf{d}_{0n}|^p,
\end{align}

which further implies
\begin{align}
\label{eqn-3.11}
|\mathbf{d}_{n}(t)|^p &+ \int_{0}^{t} |\mathbf{d}_n(s)|^{p-2} \| \mathbf{d}_{n}(s)\|^2 \, ds + \int_{0}^{t} |\mathbf{d}_n(s)|^{p-2}  \|\mathbf{d}_{n}(s)\|^{2N+2}_{L^{2N+2}} \, ds \nonumber
\\ &\leq |\mathbf{d}_{0n}|^p + C \int_{0}^{t} | \mathbf{d}_{n}(s) |^p \, ds.
\end{align}

Therefore, 
\begin{align}
\label{eqn-3.12}
\sup_{0\leq s \leq t} |\mathbf{d}_{n}(s)|^p &+ \int_{0}^{t} |\mathbf{d}_n(s)|^{p-2} \| \mathbf{d}_{n}(s)\|^2 \,ds + \int_{0}^{t} |\mathbf{d}_n(s)|^{p-2}  \|\mathbf{d}_{n}(s)\|^{2N+2}_{L^{2N+2}} \, ds \nonumber
\\ &\leq |\mathbf{d}_{0n}|^p + C \int_{0}^{t} | \mathbf{d}_{n}(s) |^p \, ds.
\end{align}

Let us fix $t \geq 0$. Then by \eqref{eqn-3.11}, for any $r \in [0,t]$, 
\begin{align}\label{eqn-3.13}
 |\mathbf{d}_{n}(r)|^p &+ \int_{0}^{r} |\mathbf{d}_n(s)|^{p-2} \| \mathbf{d}_{n}(s)\|^2 \,ds + \int_{0}^{r} |\mathbf{d}_n(s)|^{p-2}  \|\mathbf{d}_{n}(s)\|^{2N+2}_{L^{2N+2}} \, ds \nonumber
\\ &\leq |\mathbf{d}_{0n}|^p + C \int_{0}^{r} | \mathbf{d}_{n}(s) |^p \, ds \leq |\mathbf{d}_{0n}|^p + C \int_{0}^{t} | \mathbf{d}_{n}(s) |^p \, ds
\end{align}

Hence, 
\begin{align}
\label{eqn-3.14}
 |\mathbf{d}_{n}(r)|^p &\leq |\mathbf{d}_{0n}|^p + C \int_{0}^{t} | \mathbf{d}_{n}(s) |^p \, ds
\end{align}

Next, in the above we take supremum over $r \in [0,t]$ and taking expectation we get, 
\begin{align}
\label{eqn-3.15}
\mathbb{E} \bigg[ \sup_{r \in [0,t]}  |\mathbf{d}_{n}(r)|^p \bigg] \leq \mathbb{E} |\mathbf{d}_{0n}|^p + C \int_{0}^{t} \mathbb{E} \big[ | \mathbf{d}_{n}(s) |^p \big] \, ds.
\end{align}

The Gronwall lemma yields
\begin{align}
\label{eqn-3.15}
\mathbb{E} \bigg[ \sup_{r \in [0,t]}  |\mathbf{d}_{n}(r)|^p\bigg] \leq \mathbb{E} |\mathbf{d}_{0n}|^p  e^{Ct}, \;\; t\geq 0.
\end{align}

From \eqref{eqn-3.13} and \eqref{eqn-3.15} we also get 
\begin{align}
 &\mathbb{E} \bigg[   \int_{0}^{T} |\mathbf{d}_n(s)|^{p-2} \big( \| \mathbf{d}_{n}(s)\|^2 + \|\mathbf{d}_{n}(s)\|^{2N+2}_{L^{2N+2}}\big) \, ds \bigg] \nonumber
\\ & \ \ \leq \mathbb{E} \bigg[ |\mathbf{d}_{0n}|^p + C \int_{0}^{T} | \mathbf{d}_{n}(s) |^p \, ds \bigg] \leq \mathbb{E} |\mathbf{d}_{0n}|^p + C \, \mathbb{E} |\mathbf{d}_{0n}|^p \int_{0}^{T}   e^{Ct}  \, dt.
\end{align}

 So we infer
\begin{align}\label{c0}
\sup_{n \in \mathbb{N}}\bigg( \mathbb{E} \bigg[ \sup_{t \in [0,T]} \big|\mathbf{d}_{n}(t) \big|_{L^2}^p +  \int_{0}^{T} 
\big| \mathbf{d}_n(s) \big|_{L^2}^{p-2} \big( \| \mathbf{d}_{n}(s)\|^2 &+ \|\mathbf{d}_{n}(s)\|^{2N+2}_{L^{2N+2}}\big) \, 
ds \bigg] \bigg) \leq \mathbb{E} \ \mathfrak{C}_{0}(p, T),
\end{align}
where $\mathfrak{C}_{0}(p, T) := |\mathbf{d}_{0n}|^p \big( 1+CTe^{CT}\big).$
\end{proof}

\noindent Now we derive the estimates for $\mathbf{u}_{n}$ and $\nabla \mathbf{d}_{n}.$

\begin{proposition}\label{ugrdn}
For every $p \geq 1,$ there exists a positive constant $C,$ depending on $p$ and $T$ such that
\begin{align*}
\sup_{n \in \mathbb{N}} &\bigg(\mathbb{E} \bigg[ \sup_{s \in [0, T]} \Big(\Psi(\mathbf{d}_{n}(s)) + \big| \mathbf{u}_{n}(s) \big|^{2}_{\mathbb{H}} \Big)^{p}\bigg] \nonumber
\\ &+ \mathbb{E} \bigg[ \int_{0}^{T} \Big( \|\mathbf{u}_{n}(s) \|^2 + \big|\Delta \mathbf{d}_{n}(s)- f_{n}(\mathbf{d}_{n}(s))\big|_{L^2}^2 \Big) \,ds \bigg]^{p}\bigg) \leq C_{p, T},
\end{align*}
where $\Psi(\mathbf{d}_{n}(s)) := \frac{1}{2} \|\mathbf{d}_{n}(s)\|^2 + \frac{1}{2} \int_{\mathbb{O}} \mathrm{F}(|\mathbf{d}_{n}(s)|^2) \, dx.$
\end{proposition}

\begin{proof}
For all $n \in \mathbb{N}$ and all $R > 0$ let us define 
\begin{align}\label{st}
\tau_{n}^{R} := \inf\{ t \geq 0 : | \mathbf{u}_{n}(t)|_{\mathbb{H}} \geq R\} \wedge T.
\end{align}
It's a stopping time, since the process $\left( \mathbf{u}_{n}(t) \right)_{t \in [0, T]}$ is $\mathbb{F}$-adapted and right-continuous. Moreover $\tau_{n}^{R} \uparrow T$, $\mathbb{P}$-a.s., as $R \uparrow \infty.$ 

Now consider the equation \eqref{g1st}
\begin{align}\label{gg1st}
 d\mathbf{u}_{n}(t) &+ \big[ \mathscr{A}\mathbf{u}_{n}(t) + B_{n}(\mathbf{u}_{n}(t)) + M_{n}(\mathbf{d}_{n}(t)) \big] dt  = \int_{Y} P_{n}F(t, \mathbf{u}_{n}(t^{-}); y) \,\tilde{\eta}(dt, dy).
\end{align}

\noindent Define a mapping $\phi(.)$ as follows
\[ \phi(\mathbf{u}(t)) := \frac{1}{2} |\mathbf{u}(t)|^2 := \frac{1}{2} |\mathbf{u}(t)|_{\mathbb{H}}^2, \ \ \ \ \ \mathbf{u} \in \mathbb{H}\]

The first Fr\'echet derivative is given by
\[\phi'(\mathbf{u})[h] = \langle \mathbf{u}, h\rangle.\]

Now applying It\^o's formula to $\phi(\mathbf{u}_{n}(t))$
\begin{align}\label{dphi}
d\phi(\mathbf{u}_{n}(t)) &= -\big\langle  \mathscr{A}\mathbf{u}_{n}(t) + B_{n}(\mathbf{u}_{n}(t)) + M_{n}(\mathbf{d}_{n}(t)), \mathbf{u}_{n}(t) \big\rangle dt \nonumber
\\ &\quad+ \frac{1}{2} \int_{Y} \big\{\phi \big( \mathbf{u}_{n}(s^{-}) +  P_{n}F(s, \mathbf{u}_{n}(s^{-}); y) \big) - \phi \big(\mathbf{u}_{n}(s^{-}) \big) \big\} \tilde{\eta}(ds, dy)\nonumber
\\ &\quad+ \frac{1}{2} \int_{Y} \bigg\{\phi \big(\mathbf{u}_{n}(s) +  P_{n}F(s, \mathbf{u}_{n}(s); y)\big) - \phi \big( \mathbf{u}_{n}(s)\big) \nonumber
\\ & \ \ \ \ \ \ \ \ \ \ \ \ \ \ \ \ \ \ \ \ \ \ \ \ \ \quad - \big\langle \phi' \big(\mathbf{u}_{n}(s)\big),  P_{n}F(s, \mathbf{u}_{n}(s); y)\big\rangle \bigg\} \, d\nu(y) \nonumber
\\&= -\big\langle  \mathscr{A}\mathbf{u}_{n}(t) + B_{n}(\mathbf{u}_{n}(t)) + M_{n}(\mathbf{d}_{n}(t)), \mathbf{u}_{n}(t) \big\rangle dt \nonumber
\\ &\quad+ \frac{1}{2} \int_{Y} \big\{|\mathbf{u}_{n}(s^{-}) +  P_{n}F(s, \mathbf{u}_{n}(s^{-}); y)|^2 - |\mathbf{u}_{n}(s^{-})|^2 \big\} \tilde{\eta}(ds, dy)\nonumber
\\ &\quad+ \frac{1}{2} \int_{Y} \bigg\{|\mathbf{u}_{n}(s) +  P_{n}F(s, \mathbf{u}_{n}(s); y)|^2 - |\mathbf{u}_{n}(s)|^2 \nonumber
\\ & \ \ \ \ \ \ \ \ \ \ \ \ \ \ \ \ \ \ \ \ \ \ \ \ \ \ - 2 \big\langle \mathbf{u}_{n}(s),  P_{n}F(s, \mathbf{u}_{n}(s); y)\big\rangle \bigg\} \, d\nu(y).
\end{align}

Now let $\Psi(.)$ be the mapping defined by 
\[\Psi(\mathbf{d}) = \frac{1}{2} \|\mathbf{d}\|^2 + \frac{1}{2} \int_{\mathbb{O}} \mathrm{F}(|\mathbf{d}|^2) \, dx. \]

The first Fr\'echet derivative is
\begin{align}
\Psi'(\mathbf{d})[g] &= \langle \nabla \mathbf{d}, \nabla g \rangle + \langle f(\mathbf{d}), g\rangle \nonumber
\\ &= \langle -\Delta \mathbf{d}+ f(\mathbf{d}), g\rangle.
\end{align}

Now applying this formula to $\Psi(\mathbf{d}_{n}(t))$ and from \eqref{gg2nd},
\begin{align}\label{dpsi}
d\Psi(\mathbf{d}_{n}(t)) &= \Psi'(\mathbf{d}_{n})[d\mathbf{d}_{n}(t)] = \langle -\Delta \mathbf{d}_{n}+ f_{n}(\mathbf{d}_{n}), d\mathbf{d}_{n}(t) \rangle \nonumber
\\ &= \langle -\Delta \mathbf{d}_{n}+ f_{n}(\mathbf{d}_{n}), - \mathcal{A} \mathbf{d}_{n}(t) - f_{n}(\mathbf{d}_{n}(t)) - \tilde{B}_{n}(\mathbf{u}_{n}(t), \mathbf{d}_{n}(t)) \rangle \nonumber
\\ &= -|\Delta \mathbf{d}_{n}(t)- f_{n}(\mathbf{d}_{n}(t))|^2 - \langle\tilde{B}_{n}(\mathbf{u}_{n}(t), \mathbf{d}_{n}(t)), -\Delta \mathbf{d}_{n}+ f_{n}(\mathbf{d}_{n})\rangle 
\end{align}

Now using integration by parts and the divergence free condition of $\mathbf{u}_{n}$ we get,
\begin{align}\label{bdm}
&\langle \tilde{B}_{n}(\mathbf{u}_{n}, \mathbf{d}_{n}), \Delta \mathbf{d}_{n} \rangle = \int_{\mathbb{O}} \mathbf{u}_{n}^{(j)} \frac{\partial \mathbf{d}_{n}^{(k)}}{\partial x_{j}} \frac{\partial^{2} \mathbf{d}_{n}^{(k)}}{\partial x_{i} \partial x_{i}} \, dx \nonumber
\\ &= - \int_{\mathbb{O}} \frac{\partial \mathbf{u}_{n}^{(j)}}{\partial x_{i}} \frac{\partial \mathbf{d}_{n}^{(k)}}{\partial x_{j}} \frac{\partial \mathbf{d}_{n}^{(k)}}{\partial x_{i}} \, dx - \int_{\mathbb{O}} \mathbf{u}_{n}^{(j)} \frac{\partial^{2} \mathbf{d}_{n}^{(k)}}{\partial x_{j} \partial x_{i}} \frac{\partial \mathbf{d}_{n}^{(k)}}{\partial x_{i}} \, dx \nonumber
\\ &= - \int_{\mathbb{O}} \frac{\partial \mathbf{u}_{n}^{(j)}}{\partial x_{i}} \frac{\partial \mathbf{d}_{n}^{(k)}}{\partial x_{j}} \frac{\partial \mathbf{d}_{n}^{(k)}}{\partial x_{i}} \, dx - \frac{1}{2} \int_{\mathbb{O}} \mathbf{u}_{n}^{(j)} \frac{\partial}{\partial x_{j}}(|\nabla \mathbf{d}_{n}|^2) \, dx \nonumber
\\ &= \langle M_n(\mathbf{d}_{n}, \mathbf{d}_{n}), \mathbf{u}_{n} \rangle + \frac{1}{2} \int_{\mathbb{O}} \frac{\partial \mathbf{u}_{n}^{(j)}}{\partial x_{j}} |\nabla \mathbf{d}_{n}|^2 \, dx = \langle M_n(\mathbf{d}_{n}), \mathbf{u}_{n} \rangle.
\end{align}

And
\begin{align}\label{bfz}
\big\langle \tilde{B}_{n}(\mathbf{u}_{n}, \mathbf{d}_{n}), f_{n}(\mathbf{d}_{n}) \big\rangle &= \int_{\mathbb{O}} \mathbf{u}_{n}^{(i)} \frac{\partial \mathbf{d}_{n}^{(j)}}{\partial x_{i}}  \tilde{f}_{n}(|\mathbf{d}_{n}|^2) \mathbf{d}_{n}^{(j)} \, dx \nonumber
\\ &= \frac{1}{2} \int_{\mathbb{O}} \mathbf{u}_{n}^{(i)} \frac{\partial \tilde{\mathrm{F}}_{n}(|\mathbf{d}_{n}|^2)}{\partial x_{i}} \,dx = \frac{1}{2} \big\langle \mathbf{u}_{n}, \nabla \tilde{\mathrm{F}}_{n}(|\mathbf{d}_{n}|^2)\big\rangle \nonumber
\\ &= -\frac{1}{2} \big\langle \nabla \cdot \mathbf{u}_{n}, \tilde{\mathrm{F}}_{n}(|\mathbf{d}_{n}|^2) \big\rangle = 0.
\end{align}

Using \eqref{bdm}, \eqref{bfz} and the fact that $\langle {B}_{n}(\mathbf{u}_{n}(t)), \mathbf{u}_{n}(t) \rangle = 0$ and adding equations \eqref{dphi} and \eqref{dpsi} we get
\begin{align}\label{siu}
d \big[ \Psi(\mathbf{d}_{n}(t)) &+ \phi(\mathbf{u}_{n}(t)) \big] + \big( \|\mathbf{u}_{n}(t) \|^2 + |\Delta \mathbf{d}_{n}(t)- f_{n}(\mathbf{d}_{n}(t))|^2 \big) \,dt \nonumber
\\ &= \frac{1}{2} \int_{Y} \big\{|\mathbf{u}_{n}(s^{-}) +  P_{n}F(s, \mathbf{u}_{n}(s^{-}); y)|^2 - |\mathbf{u}_{n}(s^{-})|^2 \big\} \tilde{\eta}(ds, dy)\nonumber
\\ & \ \ + \frac{1}{2} \int_{Y} \bigg\{ |\mathbf{u}_{n}(s) +  P_{n}F(s, \mathbf{u}_{n}(s); y)|^2 - |\mathbf{u}_{n}(s)|^2 \nonumber
\\ & \ \ \ \ \ \ \ \ \ \ \ \ \ \ \ \ \ \ \ \ \ \ \ - 2 \ \big\langle \mathbf{u}_{n}(s),  P_{n}F(s, \mathbf{u}_{n}(s); y)\big\rangle \bigg\} \, d\nu(y).
\end{align}

Now for all $t \in [0, T]$, we can rewrite the above equation as
\begin{align}\label{ij}
\mathbb{E} \big[ \Psi(\mathbf{d}_{n}( t \wedge \tau^{R}_{n})) &+ |\mathbf{u}_{n}(t \wedge \tau^{R}_{n}) |^2 \big] \nonumber
\\ &+ \mathbb{E} \bigg[ \int_{0}^{t \wedge \tau^{R}_{n}}\big( \|\mathbf{u}_{n}(s) \|^2 + |\Delta \mathbf{d}_{n}(s)- f_{n}(\mathbf{d}_{n}(s))|^2 \big) \, ds \bigg] \nonumber
\\ &\leq \mathbb{E} \big[ \Psi(\mathbf{d}_{0n}) + |\mathbf{u}_{0n} |^2 \big] + \mathbb{E} \big[ I_n(t \wedge \tau^{R}_{n}) + J_{n}(t \wedge \tau^{R}_{n})\big].
\end{align}

where
\begin{align}\label{intt}
I_n(t \wedge \tau^{R}_{n}) := \int_{0}^{t \wedge \tau^{R}_{n}} \int_{Y} \big\{|\mathbf{u}_{n}(s^{-}) +  P_{n}F(s, \mathbf{u}_{n}(s^{-}); y)|^2 - |\mathbf{u}_{n}(s^{-})|^2 \big\} \,\tilde{\eta}(ds, dy),
\end{align}
\begin{align}\label{jntt}
and \quad J_n(t \wedge \tau^{R}_{n}) := \int_{0}^{t \wedge \tau^{R}_{n}} \int_{Y} \bigg\{ & |\mathbf{u}_{n}(s) +  P_{n}F(s, \mathbf{u}_{n}(s); y)|^2 - |\mathbf{u}_{n}(s)|^2 \nonumber
\\ &-2 \big\langle \mathbf{u}_{n}(s),  P_{n}F(s, \mathbf{u}_{n}(s); y)\big\rangle \bigg\} \, d\nu(y) ds.
\end{align}

From Assumption (B) of Section \ref{assu} we have,
\begin{align}\label{jnes}
|J_n(t \wedge \tau^{R}_{n})| &= \int_{0}^{t \wedge \tau^{R}_{n}} \int_{Y} \bigg\{|\mathbf{u}_{n}(s) +  P_{n}F(s, \mathbf{u}_{n}(s); y)|^2 - |\mathbf{u}_{n}(s)|^2 \nonumber
\\ &\qquad \qquad \qquad \qquad \quad \  -2 \big\langle \mathbf{u}_{n}(s),  P_{n}F(s, \mathbf{u}_{n}(s); y)\big\rangle \bigg\} \, d\nu(y) ds \nonumber
\\ &\leq C \int_{0}^{t \wedge \tau^{R}_{n}} \int_{Y}  | P_{n}F(s, \mathbf{u}_{n}(s); y) |^2 \, d\nu(y) ds \leq C \int_{0}^{t \wedge \tau^{R}_{n}} \{ 1 + | \mathbf{u}_{n}(s) |^2 \} \,ds.
\end{align}

Thus by Fubini's theorem
\begin{align}\label{jnr}
\mathbb{E} \big[ |J_n(t \wedge \tau^{R}_{n})| \big] &\leq C(t \wedge \tau^{R}_{n}) + \int_{0}^{t \wedge \tau^{R}_{n}} \mathbb{E}\big[ | \mathbf{u}_{n}(s) |^2 \big] \,ds \nonumber
\\ &\leq C(t \wedge \tau^{R}_{n}) +  \int_{0}^{t \wedge \tau^{R}_{n}} \mathbb{E}\big[ \Psi(\mathbf{d}_{n}(s)) + | \mathbf{u}_{n}(s) |^2 \big] \,ds.
\end{align}

Now by definition of $\tau_{n}^R$ and from \eqref{lgf}, we observe that the process $I_n(t \wedge \tau^{R}_{n})$ is an integrable martingale. Hence
\begin{align}\label{inr}
\mathbb{E} \big[ |I_n(t \wedge \tau^{R}_{n})| \big] = 0 \quad \quad \forall \  t \in [0, T].
\end{align}

From \eqref{ij}, \eqref{jnr} and \eqref{inr} we get,
\begin{align}
\mathbb{E} \big[ &\Psi(\mathbf{d}_{n}( t \wedge \tau^{R}_{n})) + |\mathbf{u}_{n}( t \wedge \tau^{R}_{n}) |^2 \big] + \mathbb{E} \bigg[ \int_{0}^{t \wedge \tau^{R}_{n}}\big( \| \mathbf{u}_{n}(s) \|^2 + |\Delta \mathbf{d}_{n}(s)- f_{n}(\mathbf{d}_{n}(s))|^2 \big) \, ds \bigg] \nonumber
\\ &\leq \mathbb{E} \big[ \Psi(\mathbf{d}_{0n}) + |\mathbf{u}_{0n} |^2 \big] + \int_{0}^{t \wedge \tau^{R}_{n}} \mathbb{E}\big[ \Psi(\mathbf{d}_{n}(s)) + | \mathbf{u}_{n}(s) |^2 \big] \,ds.
\end{align}

Applying Gronwall's lemma we get
\begin{align}
\mathbb{E} \big[&\Psi(\mathbf{d}_{n}(t \wedge \tau^{R}_{n})) + |\mathbf{u}_{n}(t \wedge \tau^{R}_{n}) |^2 \big] + \mathbb{E} \bigg[ \int_{0}^{t \wedge \tau^{R}_{n}}\big( \| \mathbf{u}_{n}(s) \|^2 + |\Delta \mathbf{d}_{n}(s)- f_{n}(\mathbf{d}_{n}(s))|^2 \big) \, ds \bigg]  \leq \mathfrak{C}(T),
\end{align}

where $\mathfrak{C}(T) := \big( C + \mathbb{E} \big[ \Psi(\mathbf{d}_{0n}) + |\mathbf{u}_{0n} |_\mathbb{H}^2 \big] \big) \big( 1 + T e^{CT}\big).$ \\ As $\mathfrak{C}(T)$ is independent of $t \in [0, T],$ $R > 0$ and $n \in \mathbb{N},$ we can further write 
\begin{align}\label{c1}
\sup_{n \in \mathbb{N}} \bigg(\mathbb{E} \Big[&\Psi(\mathbf{d}_{n}(t \wedge \tau^{R}_{n})) + \big|\mathbf{u}_{n}(t \wedge \tau^{R}_{n}) \big|_{\mathbb{H}}^2 \Big]  \nonumber
\\ &+ \mathbb{E} \bigg[ \int_{0}^{t \wedge \tau^{R}_{n}}\big( \| \mathbf{u}_{n}(s) \|^2 + \big|\Delta \mathbf{d}_{n}(s)- f_{n}(\mathbf{d}_{n}(s))\big|_{L^2}^2 \big) \, ds \bigg]\bigg)  \leq \mathfrak{C}(T).
\end{align}

Now consider the integral version of \eqref{siu} and multiplying both sides by 2, we have
\begin{align}
\psi(s) &+ \int_{0}^{t} \big( \|\mathbf{u}_{n}(s) \|^2 + |\Delta \mathbf{d}_{n}(s)- f_{n}(\mathbf{d}_{n}(s))|^2 \big) \,ds \nonumber
\\ &\leq \psi(0) + \int_{0}^{t} \int_{Y} \big\{|\mathbf{u}_{n}(s^{-}) +  P_{n}F(s, \mathbf{u}_{n}(s^{-}); y)|^2 - |\mathbf{u}_{n}(s^{-})|^2 \big\} \tilde{\eta}(ds, dy)\nonumber
\\ &\qquad\quad \ + \int_{0}^{t} \int_{Y} \bigg\{ |\mathbf{u}_{n}(s^{-}) +  P_{n}F(s, \mathbf{u}_{n}(s^{-}); y)|^2 - |\mathbf{u}_{n}(s^{-})|^2 \nonumber
\\ & \ \qquad \qquad \qquad \qquad \qquad \qquad - 2 \big\langle \mathbf{u}_{n}(s^{-}),  P_{n}F(s, \mathbf{u}_{n}(s^{-}); y)\big\rangle \bigg\} \, d\nu(y),
\end{align}
where $\psi(s) := \Psi(\mathbf{d}_{n}(s)) + \big| \mathbf{u}_{n}(s) \big|^{2}_{\mathbb{H}}.$

Now raising both side to the power $p \geq 1$, taking supremum over $s \in [0, t \wedge \tau_{n}^{R}]$ and taking expectation we have
\begin{align}\label{psiij}
\mathbb{E} \bigg[ &\sup_{s \in [0, t \wedge \tau_{n}^{R} ]} [\psi(s)]^{p}\bigg] + \mathbb{E} \bigg[ \int_{0}^{t \wedge \tau_{n}^{R}} \big( \|\mathbf{u}_{n}(s) \|^2 + |\Delta \mathbf{d}_{n}(s)- f_{n}(\mathbf{d}_{n}(s))|^2 \big) \,ds \bigg]^{p} \nonumber
\\ &\leq \mathbb{E} [\psi(0)]^{p} + c \ \mathbb{E} \bigg[ \sup_{s \in [0, t \wedge \tau_{n}^{R}]} |I_n(s)|^p \bigg]+ c \ \mathbb{E} \bigg[ \sup_{s \in [0, t \wedge \tau_{n}^{R}]} |J_n(s)|^p \bigg],
\end{align}
where $I_n(s)$ and $J_n(s)$ are defined as \eqref{intt} and \eqref{jntt} respectively.

Next we observe by the Burkholder-Davis-Gundy inequality,
\begin{align}\label{bdgi}
&\mathbb{E} \bigg[ \sup_{s \in [0, t \wedge \tau_{n}^{R} ]} | I_n(s) |^p \bigg] \nonumber
\\ &\leq \tilde{C}_p \mathbb{E} \bigg[ \bigg( \int_{0}^{t \wedge \tau^{R}_{n}} \int_{Y} \big\{ |\mathbf{u}_{n}(s^{-}) +  P_{n}F(s, \mathbf{u}_{n}(s^{-}); y)|^2 - |\mathbf{u}_{n}(s^{-})|^2 \big\}^{2} \,\nu(dy)ds \bigg)^{\frac{p}{2}}\bigg].
\end{align}

For each $r \geq 2$, using Taylor's formula, it follows that there exists a positive constant $c_r > 0$ such that for all $a, b \in \mathbb{H}$ we have
\begin{align}
\big| |a+b|^{r}_{\mathbb{H}} - |a|^{r}_{\mathbb{H}} - r|a|^{r-2}_{\mathbb{H}} \langle a, b\rangle\big| \leq c_{r} \big( |a|^{r-2}_{\mathbb{H}} + |b|^{r-2}_{\mathbb{H}}\big) |b|^{2}_{\mathbb{H}}. 
\end{align}

Using H\"older's inequality we further have
\begin{align}\label{taf}
\big( |a+b|^{r}_{\mathbb{H}} - |a|^{r}_{\mathbb{H}}\big)^{2} &\leq 2 \big\{ r^2 |a|^{2r-2}_{\mathbb{H}}|b|^{2}_{\mathbb{H}} + c_{r}^{2} \big( |a|^{r-2}_{\mathbb{H}} + |b|^{r-2}_{\mathbb{H}}\big)^2 |b|^{4}_{\mathbb{H}} \big\} \nonumber
\\ &\leq  2 r^2 |a|^{2r-2}_{\mathbb{H}}|b|^{2}_{\mathbb{H}} + 4 c_{r}^{2} |a|^{2r-4}_{\mathbb{H}}  |b|^{4}_{\mathbb{H}} + 4 c_{r}^{2} |b|^{2r}_{\mathbb{H}}.
\end{align}

Now for $r = 2$ in \eqref{taf}, using \eqref{lgf} and Young's inequality
\begin{align}
\int_{Y} &\big\{ |\mathbf{u}_{n}(s) +  P_{n}F(s, \mathbf{u}_{n}(s); y)|^2 - |\mathbf{u}_{n}(s)|^2 \big\}^{2} \,\nu(dy) \nonumber
\\ &\leq c \ |\mathbf{u}_{n}(s)|^2 \int_{Y} |F(s, \mathbf{u}_{n}(s); y)|^2 \,\nu(dy) + c \int_{Y}| F(s, \mathbf{u}_{n}(s); y)|^4 \,\nu(dy) \nonumber
\\ &\leq c + c_1 |\mathbf{u}_{n}(s)|^2 + c_2  |\mathbf{u}_{n}(s)|^4 \leq k_1 + k_2 \ |\mathbf{u}_{n}(s)|^4. 
\end{align}

Hence
\begin{align}
\bigg(\int_{0}^{t \wedge \tau^{R}_{n}} &\int_{Y} \big\{ |\mathbf{u}_{n}(s) +  P_{n}F(s, \mathbf{u}_{n}(s); y)|^2 - |\mathbf{u}_{n}(s)|^2 \big\}^{2} \,\nu(dy) ds \bigg)^{\frac{p}{2}} \nonumber
\\ &\leq c (k_1 T)^{\frac{p}{2}} + c (k_2)^{\frac{p}{2}} \bigg( \int_{0}^{t \wedge \tau^{R}_{n}} |\mathbf{u}_{n}(s)|^4 \,ds \bigg)^{\frac{p}{2}}. 
\end{align}

So we have
\begin{align}\label{inspe}
&c \ \mathbb{E} \bigg[ \sup_{s \in [0, t \wedge \tau_{n}^{R} ]} | I_n(s) |^p \bigg] \leq \tilde{C}_p c (k_1 T)^{\frac{p}{2}} + \tilde{C}_p c (k_2)^{\frac{p}{2}} \mathbb{E} \bigg[ \bigg( \int_{0}^{t \wedge \tau^{R}_{n}} |\mathbf{u}_{n}(s)|^4 \, ds\bigg)^{\frac{p}{2}} \bigg] \nonumber
\\ &\ \ \ \leq C + C \ \mathbb{E} \bigg[ \bigg( \sup_{s \in [0, t \wedge \tau_{n}^{R} ]} |\mathbf{u}_{n}(s)|^2 \bigg)^{\frac{p}{2}} \bigg( \int_{0}^{t \wedge \tau^{R}_{n}} |\mathbf{u}_{n}(s)|^2 \, ds \bigg)^{\frac{p}{2}}\bigg] \nonumber
\\ &\ \ \ \leq C + \frac{1}{2} \mathbb{E} \bigg[ \bigg( \sup_{s \in [0, t \wedge \tau_{n}^{R} ]} |\mathbf{u}_{n}(s)|^2 \bigg)^{p} \bigg] + \frac{C^2}{2} \mathbb{E} \bigg[ \bigg( \int_{0}^{t \wedge \tau^{R}_{n}} |\mathbf{u}_{n}(s)|^2 \, ds \bigg)^{p} \bigg]. 
\end{align}

Using H\"older's inequality we observe
\begin{align}
\int_{0}^{t \wedge \tau^{R}_{n}} |\mathbf{u}_{n}(s)|^2 \, ds &\leq \bigg( \int_{0}^{t \wedge \tau^{R}_{n}} \big( |\mathbf{u}_{n}(s)|^2 \big)^p \, ds \bigg)^{\frac{1}{p}} \times \bigg( \int_{0}^{t \wedge \tau^{R}_{n}} 1^{\frac{p}{p-1}} \,ds \bigg)^{\frac{p-1}{p}} \nonumber
\\ &\leq T^{\frac{p-1}{p}} \bigg( \int_{0}^{t \wedge \tau^{R}_{n}} |\mathbf{u}_{n}(s)|^{2p} \, ds \bigg)^{\frac{1}{p}}.
\end{align}

Now raising power $p$ in both side, taking expectation, then using Fubini's Theorem
\begin{align}\label{unpp} 
\mathbb{E} \bigg[ \bigg( \int_{0}^{t \wedge \tau^{R}_{n}} |\mathbf{u}_{n}(s)|^2 \, ds \bigg)^{p} \bigg] \leq T^{p-1} \int_{0}^{t \wedge \tau^{R}_{n}} \mathbb{E} \big[ |\mathbf{u}_{n}(s)|^{2p} \big] \, ds.
\end{align}

From \eqref{inspe} and \eqref{unpp} we get
\begin{align}\label{iun}
c \ &\mathbb{E} \bigg[ \sup_{s \in [0, t \wedge \tau_{n}^{R} ]} | I_n(s) |^p \bigg] \leq \frac{1}{2} \mathbb{E} \bigg[ \bigg( \sup_{s \in [0, t \wedge \tau_{n}^{R} ]} |\mathbf{u}_{n}(s)|^2 \bigg)^{p} \bigg] + C_{p, T} \int_{0}^{t \wedge \tau^{R}_{n}} \mathbb{E} \big[ |\mathbf{u}_{n}(s)|^{2p} \big] \, ds \nonumber
\\ &\leq \frac{1}{2} \mathbb{E} \bigg[ \bigg( \sup_{s \in [0, t \wedge \tau_{n}^{R} ]} |\mathbf{u}_{n}(s)|^2 \bigg)^{p} \bigg] + C_{p, T} \int_{0}^{t \wedge \tau^{R}_{n}} \mathbb{E} \big[ [\Psi(\mathbf{d}_{n}(s))]^{p} + |\mathbf{u}_{n}(s)|^{2p} \big] \, ds.
\end{align}

From \eqref{jnes} we have
\begin{align}
\mathbb{E} &|J_n(t \wedge \tau^{R}_{n})|^p \leq C \ \mathbb{E} \bigg[ \bigg( \int_{0}^{t \wedge \tau^{R}_{n}} \big\{ 1 + |\mathbf{u}_{n}(s)|^2 \big\} \, ds \bigg)^p \bigg] \nonumber
\\ &= C \ \mathbb{E} \bigg[ \bigg( (t \wedge \tau^{R}_{n}) + \int_{0}^{t \wedge \tau^{R}_{n}} |\mathbf{u}_{n}(s)|^2 \, ds \bigg)^p \bigg] \leq C_p \ \mathbb{E} \bigg[ (t \wedge \tau^{R}_{n})^p + \bigg( \int_{0}^{t \wedge \tau^{R}_{n}} |\mathbf{u}_{n}(s)|^2 \, ds \bigg)^p \bigg] \nonumber
\\ &\leq C_p \bigg[ C_{p,T} + \mathbb{E} \bigg( \int_{0}^{t \wedge \tau^{R}_{n}} |\mathbf{u}_{n}(s)|^2 \, ds \bigg)^p \bigg]. 
\end{align}

From\eqref{unpp} we obtain
\begin{align}\label{jun}
\mathbb{E}\bigg[ \sup_{s \in [0, t \wedge \tau^{R}_{n}]} |J_n(s)|^p\bigg] &\leq C_{p, T} + C_{p, T} \int_{0}^{t \wedge \tau^{R}_{n}} \mathbb{E} \big[ |\mathbf{u}_{n}(s)|^{2p} \big] \, ds \nonumber
\\ &\leq C_{p, T} + C_{p, T} \int_{0}^{t \wedge \tau^{R}_{n}} \mathbb{E} \big[ [\Psi(\mathbf{d}_{n}(s))]^{p} + |\mathbf{u}_{n}(s)|^{2p} \big] \, ds.
\end{align}

Now from \eqref{psiij}, \eqref{iun} and \eqref{jun} we obtain 
\begin{align}
 \mathbb{E} \bigg[ \sup_{s \in [0, t \wedge \tau_{n}^{R} ]} [\psi(s)]^{p}\bigg] &\leq C_{p, T} + \frac{1}{2} \mathbb{E} \bigg[ \bigg( \sup_{s \in [0, t \wedge \tau_{n}^{R} ]} |\mathbf{u}_{n}(s)|^2 \bigg)^{p} \bigg] \nonumber
\\ &\quad+ C_{p, T} \int_{0}^{t \wedge \tau^{R}_{n}} \mathbb{E} \big[ \Psi(\mathbf{d}_{n}(s))]^{p} + |\mathbf{u}_{n}(s)|^{2p} \big] \, ds.
\end{align}

As $[\psi(s)]^p > [\Psi(\mathbf{d}_{n}(s))]^{p} + |\mathbf{u}_{n}(s)|^{2p},$ we further observe
\begin{align}
&\mathbb{E} \bigg[ \sup_{s \in [0, t \wedge \tau_{n}^{R} ]} \bigg( [\Psi(\mathbf{d}_{n}(s))]^{p} + |\mathbf{u}_{n}(s)|^{2p} \bigg) \bigg] \leq \mathbb{E} \bigg[ \sup_{s \in [0, t \wedge \tau_{n}^{R} ]} [\psi(s)]^{p}\bigg] \nonumber
\\ &\leq C_{p, T} + \frac{1}{2} \mathbb{E} \bigg[ \sup_{s \in [0, t \wedge \tau_{n}^{R} ]} |\mathbf{u}_{n}(s)|^{2p} \bigg] + C_{p, T} \int_{0}^{t \wedge \tau^{R}_{n}} \mathbb{E} \Big[ [\Psi(\mathbf{d}_{n}(s))]^{p} + |\mathbf{u}_{n}(s)|^{2p} \Big] \, ds.
\end{align}

Now taking second term of Right Hand Side to the Left Hand Side and multiplying both side by 2 we have, 
\begin{align}
\mathbb{E} \bigg[\sup_{s \in [0, t \wedge \tau_{n}^{R} ]}&\bigg( [\Psi(\mathbf{d}_{n}(s))]^{p} + |\mathbf{u}_{n}(s)|^{2p} \bigg) \bigg] \leq C_{p, T} + C_{p, T} \int_{0}^{t \wedge \tau^{R}_{n}} \mathbb{E} \Big[ [\Psi(\mathbf{d}_{n}(s))]^{p} + |\mathbf{u}_{n}(s)|^{2p} \Big] \, ds.
\end{align}

Now Gronwall's Lemma yields
\begin{align}
\mathbb{E} \bigg[\sup_{s \in [0, t \wedge \tau_{n}^{R} ]}&\bigg( [\Psi(\mathbf{d}_{n}(s))]^{p} + |\mathbf{u}_{n}(s)|^{2p} \bigg) \bigg] \leq C_{p, T} \ \exp\big(C_{p, T} (T \wedge \tau_{n}^{R})\big),
\end{align}

As the constant in the Right Hand Side is independent of $n \in \mathbb{N}$ and $R > 0,$ passing to the limit as $R \to \infty,$ we obtain
\begin{align}
\sup_{n \in \mathbb{N}} \mathbb{E} \bigg[\sup_{s \in [0, T]}&\bigg( [\Psi(\mathbf{d}_{n}(s))]^{p} + \big|\mathbf{u}_{n}(s) \big|_{\mathbb{H}}^{2p} \bigg) \bigg] \leq C(p, T).
\end{align}

Finally from \eqref{psiij}, using the previous calculations and letting $R \to \infty$ we obtain,
\begin{align}\label{pufe}
\sup_{n \in \mathbb{N}} &\bigg(\mathbb{E} \bigg[ \sup_{s \in [0, T]} \big(\Psi(\mathbf{d}_{n}(s)) + \big| \mathbf{u}_{n}(s) \big|^{2}_{\mathbb{H}} \big)^{p}\bigg] \nonumber
\\ &+ \mathbb{E} \bigg[ \int_{0}^{T} \big( \|\mathbf{u}_{n}(s) \|^2 + |\Delta \mathbf{d}_{n}(s)- f_{n}(\mathbf{d}_{n}(s))|^2 \big) \,ds \bigg]^{p}\bigg) \leq C_{p, T},
\end{align}
where $\Psi(\mathbf{d}_{n}(s)) := \frac{1}{2} \|\mathbf{d}_{n}(s)\|^2 + \frac{1}{2} \int_{\mathbb{O}} \mathrm{F}(|\mathbf{d}_{n}(s)|^2) \, dx.$
\end{proof}

\noindent Now we will deduce an estimate for $\Delta \mathbf{d}_n.$

\begin{proposition}\label{Ldn}
For every $q \geq 1,$ there exists a positive constant $C$, depending on $q$ such that
\begin{align*}
\mathbb{E} \bigg[ \int_{0}^{T} |\Delta \mathbf{d}_n(s)|^2  \,ds \bigg]^q \leq C.
\end{align*}
\end{proposition}
\begin{proof}
For $N \in I_{\mathbf{n}}, H^1 \hookrightarrow L^{4N+2}$ and using Remark \ref{fes}
\begin{align}
|\Delta \mathbf{d}_n(s)|^2 &= |\Delta \mathbf{d}_n(s) - f_n(\mathbf{d}_n(s)) + f_n(\mathbf{d}_n(s))|^{2} \nonumber
\\ &\leq 2 |\Delta \mathbf{d}_n(s) - f_n(\mathbf{d}_n(s))|^2 + 2  |f_n(\mathbf{d}_n(s))|^{2} \nonumber
\\ &\leq 2 |\Delta \mathbf{d}_n(s) - f_n(\mathbf{d}_n(s))|^2 + c \|\mathbf{d}_n(s)\|^{\bar{q}}_{L^{\bar{q}}} + C,
\end{align}
where $\bar{q} = 4N+2.$ Then for some $q \geq 1$ we have
\begin{align}
\mathbb{E} \bigg[ \int_{0}^{T} |\Delta \mathbf{d}_n(s)|^2  \,ds \bigg]^q \leq c \ \mathbb{E} \bigg[ &\int_{0}^{T} |\Delta \mathbf{d}_n(s) - f_n(\mathbf{d}_n(s))|^{2} \,ds \bigg]^q + c \ \mathbb{E} \bigg[ \sup_{s \in [0, T]} \|\mathbf{d}_n(s)\|^{(4N+2)q} \bigg] + C.
\end{align}

Now from \eqref{pufe} we obtain
\begin{align}\label{ldn}
\mathbb{E} \bigg[ \int_{0}^{T} |\Delta \mathbf{d}_n(s)|^2  \,ds \bigg]^q \leq C(q).
\end{align}
\end{proof}

\begin{corollary}\label{codnh}
For every $q \geq 2,$ there exists a constant $C>0,$ depending on $q$ such that
\begin{align*}
\mathbb{E} \bigg[ \sup_{t \in [0, T]} \big\| \mathbf{d}_n(t)\big\|^{q}_{H^1} \bigg] \leq C_q.
\end{align*}
\end{corollary}
\begin{proof}
For suitable choices of $p$, from Proposition \ref{dgrdn} and Proposition \ref{ugrdn} we perceive for $q \geq 2,$
\begin{align*}
&\mathbb{E} \bigg[ \sup_{t \in [0, T]} \big\| \mathbf{d}_n(t)\big\|^{q}_{H^1} \bigg] \leq \mathbb{E} \bigg[ \bigg( \sup_{t \in [0, T]} \big\| \mathbf{d}_n(t)\big\|^{2}_{H^1} \bigg)^{\frac{q}{2}} \bigg] \nonumber
\\ &\leq c \ \mathbb{E} \bigg[ \bigg\{ \sup_{t \in [0, T]} \big( \big| \mathbf{d}_n(t)\big|^{2}_{L^2} + \| \mathbf{d}_n(t) \|^{2} \big) \bigg\}^{\frac{q}{2}} \bigg] \leq c_q \ \mathbb{E} \bigg[ \sup_{t \in [0, T]} \big| \mathbf{d}_n(t)\big|^{q}_{L^2} + \sup_{t \in [0, T]} \| \mathbf{d}_n(t) \|^{q} \bigg] \leq C_q.
\end{align*}
\end{proof}

\begin{remark}
Since a stochastic integral with respect to the time homogeneous compensated Poisson random measure is defined for all progressively measurable processes, each $\mathbf{u}_n$ satisfies a version of \eqref{g1st} with $t^-$ replaced by $t$, i.e. $\mathbf{u}_n$ satisfies
\begin{align}\label{g1st-}
d\mathbf{u}_{n}(t) + \big[ \mathscr{A}\mathbf{u}_{n}(t) &+ B_{n}(\mathbf{u}_{n}(t)) + M_{n}(\mathbf{d}_{n}(t)) \big] dt  =  \int_{Y} P_{n}F(t, \mathbf{u}_{n}(t); y) \tilde{\eta}(dt, dy).
\end{align}
\end{remark}

\section{Tightness of the Laws of Approximating Sequences}\label{tolas}

 In this section we will show that all the conditons of Corollary \ref{tc1} and Corollary \ref{tc2} satisfiy for $p = 2$. This will yield tightness of laws of $\mathbf{u}_{n}$ and $\mathbf{d}_{n}$. Let us consider the space $\mathcal{Z}_{T} = \mathcal{Z}_{T, 1} \times \mathcal{Z}_{T, 2},$ where
\begin{align}
\mathcal{Z}_{T, 1} = L^{2}_{w}(0, T; \mathbb{V}) \cap L^{2}(0, T; \mathbb{H}) \cap \mathbb{D}([0, T]; \mathbb{V}') \cap \mathbb{D}([0, T]; \mathbb{H}_{w})
\end{align}
and
\begin{align}
\mathcal{Z}_{T, 2} = L^{2}_{w}(0, T; H^{2}) \cap L^{2}(0, T; H^{1}) \cap \mathbb{C}([0, T]; (H^2)') \cap \mathbb{C}([0, T]; H^{1}_{w}).
\end{align}

For each $n \in \mathbb{N},$ the solutions $(\mathbf{u}_{n}, \mathbf{d}_{n})$ of the Galerkin approximation equations define measures $\mathscr{L}(\mathbf{u}_{n}, \mathbf{d}_{n})$ on $(\mathcal{Z}_T, \mathscr{T}),$ where $\mathscr{T}$ is the supremum of $\mathscr{T}^1$ and $\mathscr{T}^2.$ We will show that the set of measures $\{ \mathscr{L}(\mathbf{u}_{n}, \mathbf{d}_{n}), n \in \mathbb{N}\}$ is tight on $(\mathcal{Z}_T, \mathscr{T})$ using Corollary \ref{tc1} and Corollary \ref{tc2}. 

Before embarking on to the main result of this section let us write-down some important estimates we derived in Section \ref{EE}.

 From Proposition \ref{dgrdn} we observe that for $p \geq 2,$
\begin{align}\label{ssdnp}
\sup_{n \in \mathbb{N}} \ \mathbb{E} \bigg[ \sup_{t \in [0, T]}  \big|\mathbf{d}_{n}(t) \big|_{L^2}^p \bigg] \leq \tilde{C}_{p, T},
\end{align}
for some constant $\tilde{C}_{p, T},$ independent of $t \in [0, T],$ $R > 0$ and $n \in \mathbb{N}.$ 
 In particular from \eqref{ssdnp},
\begin{align}\label{dnc}
\sup_{n \in \mathbb{N}} \mathbb{E} \bigg[ \int_{0}^{T} |\mathbf{d}_{n}(s)|^p \,ds\bigg] \leq \tilde{C}_{p, T}.
\end{align}

Again from \eqref{c0} we have
\begin{align}\label{gdnc}
\sup_{n \in \mathbb{N}} \mathbb{E} \bigg[ \int_{0}^{T} |\mathbf{d}_n(s)|^{p-2} \| \mathbf{d}_{n}(s)\|^2 \,ds\bigg] \leq C_{T, p}.
\end{align}

From equations \eqref{dnc} and \eqref{gdnc}, putting $p = 2$, we get
\begin{align}\label{dnv2}
\sup_{n \in \mathbb{N}} \mathbb{E} \bigg[ \int_{0}^{T} \big\| \mathbf{d}_{n}(s) \big\|_{H^1}^2 \,ds \bigg] \leq C_T.
\end{align}

Since $L^2 \big(\Omega; L^{\infty}(0, T; H^1) \big) \subset L^1 \big(\Omega; L^{\infty}(0, T; H^1) \big),$ From Corollary \ref{codnh} we get,
\begin{align}\label{dnH1}
\sup_{n \in \mathbb{N}} \ \mathbb{E} \bigg[ \sup_{t \in [0, T]} \big\|\mathbf{d}_{n}(t) \big\|_{H^1} \bigg] \leq C_1.
\end{align}

Now from Corollary \ref{Ldn} for $q=2$ and \eqref{dnv2} we obtain,
\begin{align}\label{dnH2}
\sup_{n \in \mathbb{N}} \mathbb{E} \bigg[ \int_{0}^{T} \big\| \mathbf{d}_{n}(s) \big\|_{H^2}^2 \,ds \bigg] \leq C(T).
\end{align}
 
From Proposition \ref{ugrdn} we obtain,
\begin{align}\label{u1tip}
\sup_{n \in \mathbb{N}} \ \mathbb{E} \bigg[\sup_{t \in [0, T]}\big|\mathbf{u}_{n}(t)\big|^{2p}_{\mathbb{H}} \bigg] \leq C_{p, T},
\end{align}

and for $p=1$ in the above case we get,
\begin{align}\label{u1ti}
\sup_{n \in \mathbb{N}} \ \mathbb{E} \bigg[\sup_{t \in [0, T]}\big|\mathbf{u}_{n}(t)\big|^2_{\mathbb{H}} \bigg] \leq C_{1, T},
\end{align}

Using similar argument as of \eqref{dnH1} we obtain,
\begin{align}\label{tiu1}
\sup_{n \in \mathbb{N}} \ \mathbb{E} \bigg[\sup_{t \in [0, T]} |\mathbf{u}_{n}(t)|_{\mathbb{H}} \bigg] \leq \tilde{C}_T,
\end{align}

From \eqref{u1ti} in particular we have,
\begin{align}\label{unh}
\sup_{n \in \mathbb{N}} \mathbb{E} \bigg[ \int_{0}^{T} \big| \mathbf{u}_{n}(s) \big|_{\mathbb{H}}^2 \,ds \bigg] \leq C'_T.
\end{align}

Again from Proposition \ref{ugrdn} for $p=1$ we get, 
\begin{align}\label{ung}
\sup_{n \in \mathbb{N}} \mathbb{E} \bigg[ \int_{0}^{T} \| \mathbf{u}_{n}(s) \|^2 \,ds \bigg] \leq C_{1, T}.
\end{align}

So from \eqref{unh} and \eqref{ung} we have
\begin{align}\label{unv}
\sup_{n \in \mathbb{N}} \mathbb{E} \bigg[ \int_{0}^{T} \big\| \mathbf{u}_{n}(s) \big\|_{\mathbb{V}}^2 \,ds \bigg] \leq \tilde{C}(T).
\end{align}

Now we can state and prove the tightness Lemma.

\begin{lemma}\label{tig}
The set of measures $\{ \mathscr{L}(\mathbf{u}_{n}, \mathbf{d}_{n}), n \in \mathbb{N}\}$ is tight on $(\mathcal{Z}_T, \mathscr{T}).$
\end{lemma}
\begin{proof}

From \eqref{tiu1}, \eqref{unv}, \eqref{dnH1} and \eqref{dnH2} we obtain the first two conditions of Corollary \ref{tc1} and Corollary \ref{tc2} for $\mathbf{u}_{n}$ and $\mathbf{d}_{n}$ respectively.

Now it is sufficient to prove that the sequences $(\mathbf{u}_{n})_{n \in \mathbb{N}}$ and $(\mathbf{d}_{n})_{n \in \mathbb{N}}$ satisfy the Aldous condtion in the space $\mathbb{V}'$ and $(H^2)'$ respectively for tightness. We will use Lemma \ref{aco}. Let $(\tau_{n})_{n \in \mathbb{N}}$ be a sequence of stopping times such that $0 \leq \tau_{n} \leq T.$ From \eqref{g1st-} we have
\begin{align*}
\mathbf{u}_{n}(t) &= \mathbf{u}_{0n} - \int_{0}^{t} \mathscr{A} \mathbf{u}_{n}(s) \,ds - \int_{0}^{t} B_n(\mathbf{u}_{n}(s)) \,ds - \int_{0}^{t} M_n(\mathbf{d}_{n}(s)) \,ds 
\\ &\qquad\qquad\qquad\qquad\qquad+ \int_{0}^{t} \int_{Y} P_n F(s, \mathbf{u}_{n}(s), y) \,\tilde{\eta}(ds, dy)
\\ &=: k_{1}^{n} + \sum_{j = 2}^{5} k_{j}^{n}(t) , \qquad t \in [0, T]. 
\end{align*}

Let $\theta > 0.$ We will check that each term $k_{j}^{n}, j = 1, \cdots, 5,$ satisfies condition \eqref{aco1} in Lemma \ref{aco}.\\
It is easy to see that $k_{1}^{n}$ satifies condition \eqref{aco1} with $\alpha = 1$ and $\beta = 1.$

Now consider $k_{2}^{n}(t)$. Since $\mathscr{A} : \mathbb{V} \to \mathbb{V}'$ and $|\mathscr{A}(\mathbf{u})|_{\mathbb{V}'} \leq \| \mathbf{u} \|$, by the H\"older inequality and \eqref{ung} we have
\begin{align*}
&\mathbb{E}\big[ \big|k_{2}^{n}(\tau_n + \theta) - k_{2}^{n}(\tau_n) \big|_{\mathbb{V}'} \big] = \mathbb{E}\bigg[ \bigg| \int_{\tau_n}^{\tau_n + \theta} \mathscr{A} \mathbf{u}_{n}(s) \,ds \bigg|_{\mathbb{V}'} \bigg] \leq c \ \mathbb{E}\bigg[ \int_{\tau_n}^{\tau_n + \theta} \big|\mathscr{A} \mathbf{u}_{n}(s)\big|_{\mathbb{V}'} \,ds \bigg] 
\\ &\leq c \ \mathbb{E}\bigg[ \int_{\tau_n}^{\tau_n + \theta} \| \mathbf{u}_{n}(s)\| \, ds \bigg] \leq c \ \mathbb{E}\bigg[ \theta^{\frac{1}{2}} \bigg( \int_{0}^{T} \| \mathbf{u}_{n}(s)\|^2 \,ds \bigg)^{\frac{1}{2}} \bigg] \leq c \cdot \bigg( \mathbb{E}\bigg[ \int_{0}^{T} \| \mathbf{u}_{n}(s)\|^2 \,ds \bigg]\bigg)^{\frac{1}{2}} \cdot \theta^{\frac{1}{2}} 
\\ &\leq c \sqrt{C_{1, T}} \cdot \theta^{\frac{1}{2}} = c_2 \cdot \theta^{\frac{1}{2}}.
\end{align*}
Thus $ k_{2}^{n}$ satisfies condition \eqref{aco1} with $\alpha = 1$ and $\beta = \frac{1}{2}.$

Let us consider the term  $k_{3}^{n}.$ By \eqref{BHH} and \eqref{u1ti} we have
\begin{align*} 
&\mathbb{E}\big[ \big|k_{3}^{n}(\tau_n + \theta) - k_{3}^{n}(\tau_n) \big|_{\mathbb{V}'} \big] = \mathbb{E}\bigg[ \bigg| \int_{\tau_n}^{\tau_n + \theta} B_n(\mathbf{u}_{n}(s)) \,ds \bigg|_{\mathbb{V}'} \bigg] \leq c \ \mathbb{E}\bigg[ \int_{\tau_n}^{\tau_n + \theta} \big| B(\mathbf{u}_{n}(s))\big|_{\mathbb{V}'} \,ds \bigg] 
\\ &\leq c \ \mathbb{E}\bigg[ \int_{\tau_n}^{\tau_n + \theta} \big|\mathbf{u}_{n}(s)\big|^{2}_{\mathbb{H}} \,ds \bigg] \leq c \ \mathbb{E}\bigg[ \sup_{s \in [0, T]} \big|\mathbf{u}_{n}(s)\big|^{2}_{\mathbb{H}} \bigg] \cdot \theta \leq c \ C_{1, T} \cdot \theta =: c_3 \cdot \theta.
\end{align*}
Thus $ k_{3}^{n}$ satisfies condition \eqref{aco1} with $\alpha = 1$ and $\beta = 1.$

Now consider $k_{4}^{n}.$ From \eqref{md1d2} we have
\begin{align}\label{k4} 
&\mathbb{E}\big[ \big|k_{4}^{n}(\tau_n + \theta) - k_{4}^{n}(\tau_n) \big|_{\mathbb{V}'} \big] = \mathbb{E}\bigg[ \bigg| \int_{\tau_n}^{\tau_n + \theta} M_n(\mathbf{d}_{n}(s)) \,ds \bigg|_{\mathbb{V}'} \bigg] \nonumber
\\ &\leq c \ \mathbb{E}\bigg[ \int_{\tau_n}^{\tau_n + \theta} \big| M(\mathbf{d}_{n}(s))\big|_{\mathbb{V}'} \,ds \bigg] \leq c \ \mathbb{E}\bigg[ \int_{\tau_n}^{\tau_n + \theta} \|\mathbf{d}_{n}(s)\|^{2 - \frac{\mathbf{n}}{2}} |\Delta \mathbf{d}_{n}(s)|^{\frac{\mathbf{n}}{2}} \,ds \bigg]
\end{align}

We will estimate differently for $\mathbf{n} = 2 \text{ and } 3.$ 

First consider $\mathbf{n} = 2$. From Proposition \ref{Ldn}, Corollary \ref{codnh} and using H\"older's inequality we obtain,
\begin{align}\label{gdld} 
&\mathbb{E}\bigg[ \int_{\tau_n}^{\tau_n + \theta} \|\mathbf{d}_{n}(s)\| \cdot |\Delta \mathbf{d}_{n}(s)| \,ds \bigg] \leq \mathbb{E}\bigg[ \bigg( \int_{\tau_n}^{\tau_n + \theta} \|\mathbf{d}_{n}(s)\|^2 \,ds \bigg)^{\frac{1}{2}} \bigg( \int_{\tau_n}^{\tau_n + \theta} |\Delta \mathbf{d}_{n}(s)|^2 \,ds\bigg)^{\frac{1}{2}} \bigg] \nonumber
\\ &\leq \bigg\{ \mathbb{E}\bigg[ \sup_{s \in [0, T]}\|\mathbf{d}_{n}(s)\|^2 \cdot \theta \bigg] \bigg\}^{\frac{1}{2}} \cdot \bigg\{ \mathbb{E}\bigg[ \int_{\tau_n}^{\tau_n + \theta} |\Delta \mathbf{d}_{n}(s)|^2 \,ds \bigg] \bigg\}^{\frac{1}{2}} \nonumber
\\ &\leq \sqrt{C_2} \cdot \theta^{\frac{1}{2}} \ \bigg\{\mathbb{E}\bigg[ \int_{0}^{T} |\Delta \mathbf{d}_{n}(s)|^2 \,ds \bigg]\bigg\}^{\frac{1}{2}} \leq \sqrt{C_2} \cdot \theta^{\frac{1}{2}} \cdot \sqrt{C(1)} =: c_4 \cdot \theta^{\frac{1}{2}}.
\end{align}

Now for $\mathbf{n} = 3,$ from  Proposition \ref{Ldn}, Corollary \ref{codnh} and using H\"older's inequality we obtain,
\begin{align}\label{gdld3} 
&\mathbb{E}\bigg[ \int_{\tau_n}^{\tau_n + \theta} \|\mathbf{d}_{n}(s)\|^{\frac{1}{2}} \cdot |\Delta \mathbf{d}_{n}(s)|^{\frac{3}{2}} \,ds \bigg] \leq \mathbb{E}\bigg[ \bigg( \int_{\tau_n}^{\tau_n + \theta} \big(\|\mathbf{d}_{n}(s)\|^{\frac{1}{2}}\big)^4 \, ds \bigg)^{\frac{1}{4}} \bigg( \int_{\tau_n}^{\tau_n + \theta} \big(|\Delta \mathbf{d}_{n}(s)|^{\frac{3}{2}}\big)^{\frac{4}{3}} \,ds\bigg)^{\frac{3}{4}} \bigg] \nonumber
\\ &\leq  \mathbb{E}\bigg[ \bigg( \sup_{s \in [0, T]} \|\mathbf{d}_{n}(s)\|^{2} \cdot \theta \bigg)^{\frac{1}{4}} \bigg( \int_{\tau_n}^{\tau_n + \theta} |\Delta \mathbf{d}_{n}(s)|^{2} \,ds \bigg)^{\frac{3}{4}}\bigg] \nonumber
\\ &\leq \theta^{\frac{1}{4}} \cdot \bigg\{ \mathbb{E}\bigg[ \sup_{s \in [0, T]} \|\mathbf{d}_{n}(s)\|^{2} \bigg] \bigg\}^{\frac{1}{4}} \cdot \bigg\{ \mathbb{E}\bigg[ \int_{\tau_n}^{\tau_n + \theta} |\Delta \mathbf{d}_{n}(s)|^{2} \,ds \bigg] \bigg\}^{\frac{3}{4}} \nonumber
\\ &\leq \theta^{\frac{1}{4}}  (C_2)^{\frac{1}{4}} \ \bigg\{ \mathbb{E}\bigg[ \int_{0}^{T} |\Delta \mathbf{d}_{n}(s)|^{2} \,ds \bigg] \bigg\}^{\frac{3}{4}} \leq \theta^{\frac{1}{4}}  (C_2)^{\frac{1}{4}} (C(1))^{\frac{3}{4}} =: c_4 \cdot \theta^{\frac{1}{4}}.
\end{align}

where $C_2$ and $C(1)$ are constants used in Corollary \ref{codnh} and  Proposition \ref{Ldn} respectively. Thus  $k_{4}^{n}$  satisfies condition \eqref{aco1} with $\alpha = 1$, $\beta = \frac{1}{2}$ for 2-D and $\alpha = 1$, $\beta = \frac{1}{4}$ for 3-D.

Now consider $k_{5}^{n}.$ Since the embedding $\mathbb{H} \hookrightarrow \mathbb{V}'$ is continuous, from \eqref{isof}, \eqref{lgf}, \eqref{u1ti} and using Burkholder-Davis-Gundy inequality
\begin{align} 
\mathbb{E}&\big[ \big|k_{5}^{n}(\tau_n + \theta) - k_{5}^{n}(\tau_n) \big|^{2}_{\mathbb{V}'} \big] = \mathbb{E}\bigg[ \bigg| \int_{\tau_n}^{\tau_n + \theta} \int_{Y}  P_n F(s, \mathbf{u}_{n}(s), y) \,\tilde{\eta}(ds, dy) \bigg|^{2}_{\mathbb{V}'} \bigg] \nonumber
\\ &\leq c \ \mathbb{E}\bigg[ \bigg| \int_{\tau_n}^{\tau_n + \theta} \int_{Y}  P_n F(s, \mathbf{u}_{n}(s), y) \,\tilde{\eta}(ds, dy) \bigg|^{2}_{\mathbb{H}} \bigg] = c \ \mathbb{E}\bigg[  \int_{\tau_n}^{\tau_n + \theta} \int_{Y} \big| P_n F(s, \mathbf{u}_{n}(s), y)\big|^{2}_{\mathbb{H}} \,\nu(dy)ds  \bigg] \nonumber
\\ &\leq c \ \mathbb{E}\bigg[ \int_{\tau_n}^{\tau_n + \theta} \big( 1+|\mathbf{u}_{n}(s)|^{2}_{\mathbb{H}} \big) \,ds \bigg] \leq c \cdot \theta + c \ \mathbb{E}\bigg[ \int_{\tau_n}^{\tau_n + \theta} |\mathbf{u}_{n}(s)|^{2}_{\mathbb{H}} \,ds \bigg] \nonumber 
\\ &\leq c \cdot \theta + c \cdot \theta \ \mathbb{E}\bigg[ \sup_{s \in [0, T]} |\mathbf{u}_{n}(s)|^{2}_{\mathbb{H}} \bigg] \leq c \cdot \theta (1 + C_{1, T}) =: c_5 \cdot \theta.
\end{align}
Where the constant $C_{1, T}$ is used in \eqref{u1ti}. Thus $ k_{5}^{n}$ satisfies condition \eqref{aco1} with $\alpha = 2$ and $\beta = 1.$

Hence by Lemma \ref{aco} the sequence $(\mathbf{u}_{n})_{n \in \mathbb{N}}$ satisfies the Aldous condition in the space $\mathbb{V}'.$

Now we can rewrite \eqref{gg2nd} further as
\begin{align*}
 \mathbf{d}_{n}(t) &= \mathbf{d}_{0n} - \int_{0}^{t} \mathcal{A} \mathbf{d}_{n}(s) \,ds - \int_{0}^{t} \tilde{B}_{n}(\mathbf{u}_{n}(s), \mathbf{d}_{n}(s)) \,ds - \int_{0}^{t} f_{n}(\mathbf{d}_{n}(s)) \,ds
\\ &=: j_{1}^{n} + \sum_{k = 2}^{4} j_{k}^{n}(t) , \qquad t \in [0, T].
\end{align*}

It is easy to verify that $j_{1}^{n}$ satifies condition \eqref{aco1} with $\alpha = 1$ and $\beta = 1.$

Now consider $ j_{2}^{n}(t)$. Since the embedding $(H^1)' \hookrightarrow (H^2)'$ is continuous, by \eqref{Ah1}, \eqref{dnv2} and the H\"older inequality we get,
\begin{align*}
&\mathbb{E}\big[ \big|j_{2}^{n}(\tau_n + \theta) - j_{2}^{n}(\tau_n) \big|_{(H^2)'} \big] = \mathbb{E}\bigg[ \bigg| \int_{\tau_n}^{\tau_n + \theta} \mathcal{A} \mathbf{d}_{n}(s) \,ds \bigg|_{(H^2)'} \bigg]
\\ &\leq c \ \mathbb{E}\bigg[ \int_{\tau_n}^{\tau_n + \theta} \big| \mathcal{A} \mathbf{d}_{n}(s)\big|_{(H^1)'} \,ds \bigg] \leq c \ \mathbb{E}\bigg[ \theta^{\frac{1}{2}} \bigg( \int_{0}^{T} \big\| \mathbf{d}_{n}(s) \big\|_{H^1}^{2} \,ds \bigg)^{\frac{1}{2}} \bigg]
\\ &\leq c \ \theta^{\frac{1}{2}} \ \bigg(\mathbb{E}\bigg[ \int_{0}^{T} \big\| \mathbf{d}_{n}(s) \big\|_{H^1}^{2} \,ds \bigg] \bigg)^{\frac{1}{2}} \leq c \ \theta^{\frac{1}{2}} \sqrt{C_T} =: \bar{c}_2 \cdot \theta^{\frac{1}{2}}.
\end{align*}
Where the constant $C_T$ is coming from \eqref{dnv2}. Thus $ j_{2}^{n}$ satisfies condition \eqref{aco1} with $\alpha = 1$ and $\beta = \frac{1}{2}.$

Now consider $ j_{3}^{n}(t)$. Since the embedding $L^2 \hookrightarrow (H^2)'$ is continuous, from \eqref{btil}
\begin{align}\label{j3nt}
&\mathbb{E}\big[ \big|j_{3}^{n}(\tau_n + \theta) - j_{3}^{n}(\tau_n) \big|_{(H^2)'} \big] = \mathbb{E}\bigg[ \bigg| \int_{\tau_n}^{\tau_n + \theta} \tilde{B}_{n}(\mathbf{u}_{n}(s), \mathbf{d}_{n}(s)) \,ds \bigg|_{(H^2)'} \bigg] \nonumber
\\ &\leq c \ \mathbb{E}\bigg[ \int_{\tau_n}^{\tau_n + \theta} |\tilde{B}(\mathbf{u}_{n}(s), \mathbf{d}_{n}(s))|_{L^2} \,ds \bigg] \nonumber
\\ &\leq c \ \mathbb{E}\bigg[ \int_{\tau_n}^{\tau_n + \theta} \big\{ \big|\mathbf{u}_{n}(s) \big|_{\mathbb{H}}^{1- \frac{\mathbf{n}}{4}} \ \| \mathbf{u}_{n}(s)\|^{\frac{\mathbf{n}}{4}} \ \| \mathbf{d}_{n}(s)\|^{1- \frac{\mathbf{n}}{4}} \ \big|\Delta \mathbf{d}_{n}(s)\big|_{L^2}^{\frac{\mathbf{n}}{4}}\big\} \,ds \bigg]
\end{align}

We will prove for $\mathbf{n} = 2$ and 3 separately.\\  First consider $\mathbf{n} = 2$. Using Young's inequality for this case in \eqref{j3nt} we have,
\begin{align}\label{j3nn}
&\mathbb{E}\big[ \big|j_{3}^{n}(\tau_n + \theta) - j_{3}^{n}(\tau_n) \big|_{(H^2)'} \big] \leq  c \ \mathbb{E}\bigg[ \int_{\tau_n}^{\tau_n + \theta} \big( |\mathbf{u}_{n}(s)| \cdot \| \mathbf{u}_{n}(s)\| \big)^{\frac{1}{2}} \big( \| \mathbf{d}_{n}(s)\| \cdot |\Delta \mathbf{d}_{n}(s)| \big)^{\frac{1}{2}} \,ds \bigg] \nonumber
\\ &\leq \tilde{c} \ \mathbb{E}\bigg[ \int_{\tau_n}^{\tau_n + \theta} |\mathbf{u}_{n}(s)| \cdot \| \mathbf{u}_{n}(s)\| \,ds + \int_{\tau_n}^{\tau_n + \theta} \| \mathbf{d}_{n}(s)\| \cdot |\Delta \mathbf{d}_{n}(s)| \,ds \bigg]. 
\end{align}

From \eqref{u1ti}, \eqref{ung} and using H\"older's inequality, 
\begin{align}\label{hgu}
&\mathbb{E}\bigg[ \int_{\tau_n}^{\tau_n + \theta} |\mathbf{u}_{n}(s)| \cdot \| \mathbf{u}_{n}(s)\| \,ds \bigg] \leq \mathbb{E}\bigg[ \bigg( \int_{\tau_n}^{\tau_n + \theta} |\mathbf{u}_{n}(s)|^2 \bigg)^{\frac{1}{2}} \bigg( \int_{\tau_n}^{\tau_n + \theta} \| \mathbf{u}_{n}(s)\|^2 \,ds \bigg)^{\frac{1}{2}} \bigg] \nonumber
\\ &\leq \bigg\{ \mathbb{E}\bigg[ \int_{\tau_n}^{\tau_n + \theta} |\mathbf{u}_{n}(s)|^2 \bigg] \bigg\}^{\frac{1}{2}} \cdot \bigg\{ \mathbb{E}\bigg[ \int_{\tau_n}^{\tau_n + \theta} \| \mathbf{u}_{n}(s)\|^2 \,ds \bigg] \bigg\}^{\frac{1}{2}} \nonumber
\\ &\leq \bigg\{ \mathbb{E}\bigg[ \sup_{s \in [0, T]} |\mathbf{u}_{n}(s)|^2 \bigg] \cdot \theta \bigg\}^{\frac{1}{2}} \cdot \bigg\{ \mathbb{E}\bigg[ \int_{0}^{T} \| \mathbf{u}_{n}(s)\|^2 \,ds \bigg] \bigg\}^{\frac{1}{2}} \nonumber
\\ &\leq \sqrt{C_{1, T}} \cdot \theta^{\frac{1}{2}} \cdot \sqrt{C_{1, T}} \leq \tilde{C}_3 \cdot \theta^{\frac{1}{2}}.
\end{align}
 
Using \eqref{gdld} in \eqref{j3nn} and combining with \eqref{hgu}, for $\mathbf{n} = 2,$ we obtain
\begin{align*}
\mathbb{E}\big[ \big|j_{3}^{n}(\tau_n + \theta) - j_{3}^{n}(\tau_n) \big|_{(H^2)'} \big] \leq \bar{c}_3 \cdot \theta^{\frac{1}{2}}.
\end{align*}

Now for $\mathbf{n} = 3.$ Using Young's inequality for the case in \eqref{j3nt} we have,
\begin{align}\label{j3n}
&\mathbb{E}\big[ \big|j_{3}^{n}(\tau_n + \theta) - j_{3}^{n}(\tau_n) \big|_{(H^2)'} \big] \leq  c \ \mathbb{E}\bigg[ \int_{\tau_n}^{\tau_n + \theta} \big( |\mathbf{u}_{n}(s)|^{\frac{1}{4}} \| \mathbf{u}_{n}(s)\|^{\frac{3}{4}} \big) \big( \| \mathbf{d}_{n}(s)\|^{\frac{1}{4}} |\Delta \mathbf{d}_{n}(s)|^{\frac{3}{4}} \big) \,ds \bigg] \nonumber
\\ &\leq \tilde{c} \ \mathbb{E}\bigg[ \int_{\tau_n}^{\tau_n + \theta} |\mathbf{u}_{n}(s)|^{\frac{1}{2}} \| \mathbf{u}_{n}(s)\|^{\frac{3}{2}} \,ds + \int_{\tau_n}^{\tau_n + \theta} \| \mathbf{d}_{n}(s)\|^{\frac{1}{2}} |\Delta \mathbf{d}_{n}(s)|^{\frac{3}{2}} \,ds \bigg]. 
\end{align}

Now for the first term of right hand side of the above inequality, from  \eqref{u1ti} and \eqref{ung} and using H\"older's inequality repeatedly we have,
\begin{align}\label{ujn3}
&\mathbb{E}\bigg[ \int_{\tau_n}^{\tau_n + \theta} |\mathbf{u}_{n}(s)|^{\frac{1}{2}} \| \mathbf{u}_{n}(s)\|^{\frac{3}{2}} \,ds \bigg] \leq \mathbb{E}\bigg[ \bigg( \int_{\tau_n}^{\tau_n + \theta} \big(|\mathbf{u}_{n}(s)|^{\frac{1}{2}}\big)^4 \, ds \bigg)^{\frac{1}{4}} \bigg( \int_{\tau_n}^{\tau_n + \theta} \big(\| \mathbf{u}_{n}(s)\|^{\frac{3}{2}}\big)^{\frac{4}{3}} \,ds\bigg)^{\frac{3}{4}} \bigg] \nonumber
\\ &\leq  \mathbb{E}\bigg[ \bigg( \sup_{s \in [0, T]} |\mathbf{u}_{n}(s)|^{2} \cdot \theta \bigg)^{\frac{1}{4}} \bigg( \int_{\tau_n}^{\tau_n + \theta} \| \mathbf{u}_{n}(s)\|^{2} \,ds \bigg)^{\frac{3}{4}}\bigg] \nonumber
\\ &\leq \theta^{\frac{1}{4}} \cdot \bigg\{ \mathbb{E}\bigg[ \sup_{s \in [0, T]} |\mathbf{u}_{n}(s)|^{2} \bigg] \bigg\}^{\frac{1}{4}} \cdot \bigg\{ \mathbb{E}\bigg[ \int_{\tau_n}^{\tau_n + \theta} \| \mathbf{u}_{n}(s)\|^{2} \,ds \bigg] \bigg\}^{\frac{3}{4}} \nonumber
\\ &\leq \theta^{\frac{1}{4}}  (C_{1, T})^{\frac{1}{4}} \ \bigg\{ \mathbb{E}\bigg[ \int_{0}^{T}\| \mathbf{u}_{n}(s)\|^{2} \,ds \bigg] \bigg\}^{\frac{3}{4}} \leq \theta^{\frac{1}{4}}  (C_{1, T})^{\frac{1}{4}} (C_{1, T})^{\frac{3}{4}} =: \bar{c}_3 \cdot \theta^{\frac{1}{4}}.
\end{align}

Now refering to \eqref{gdld3} and using the estimate in \eqref{j3n} and combining with \eqref{ujn3}, we obtain for $\mathbf{n} = 3,$
\begin{align*}
\mathbb{E}\big[ \big|j_{3}^{n}(\tau_n + \theta) - j_{3}^{n}(\tau_n) \big|_{(H^2)'} \big] \leq \bar{c}_3 \cdot \theta^{\frac{1}{4}}.
\end{align*}

Thus $j_{3}^{n}(t)$ satisfies condition \eqref{aco1} with $\alpha = 1$ and $\beta = \frac{1}{2}$ for 2-D and $\alpha = 1$ and $\beta = \frac{1}{4}$ for 3-D.

At last consider $ j_{4}^{n}(t)$. The embedding $L^2 \hookrightarrow (H^2)'$ is continuous and $H^1 \hookrightarrow L^{\bar{q}}$ for $\bar{q} = 4N+2$. Then from Remark \ref{fes} and Corollary \ref{codnh} we obtain,
\begin{align}
&\mathbb{E}\big[ \big|j_{4}^{n}(\tau_n + \theta) - j_{4}^{n}(\tau_n) \big|_{(H^2)'} \big] = \mathbb{E}\bigg[ \bigg| \int_{\tau_n}^{\tau_n + \theta} f_n(\mathbf{d}_{n}(s)) \,ds \bigg|_{(H^2)'} \bigg] \nonumber
\\ &\leq c \ \mathbb{E}\bigg[ \int_{\tau_n}^{\tau_n + \theta} |f(\mathbf{d}_{n}(s))|_{L^2} \,ds \bigg] \leq c \cdot \theta + c \ \mathbb{E}\bigg[ \int_{\tau_n}^{\tau_n + \theta} \big\| \mathbf{d}_{n}(s) \big\|_{L^{\bar{q}}}^{2N+1} \,ds \bigg] \nonumber
\\ &\leq c \cdot \theta + c \ \mathbb{E}\bigg[ \sup_{s \in [0, T]} \big\| \mathbf{d}_{n}(s) \big\|_{L^{\bar{q}}}^{2N+1} \cdot \theta \bigg] \leq  c \cdot \theta + c \cdot \theta \ \mathbb{E}\bigg[ \sup_{s \in [0, T]} \big\| \mathbf{d}_{n}(s)\big\|_{H^1}^{2N+1} \bigg] \nonumber
\\ &\leq c \cdot \theta \ (1 + C_{2N+1}) \leq \bar{c}_4 \cdot \theta.
\end{align}

The constant $C_{2N+1}$ is coming from  Corollary \ref{codnh}. Thus $ j_{4}^{n}(t)$ satisfies condition \eqref{aco1} with $\alpha = 1$ and $\beta = 1.$

\end{proof}

\section{Existence of Martingale Solution}\label{eoms}

We will now prove the existence of a martingale solution. The main difficulties lie in the terms containing the nonlinearity of $B, M$ and the noise term $F$. The Skorokhod Theorem for nonmetric spaces helps us constructing a martingale solution. 

\subsection{Construction of New Probability Space and Processes}

By Lemma \ref{tig} we have shown the set of measures $\{ \mathscr{L}(\mathbf{u}_{n}, \mathbf{d}_{n}), n \in \mathbb{N}\}$ is tight on $(\mathcal{Z}_{T, 1} \times \mathcal{Z}_{T, 2}, \mathscr{T}).$ Let $\eta_n := \eta, n \in \mathbb{N}.$ Then the set of measures $\{ \mathscr{L}(\eta_{n}), n \in \mathbb{N}\}$ is tight on the space $M_{\bar{\mathbb{N}}}([0, T] \times Y).$ Thus the set $\{ \mathscr{L}(\mathbf{u}_{n}, \mathbf{d}_{n}, \eta_{n}),$ $n \in \mathbb{N}\}$ is tight on $\mathcal{Z}_{T} \times M_{\bar{\mathbb{N}}}([0, T] \times Y)$.

By Corollary \ref{set3} and Remark \ref{sep}, there exists a subsequence $(n_k)_{k \in \mathbb{N}},$ a probability space $(\bar{\Omega}, \bar{\mathcal{F}}, \bar{\mathbb{P}})$ and on this space, $\mathcal{Z}_{T} \times M_{\bar{\mathbb{N}}}([0, T] \times Y)$-valued random variables $(\mathbf{u}_{\ast}, \mathbf{d}_{\ast}, \eta_{\ast}), (\bar{\mathbf{u}}_k, \bar{\mathbf{d}}_k, \bar{\eta}_k), k \in \mathbb{N}$ such that  
\begin{enumerate}[label=(\alph*)]
\item
$\mathscr{L}\big( (\bar{\mathbf{u}}_k, \bar{\mathbf{d}}_k, \bar{\eta}_k)\big) = \mathscr{L}\big( (\mathbf{u}_{n_k}, \mathbf{d}_{n_k}, \eta_{n_k})\big)$ for all $k \in \mathbb{N};$
\item
$(\bar{\mathbf{u}}_k, \bar{\mathbf{d}}_k, \bar{\eta}_k) \to (\mathbf{u}_{\ast}, \mathbf{d}_{\ast}, \eta_{\ast})$ in $\mathcal{Z}_{T} \times M_{\bar{\mathbb{N}}}([0, T] \times Y)$ with probability 1 on $(\bar{\Omega}, \bar{\mathcal{F}}, \bar{\mathbb{P}})$ as $k \to \infty;$
\item
$\bar{\eta}_k(\bar{\omega}) = \eta_{\ast}(\bar{\omega})$ for all $\bar{\omega} \in \bar{\Omega}.$
\end{enumerate}

\noindent We will denote these sequences again by $\big( (\mathbf{u}_{n}, \mathbf{d}_{n}, \eta_{n})\big)_{n \in \mathbb{N}}$ and $\big( (\bar{\mathbf{u}}_n, \bar{\mathbf{d}}_n, \bar{\eta}_n)\big)_{n \in \mathbb{N}}.$

Using the definiton of $\mathcal{Z}_T$, we have $\bar{\mathbb{P}}$-a.s.
\begin{align}\label{zt1c}
\bar{\mathbf{u}}_n \to \mathbf{u}_{\ast} \ in \ L^{2}_{w}(0, T; \mathbb{V}) \cap L^{2}(0, T; \mathbb{H}) \cap \mathbb{D}([0, T]; \mathbb{V}') \cap \mathbb{D}([0, T]; \mathbb{H}_{w}) 
\end{align}
and
\begin{align}\label{zt2c}
\bar{\mathbf{d}}_n \to \mathbf{d}_{\ast} \ \text{in} \ L^{2}_{w}(0, T; H^{2}) \cap L^{2}(0, T; H^{1}) \cap \mathbb{C}([0, T]; (H^2)') \cap \mathbb{C}([0, T]; H^{1}_{w}).
\end{align}

\subsection{Properties of The New Processes and The Limiting Processes} 

We have the following result due to Kuratowski Theorem.

\begin{proposition}
The Borel subsets of $\mathbb{D}([0, T], \mathbb{H}_n)$ are Borel subsets of $\mathcal{Z}_{T, 1}$ $($defined in Section \ref{tolas}$)$ and the Borel subsets of $\mathbb{D}([0, T], \mathbb{L}_n)$ are Borel subsets of $\mathcal{Z}_{T, 2}.$
\end{proposition}
So we obtain the following results.

\begin{corollary}
$\bar{\mathbf{u}}_{n}$ and $\bar{\mathbf{d}}_{n}$ take values in $\mathbb{H}_n$ and $\mathbb{L}_n$ respectively. The laws of $\mathbf{u}_{n}$ and $\bar{\mathbf{u}}_{n}$ are equal on $\mathbb{D}([0, T], \mathbb{H}_n)$ and the laws of $\mathbf{d}_{n}$ and $\bar{\mathbf{d}}_{n}$ are equal on $\mathbb{C}([0, T], \mathbb{L}_n)$.
\end{corollary}

Since the random variables $\mathbf{u}_{n}$ and $\bar{\mathbf{u}}_{n}$ are identically distributed, from the above results, \eqref{u1ti} and \eqref{unv} we have, 
\begin{align}\label{unb2p}
\sup_{n \in \mathbb{N}} \bar{\mathbb{E}} \bigg[ \sup_{s \in [0, T]} \big| \bar{\mathbf{u}}_{n}(s)\big|^{2}_{\mathbb{H}}\bigg] \leq C_{1, T},
\end{align}
and
\begin{align}\label{unbv}
\sup_{n \in \mathbb{N}} \bar{\mathbb{E}} \bigg[ \int_{0}^{T} \big\|  \bar{\mathbf{u}}_{n}(s) \big\|^{2}_{\mathbb{V}} \,ds \bigg] \leq \tilde{C}(T). 
\end{align}

From \eqref{unb2p} and Banach-Alaoglu Theorem we conclude there exists a subsequence of $(\bar{\mathbf{u}}_{n})$ convergent weak star in $L^2(\bar{\Omega}; L^{\infty}(0, T; \mathbb{H}))$. So from \eqref{zt1c} we infer $\mathbf{u}_{\ast} \in L^2(\bar{\Omega}; L^{\infty}(0, T; \mathbb{H})),$ i.e.,
\begin{align}\label{us2p}
\bar{\mathbb{E}} \bigg[ \sup_{t \in [0, T]} \big| \mathbf{u}_{\ast}(t) \big|^{2}_{\mathbb{H}}\bigg] < \infty,
\end{align}

Similarly by \eqref{zt1c}, \eqref{unbv} and Banach-Alaoglu Theorem, there exists a subsequence of $(\bar{\mathbf{u}}_{n}),$ weakly convergent in $L^2([0, T] \times \bar{\Omega}; \mathbb{V}),$ i.e.,
\begin{align}\label{usv2}
\bar{\mathbb{E}} \bigg[ \int_{0}^{T} \big\| \mathbf{u}_{\ast}(t) \big\|^{2}_{\mathbb{V}} \,ds \bigg] < \infty. 
\end{align}

Also from \eqref{u1tip} we get,
\begin{align}\label{unb2pg}
\sup_{n \in \mathbb{N}} \bar{\mathbb{E}} \bigg[ \sup_{s \in [0, T]} \big| \bar{\mathbf{u}}_{n}(s)\big|^{2p}_{\mathbb{H}}\bigg] \leq C_{p, T},
\end{align}

and from Proposition \ref{ugrdn}, for $p \geq 2$ we observe,
\begin{align}\label{unbvn}
\sup_{n \in \mathbb{N}} \bar{\mathbb{E}} \bigg[ \int_{0}^{T} \|  \bar{\mathbf{u}}_{n}(s) \|^{2} \,ds \bigg]^p \leq C_{p, T}. 
\end{align}

Since the random variables $\mathbf{d}_{n}$ and $\bar{\mathbf{d}}_{n}$ are identically distributed, from \eqref{ssdnp} and \eqref{dnv2}, we have for $p =2,$
\begin{align}\label{dnbp}
\sup_{n \in \mathbb{N}} \bar{\mathbb{E}} \bigg[ \sup_{s \in [0, T]} \big| \bar{\mathbf{d}}_{n}(s)\big|_{L^2}^{2} \bigg] \leq \tilde{C}_{2, T},
\end{align}
and
\begin{align}\label{dnbv}
\sup_{n \in \mathbb{N}} \bar{\mathbb{E}} \bigg[ \int_{0}^{T} \big\|  \bar{\mathbf{d}}_{n}(s) \big\|^{2}_{H^1} \,ds \bigg] \leq C_T. 
\end{align}

From \eqref{dnbp} and Banach-Alaoglu Theorem we conclude there exists a subsequence of $(\bar{\mathbf{d}}_{n})$ convergent weak star in $L^2(\bar{\Omega}; L^{\infty}(0, T; L^2))$. So from \eqref{zt2c} we infer $\mathbf{d}_{\ast} \in L^2(\bar{\Omega}; L^{\infty}(0, T; L^2)),$ i.e.,
\begin{align}\label{dsp}
\bar{\mathbb{E}} \bigg[ \sup_{t \in [0, T]} \big| \mathbf{d}_{\ast}(s) \big|_{L^2}^{2} \bigg] < \infty,
\end{align}

Similarly for $p =2,$ from \eqref{zt2c}, \eqref{dnbv} and Banach-Alaoglu Theorem, there exists a subsequence of $(\bar{\mathbf{d}}_{n}),$ weakly convergent in $L^2([0, T] \times \bar{\Omega}; H^1),$ i.e.,
\begin{align}\label{dsv2}
\bar{\mathbb{E}} \bigg[ \int_{0}^{T} \big\| \mathbf{d}_{\ast}(s) \big\|^{2}_{H^1} \,ds \bigg] < \infty. 
\end{align}

From \eqref{ssdnp} we also have for $p \geq 2$,
\begin{align}\label{dsp1}
\bar{\mathbb{E}} \bigg[ \sup_{t \in [0, T]} \big| \bar{\mathbf{d}}_n(s) \big|_{L^2}^{p} \bigg] < \tilde{C}_{p, T},
\end{align}

Again from Corollary \ref{codnh} we obtain,
\begin{align}\label{2prop}
\bar{\mathbb{E}} \bigg[ \sup_{s \in [0, T]} \big\| \bar{\mathbf{d}}_n(s)\big\|^{q}_{H^1} \bigg] \leq  C_q.
\end{align}

Similarly, from Proposition \ref{Ldn} we have, 
\begin{align}\label{dnbla}
\bar{\mathbb{E}} \bigg[ \int_{0}^{T} \big|\Delta \bar{\mathbf{d}}_n(s) \big|_{L^2}^2  \,ds \bigg]^q \leq C(q).
\end{align}
So from \eqref{dnH2} and using  Banach-Alaoglu Theorem, we have a subsequence of $\bar{\mathbf{d}}_n,$ convergent weakly in $L^2([0, T] \times \bar{\Omega}; H^2).$ As from \eqref{zt2c}, $\bar{\mathbf{d}}_n \to \mathbf{d}_{\ast}$ in $L^{2}_{w}([0, T]; H^2),$ we obtain for $q=1,$
 \begin{align}\label{dnsl}
\bar{\mathbb{E}} \bigg[ \int_{0}^{T} \big\| \mathbf{d}_{\ast}(s) \big\|_{H^2}^2  \,ds \bigg] \leq C.
\end{align}

Also from Remark \ref{fes}, we obtain
\begin{align}
|f(\bar{\mathbf{d}}_n)|_{\mathbb{R}^{\mathbf{n}}} \leq c \ \big(1 + \big|\bar{\mathbf{d}}_n \big|_{\mathbb{R}^{\mathbf{n}}}^{2N+1} \big)
\end{align}

\subsection{Convergence of the New Processes to the Corresponding Limiting Processes}

Let us fix $v \in \mathbb{V}$. Let us denote
\begin{align}\label{kn}
&\mathscr{K}_n(\bar{\mathbf{u}}_n, \bar{\mathbf{d}}_n, \bar{\eta}_n, v)(t) := \big( \bar{\mathbf{u}}_{n}(0), v \big)_{\mathbb{H}} - \int_{0}^{t} \big\langle \mathscr{A} \bar{\mathbf{u}}_n(s), v \big\rangle \, ds \nonumber
\\ & \ \ - \int_{0}^{t} \big\langle B_n(\bar{\mathbf{u}}_n(s)), v \big\rangle \,ds - \int_{0}^{t} \big\langle M_n(\bar{\mathbf{d}}_n(s)), v \big\rangle \,ds \nonumber
\\ & \ \ + \int_{0}^{t} \int_{Y} \big( P_n F(s, \bar{\mathbf{u}}_n(s); y), v\big)_{\mathbb{H}} \, \tilde{\bar{\eta}}_n(ds, dy), \qquad t \in [0, T]
\end{align}
and fixing $v \in H^2,$ denote
\begin{align}\label{Ln}
&\Lambda_n(\bar{\mathbf{u}}_n, \bar{\mathbf{d}}_n, v)(t) := \big( \bar{\mathbf{d}}_{n}(0), v \big)_{L^2} - \int_{0}^{t} \big\langle \mathcal{A} \bar{\mathbf{d}}_n(s), v \big\rangle \, ds \nonumber
\\ & \ \ - \int_{0}^{t} \big\langle \tilde{B}_n(\bar{\mathbf{u}}_n(s), \bar{\mathbf{d}}_n(s)), v \big\rangle \,ds - \int_{0}^{t} \big\langle f_n(\bar{\mathbf{d}}_n(s)), v \big\rangle \,ds,  \qquad t \in [0, T].
\end{align}
Now for the limiting processes we denote for $v \in \mathbb{V},$
\begin{align}\label{ks}
&\mathscr{K}(\mathbf{u}_{\ast}, \mathbf{d}_{\ast}, \eta_{\ast}, v)(t) := \big( \mathbf{u}_{\ast}(0), v \big)_{\mathbb{H}} - \int_{0}^{t} \big\langle \mathscr{A} \mathbf{u}_{\ast}(s), v \big\rangle \, ds \nonumber
\\ & \ \ - \int_{0}^{t} \big\langle B(\mathbf{u}_{\ast}(s)), v \big\rangle \,ds - \int_{0}^{t} \big\langle M(\mathbf{d}_{\ast}(s)), v \big\rangle \,ds \nonumber
\\ & \ \ + \int_{0}^{t} \int_{Y} \big( F(s, \mathbf{u}_{\ast}(s); y), v\big)_{\mathbb{H}} \, \tilde{\eta}_{\ast}(ds, dy), \qquad t \in [0, T]
\end{align}
and for $v \in H^2,$
\begin{align}\label{Ls}
&\Lambda(\mathbf{u}_{\ast}, \mathbf{d}_{\ast}, v)(t) := \big( \mathbf{d}_{\ast}(0), v \big)_{L^2} - \int_{0}^{t} \big\langle \mathcal{A} \mathbf{d}_{\ast}(s), v \big\rangle \, ds \nonumber
\\ & \ \ - \int_{0}^{t} \big\langle \tilde{B}(\mathbf{u}_{\ast}(s), \mathbf{d}_{\ast}(s)), v \big\rangle \,ds - \int_{0}^{t} \big\langle f(\mathbf{d}_{\ast}(s)), v \big\rangle \,ds, \qquad t \in [0, T].
\end{align}

We will show that 
\begin{align}\label{knkst}
\lim_{n \to \infty} \| \mathscr{K}_n(\bar{\mathbf{u}}_n, \bar{\mathbf{d}}_n, \bar{\eta}_n, v) - \mathscr{K}(\mathbf{u}_{\ast}, \mathbf{d}_{\ast}, \eta_{\ast}, v) \|_{L^2([0, T] \times \bar{\Omega})} = 0. 
\end{align}
and
\begin{align}\label{LnLst}
\lim_{n \to \infty} \| \Lambda_n(\bar{\mathbf{u}}_n, \bar{\mathbf{d}}_n, v) - \Lambda(\mathbf{u}_{\ast}, \mathbf{d}_{\ast}, v) \|_{L^2([0, T] \times \bar{\Omega})} = 0. 
\end{align}

Now for proving \eqref{knkst}, using Fubini's Theorem, we have
\begin{align}
&\| \mathscr{K}_n(\bar{\mathbf{u}}_n, \bar{\mathbf{d}}_n, \bar{\eta}_n, v) - \mathscr{K}(\mathbf{u}_{\ast}, \mathbf{d}_{\ast}, \eta_{\ast}, v) \|^{2}_{L^2([0, T] \times \bar{\Omega})} \nonumber
\\ &= \int_{0}^{T} \int_{\bar{\Omega}} |\mathscr{K}_n(\bar{\mathbf{u}}_n, \bar{\mathbf{d}}_n, \bar{\eta}_n, v)(t) - \mathscr{K}(\mathbf{u}_{\ast}, \mathbf{d}_{\ast}, \eta_{\ast}, v)(t)|^2 \,d\bar{\mathbb{P}}(\omega) \ dt \nonumber
\\ &= \int_{0}^{T} \bar{\mathbb{E}} \big[ |\mathscr{K}_n(\bar{\mathbf{u}}_n, \bar{\mathbf{d}}_n, \bar{\eta}_n, v)(t) - \mathscr{K}(\mathbf{u}_{\ast}, \mathbf{d}_{\ast}, \eta_{\ast}, v)(t)|^2 \big] \,dt.
\end{align}

So we will show each term of right hand side of \eqref{kn} tends to the corresponding terms in \eqref{ks} in $L^2([0, T] \times \bar{\Omega}).$ Similarly for proving \eqref{LnLst}, we will show each term of right hand side of \eqref{Ln} tends to the corresponding terms in \eqref{Ls} in $L^2([0, T] \times \bar{\Omega}).$ So we need to prove the following Lemmas.

\begin{lemma}\label{knterm}
For all $v \in \mathbb{V}$
\begin{enumerate}[label=(\alph*)]
\item
$\lim_{n \to \infty} \bar{\mathbb{E}} \big[\int_{0}^{T} \big| (\bar{\mathbf{u}}_{n}(t) - \mathbf{u}_{\ast}(t), v)_{\mathbb{H}}\big|^{2} \, dt \big] = 0,$ 
\item
$\lim_{n \to \infty} \int_{0}^{T} \bar{\mathbb{E}} \big[ \big| (\bar{\mathbf{u}}_{n}(0) - \mathbf{u}_{\ast}(0), v )_{\mathbb{H}} \big|^2 \big] = 0,$
\item
$\lim_{n \to \infty} \int_{0}^{T} \bar{\mathbb{E}} \big[ \big| \int_{0}^{t} \langle \mathscr{A} \bar{\mathbf{u}}_{n}(s) - \mathscr{A}\mathbf{u}_{\ast}(s), v \rangle \,ds \big|^{2} \big] \,dt = 0,$
\item
$\lim_{n \to \infty} \int_{0}^{T} \bar{\mathbb{E}} \big[ \big| \int_{0}^{t} \big\langle B_n \big( \bar{\mathbf{u}}_{n}(s) \big)-B \big( \mathbf{u}_{\ast}(s) \big), v \big\rangle \,ds \big|^{2} \big] \,dt =0,$
\item
$\lim_{n \to \infty} \int_{0}^{T} \bar{\mathbb{E}} \big[ \big| \int_{0}^{t} \big\langle M_n \big( \bar{\mathbf{d}}_{n}(s) \big)-M \big( \mathbf{d}_{\ast}(s) \big), v \big\rangle \,ds \big|^2 \big] \,dt =0,$
\item
$\lim_{n \to \infty} \int_{0}^{T} \bar{\mathbb{E}} \big[ \big| \int_{0}^{t} \int_{Y} \big\langle P_n F(s, \bar{\mathbf{u}}_{n}(s), y) - F(s, \mathbf{u}_{\ast}(s), y), v \big\rangle \,\tilde{\eta}_{\ast}(ds, dy) \big|^2 \,dt \big] = 0.$
\end{enumerate}
\end{lemma}

\begin{lemma}\label{lnterm}
For all $v \in H^2$
\begin{enumerate}[label=(\alph*)]
\item
$\lim_{n \to \infty} \bar{\mathbb{E}} \big[\int_{0}^{T} \big| (\bar{\mathbf{d}}_{n}(t) - \mathbf{d}_{\ast}(t), v)_{L^2}\big|^{2} \, dt \big] = 0,$
\item
$\lim_{n \to \infty} \int_{0}^{T} \bar{\mathbb{E}} \big[ \big| (\bar{\mathbf{d}}_{n}(0) - \mathbf{d}_{\ast}(0), v )_{L^2} \big|^2 \big] = 0,$
\item
$\lim_{n \to \infty} \int_{0}^{T} \bar{\mathbb{E}} \big[ \big| \int_{0}^{t} \langle \mathcal{A} \bar{\mathbf{d}}_{n}(s) - \mathcal{A}\mathbf{d}_{\ast}(s), v \rangle \,ds \big|^{2} \big] \,dt = 0,$
\item
$\lim_{n \to \infty} \int_{0}^{T} \bar{\mathbb{E}} \big[ \big| \int_{0}^{t} \big\langle \tilde{B}_n \big( \bar{\mathbf{u}}_{n}(s), \bar{\mathbf{d}}_{n}(s) \big) - \tilde{B} \big( \mathbf{u}_{\ast}(s), \mathbf{d}_{\ast}(s) \big), v \big\rangle \,ds \big|^2 \big] \,dt =0,$
\item
$\lim_{n \to \infty} \int_{0}^{T} \bar{\mathbb{E}} \big[ \big| \int_{0}^{t} \big\langle f_n \big( \bar{\mathbf{d}}_{n}(s) \big) - f \big( \mathbf{d}_{\ast}(s) \big), v \big\rangle \,ds \big|^2 \big] \,dt =0.$
\end{enumerate}
\end{lemma}

\begin{proof}
First we establish the proof of Lemma \ref{knterm}. 
\begin{enumerate}[label=(\alph*)]
\item
Let us consider
\begin{align}
&\| ( \bar{\mathbf{u}}_{n}(\cdot), v )_{\mathbb{H}} - ( \mathbf{u}_{\ast}(\cdot), v)_{\mathbb{H}}\|^{2}_{L^2([0, T] \times \bar{\Omega})} = \int_{\bar{\Omega}} \int_{0}^{T} \big| (\bar{\mathbf{u}}_{n}(t) - \mathbf{u}_{\ast}(t), v)_{\mathbb{H}}\big|^{2} \, dt \ \bar{\mathbb{P}}(d\omega) \nonumber
\\ &= \bar{\mathbb{E}} \bigg[ \int_{0}^{T} \big| (\bar{\mathbf{u}}_{n}(t) - \mathbf{u}_{\ast}(t), v)_{\mathbb{H}}\big|^{2} \, dt \bigg]
\end{align}
Moreover,
\begin{align}\label{uudas}
&\int_{0}^{T} \big| (\bar{\mathbf{u}}_{n}(t) - \mathbf{u}_{\ast}(t), v)_{\mathbb{H}}\big|^{2} \, dt = \int_{0}^{T} \big| _{\mathbb{V}'}\big\langle \bar{\mathbf{u}}_{n}(t) - \mathbf{u}_{\ast}(t), v \big\rangle_{\mathbb{V}}\big|^{2} \, dt \leq \|v\|_{\mathbb{V}}^{2} \int_{0}^{T} \big| \bar{\mathbf{u}}_{n}(t) - \mathbf{u}_{\ast}(t) \big|_{\mathbb{V}'}^{2} \,dt
\end{align}

By \eqref{zt1c}, $\bar{\mathbf{u}}_{n} \to \mathbf{u}_{\ast}$ in $\mathbb{D}([0, T]; \mathbb{V}')$ and from \eqref{unb2p}, $\sup_{t \in [0, T]} |\bar{\mathbf{u}}_{n}(t)|_{\mathbb{H}}^{2} < \infty$, $\bar{\mathbb{P}}$-a.s.. The embedding $\mathbb{H} \hookrightarrow \mathbb{V}'$ is continuous. Then by Dominated Convergence Theorem we observe that $\bar{\mathbf{u}}_{n} \to \mathbf{u}_{\ast}$ in $L^2(0, T; \mathbb{V}').$ So from \eqref{uudas},
\begin{align}\label{1lim}
\lim_{n \to \infty} \int_{0}^{T} \big| (\bar{\mathbf{u}}_{n}(t) - \mathbf{u}_{\ast}(t), v)_{\mathbb{H}}\big|^{2} \, dt = 0.
\end{align}

Moreover, from \eqref{unb2pg}, Proposition \ref{ush2p} in Appendix and using H\"older's inequality, for every $n \in \mathbb{N}$ and every $r > 1$ we obtain
\begin{align}\label{1uni}
&\bar{\mathbb{E}} \bigg[ \bigg| \int_{0}^{T} \big| \bar{\mathbf{u}}_{n}(t) - \mathbf{u}_{\ast}(t) \big|^{2}_{\mathbb{H}} \,dt \bigg|^{2+r} \bigg] \leq c \ \bar{\mathbb{E}} \bigg[ \int_{0}^{T} \big( \big| \bar{\mathbf{u}}_{n}(t) \big|^{2(2+r)}_{\mathbb{H}} + \big| \mathbf{u}_{\ast}(t) \big|^{2(2+r)}_{\mathbb{H}} \,dt \big) \bigg] \nonumber
\\ & \ \leq \tilde{c} \ \bar{\mathbb{E}} \bigg[ \sup_{t \in [0, T]} \big| \bar{\mathbf{u}}_{n}(t) \big|^{2(2+r)}_{\mathbb{H}} \bigg] \leq \tilde{c} \cdot C(4+2r, T) < \infty.
\end{align}
for some constant $\tilde{c} > 0.$ Then by \eqref{1lim}, \eqref{1uni} and Vitali's Theorem we obtain
\begin{align*}
\lim_{n \to \infty} \bar{\mathbb{E}} \bigg[\int_{0}^{T} \big| (\bar{\mathbf{u}}_{n}(t) - \mathbf{u}_{\ast}(t), v)_{\mathbb{H}}\big|^{2} \, dt \bigg] = 0, \quad \text{which proves (a)}.
\end{align*}

\item
 From \eqref{zt1c}, $\bar{\mathbf{u}}_n \to \mathbf{u}_{\ast}$ in $\mathbb{D}([0, T]; \mathbb{H}_{w}) \ \bar{\mathbb{P}}$-a.s. and $\mathbf{u}_{\ast}$ is right continuous at $t =0.$ So we obtain $( \bar{\mathbf{u}}_{n}(0), v )_{\mathbb{H}} \to ( \mathbf{u}_{\ast}(0), v)_{\mathbb{H}}, \bar{\mathbb{P}}$-a.s. From \eqref{unb2p} and applying Vitali's Theorem, we get
\begin{align*}
\lim_{n \to \infty} \bar{\mathbb{E}} \bigg[ \big| (\bar{\mathbf{u}}_{n}(0) - \mathbf{u}_{\ast}(0), v )_{\mathbb{H}} \big|^2 \bigg] = 0
\end{align*}
Hence,
\begin{align}
\lim_{n \to \infty} \big\| (\bar{\mathbf{u}}_{n}(0) - \mathbf{u}_{\ast}(0), v )_{\mathbb{H}} \big\|^{2}_{L^2([0, T] \times \bar{\Omega})} = 0.
\end{align}
which proves (b).

\item
Now from \eqref{zt1c},  $\bar{\mathbf{u}}_n \to \mathbf{u}_{\ast}$ in $L^{2}_{w}(0, T; \mathbb{V}), \bar{\mathbb{P}}$-a.s., then from \eqref{scra} and for all $v \in \mathbb{V}$ we obtain $\bar{\mathbb{P}}$-a.s.,
\begin{align}\label{2lim}
&\lim_{n \to \infty} \int_{0}^{t} \langle \mathscr{A} \bar{\mathbf{u}}_{n}(s), v \rangle \,ds =  \lim_{n \to \infty} \int_{0}^{t} ((\bar{\mathbf{u}}_{n}(s), v )) \,ds \nonumber
\\ &= \int_{0}^{t} (( \mathbf{u}_{\ast}(s), v )) \,ds =  \int_{0}^{t} \langle \mathscr{A}\mathbf{u}_{\ast}(s), v \rangle \,ds
\end{align}

By \eqref{unbvn} and using H\"older's inequality we obtain for all $t \in [0, T], r > 2$  and $n \in \mathbb{N},$
\begin{align}\label{2uni}
&\bar{\mathbb{E}} \bigg[ \bigg| \int_{0}^{t} {_{\mathbb{V}'}}\langle \mathscr{A}\bar{\mathbf{u}}_{n}(s), v \rangle_{\mathbb{V}} \,ds \bigg|^{2 + r} \bigg] = \bar{\mathbb{E}} \bigg[ \bigg| \int_{0}^{t} ((\bar{\mathbf{u}}_{n}(s), v )) \,ds \bigg|^{2+r} \bigg] \nonumber
\\ &\leq \bar{\mathbb{E}} \bigg[ \bigg( \int_{0}^{t} \|\bar{\mathbf{u}}_{n}(s)\| \| v \|_{\mathbb{V}} \,ds \bigg)^{2+r} \bigg] \leq c \ \| v \|^{2+r}_{\mathbb{V}} \ \bar{\mathbb{E}} \bigg[ \bigg( \int_{0}^{T} \|\bar{\mathbf{u}}_{n}(s)\| \,ds \bigg)^{2+r} \bigg] \nonumber
\\ &\leq \tilde{c} \ \bar{\mathbb{E}} \bigg[ \bigg(\int_{0}^{T} \|\bar{\mathbf{u}}_{n}(s)\|^{2} \,ds \bigg)^{1+ \frac{r}{2}} \bigg]  \leq C
\end{align}
for some constant $C > 0.$ Then by \eqref{2lim}, \eqref{2uni} and using Vitali's Theorem we obtain for all $t \in [0, T],$
\begin{align}\label{aunbc}
\lim_{n \to \infty} \bar{\mathbb{E}} \bigg[\bigg| \int_{0}^{t} \langle \mathscr{A} \bar{\mathbf{u}}_{n}(s) - \mathscr{A}\mathbf{u}_{\ast}(s), v \rangle \,ds \bigg|^{2} \bigg] = 0
\end{align}

Now from \eqref{unbv}, using Dominated Convergence Theorem, for all $t \in [0, T]$ and all $n \in \mathbb{N}$ we get,
\begin{align}
\lim_{n \to \infty} \int_{0}^{T} \bar{\mathbb{E}} \bigg[ \bigg| \int_{0}^{t} \langle \mathscr{A} \bar{\mathbf{u}}_{n}(s) - \mathscr{A}\mathbf{u}_{\ast}(s), v \rangle \,ds \bigg|^{2} \bigg] \,dt = 0.
\end{align}

\noindent Now we advance to the nonlinear term.

\item
From Lemma \ref{Pnu} and Lemma \ref{bnlim} of Appendix, we have
\begin{align}\label{lim3}
&\lim_{n \to \infty} \int_{0}^{t} \big\langle B_n \big( \bar{\mathbf{u}}_{n}(s), \bar{\mathbf{u}}_{n}(s) \big)-B \big( \mathbf{u}_{\ast}(s), \mathbf{u}_{\ast}(s) \big), v \big\rangle \,ds \nonumber
\\ &= \lim_{n \to \infty} \int_{0}^{t} \big\langle B \big( \bar{\mathbf{u}}_{n}(s), \bar{\mathbf{u}}_{n}(s) \big)-B \big( \mathbf{u}_{\ast}(s), \mathbf{u}_{\ast}(s) \big), P_{n} v \big\rangle \,ds = 0 \quad \bar{\mathbb{P}}\text{-a.s.}
\end{align}

Now from \eqref{BHH}, \eqref{unb2pg} and using H\"older's inequality, we obtain for all $t \in [0, T], r >1$ and $n \in \mathbb{N}$
\begin{align}\label{3uni}
&\bar{\mathbb{E}} \bigg[ \bigg| \int_{0}^{t} \big\langle B_n \big( \bar{\mathbf{u}}_{n}(s) \big), v \big\rangle \,ds \bigg|^{1+r} \bigg] \leq \bar{\mathbb{E}} \bigg[ \bigg( \int_{0}^{t} \|B_n \big( \bar{\mathbf{u}}_{n}(s) \big)\|_{\mathbb{V}'} \|v\|_{\mathbb{V}} \,ds \bigg)^{1+r} \bigg] \nonumber
\\ &\leq \|v\|_{\mathbb{V}}^{1+r} \ t^r \ \bar{\mathbb{E}} \bigg[ \int_{0}^{t} \|B_n \big( \bar{\mathbf{u}}_{n}(s) \big)\|^{1+r}_{\mathbb{V}'} \,ds \bigg] \leq c \ \bar{\mathbb{E}} \bigg[ \int_{0}^{t} \big| \bar{\mathbf{u}}_{n}(s) \big|_{\mathbb{H}}^{2+2r} \,ds \bigg] \nonumber
\\ &\leq C \ \bar{\mathbb{E}} \bigg[ \sup_{t \in [0, T]} \big| \bar{\mathbf{u}}_{n}(s) \big|_{\mathbb{H}}^{2+2r} \,ds \bigg] \leq C(1+r, T).
\end{align}

So from \eqref{lim3}, \eqref{3uni} and using Vitali's Theorem we obtain for all $t \in [0, T],$
\begin{align}\label{dctc1}
\lim_{n \to \infty} \bar{\mathbb{E}} \bigg[ \bigg| \int_{0}^{t} \big\langle B_n \big( \bar{\mathbf{u}}_{n}(s) \big)-B \big( \mathbf{u}_{\ast}(s) \big), v \big\rangle \,ds \bigg|^2 \bigg] =0.
\end{align}

Now from \eqref{unb2pg}, for all $t \in [0, T]$ and all $n \in \mathbb{N},$
\begin{align}
\bar{\mathbb{E}} \bigg[ \bigg|  \int_{0}^{t} \big\langle B_n \big( \bar{\mathbf{u}}_{n}(s) \big), v \big\rangle \,ds  \bigg|^2 \bigg] \leq c \ \bar{\mathbb{E}} \bigg[ \sup_{t \in [0, T]} \big| \bar{\mathbf{u}}_{n}(s) \big|_{\mathbb{H}}^{4} \,ds \bigg] \leq C_{2, T},
\end{align}
where $C_{2, T}$ is a positive constant. Then from \eqref{dctc1} and Dominated Convergence Theorem we obtain,
\begin{align}
\lim_{n \to \infty} \int_{0}^{T} \bar{\mathbb{E}} \bigg[ \bigg| \int_{0}^{t} \big\langle B_n \big( \bar{\mathbf{u}}_{n}(s) \big)-B \big( \mathbf{u}_{\ast}(s) \big), v \big\rangle \,ds \bigg|^2 \bigg] \,dt =0.
\end{align}

\item
Now we come to the second nonlinear term. From Lemma \ref{Pnu} and Lemma \ref{mlimi} we have,
\begin{align}\label{dlim4}
&\lim_{n \to \infty} \int_{0}^{t} \big\langle M_n \big( \bar{\mathbf{d}}_{n}(s) \big) - M \big( \mathbf{d}_{\ast}(s) \big), v \big\rangle \,ds \nonumber
\\ &= \lim_{n \to \infty} \int_{0}^{t} \big\langle M \big( \bar{\mathbf{d}}_{n}(s)\big) - M \big( \mathbf{d}_{\ast}(s) \big), P_{n} v \big\rangle \,ds = 0 \quad \bar{\mathbb{P}}\text{-a.s.}
\end{align}

Now from \eqref{md1d2} and using H\"older's inequality, we obtain for all $t \in [0, T]$ and $n \in \mathbb{N}$
\begin{align}\label{5uni}
&\bar{\mathbb{E}} \bigg[ \bigg| \int_{0}^{t} \big\langle M_n \big( \bar{\mathbf{d}}_{n}(s) \big), v \big\rangle \,ds \bigg|^{r} \bigg] \leq \bar{\mathbb{E}} \bigg[ \bigg( \int_{0}^{t} \|M_n \big( \bar{\mathbf{d}}_{n}(s) \big)\|_{\mathbb{V}'} \|v\|_{\mathbb{V}} \,ds \bigg)^{r} \bigg] \nonumber
\\ &\leq \|v\|_{\mathbb{V}}^{r} \ t^{r-1} \ \bar{\mathbb{E}} \bigg[ \int_{0}^{t} \|M_n \big( \bar{\mathbf{d}}_{n}(s) \big)\|^{r}_{\mathbb{V}'} \,ds \bigg] \nonumber
\\ &\leq c_{t} \ \bar{\mathbb{E}} \bigg[ \int_{0}^{t} \|\bar{\mathbf{d}}_{n}(s)\|^{2r-\frac{r \mathbf{n}}{2}}  |\Delta \bar{\mathbf{d}}_{n}(s)|_{L^2}^{\frac{r \mathbf{n}}{2}} \,ds \bigg]
\end{align}

Now we will estimate separately for $\mathbf{n} = 2$ and $3.$ First consider the case for $\mathbf{n} =2.$ From \eqref{2prop} and \eqref{dnbla}, for $r \in (1, 2],$ we have
\begin{align}\label{unin2}
&\bar{\mathbb{E}} \bigg[ \int_{0}^{t} \|\bar{\mathbf{d}}_{n}(s)\|^{r}  |\Delta \bar{\mathbf{d}}_{n}(s)|_{L^2}^{r} \,ds \bigg] \leq  \bar{\mathbb{E}} \bigg[ \sup_{s \in [0, T]} \|\bar{\mathbf{d}}_{n}\|^r \int_{0}^{t}   |\Delta \bar{\mathbf{d}}_{n}|^r \,ds \bigg] \nonumber
\\ &\leq \bigg\{ \bar{\mathbb{E}} \bigg[ \sup_{s \in [0, T]} \|\bar{\mathbf{d}}_{n}\|^{2r} \bigg] \bigg\}^{\frac{1}{2}} \bigg\{ \bar{\mathbb{E}} \bigg[ \int_{0}^{t} |\Delta \bar{\mathbf{d}}_{n}|^r \,ds \bigg]^{2} \bigg\}^{\frac{1}{2}} \leq C(r, T).
\end{align}

Now for $\mathbf{n} = 3,$ from \eqref{2prop} and \eqref{dnbla}, for $r \in [1, \frac{4}{3}),$ we obtain
\begin{align}\label{unin3}
&\bar{\mathbb{E}} \bigg[ \int_{0}^{t} \|\bar{\mathbf{d}}_{n}\|^{\frac{r}{2}}  |\Delta \bar{\mathbf{d}}_{n}|^{\frac{3r}{2}} \,ds \bigg] \leq \bar{\mathbb{E}} \bigg[  \bigg(\int_{0}^{t} (\|\bar{\mathbf{d}}_{n}\|^\frac{r}{2})^{\frac{4}{4-3r}} \,ds \bigg)^{\frac{4-3r}{4}} \bigg( \int_{0}^{t} (|\Delta \bar{\mathbf{d}}_{n}|^{\frac{3r}{2}})^{\frac{4}{3r}} \,ds \bigg)^{\frac{3r}{4}} \bigg] \nonumber
\\ &\leq \bar{\mathbb{E}} \bigg[  \bigg(\int_{0}^{t} \|\bar{\mathbf{d}}_{n}\|^\frac{2r}{4-3r} \,ds \bigg)^{\frac{4-3r}{4}} \bigg( \int_{0}^{t} (|\Delta \bar{\mathbf{d}}_{n}|^2 \,ds \bigg)^{\frac{3r}{4}} \bigg] \nonumber
\\ &\leq \bigg\{ \bar{\mathbb{E}} \bigg[ \int_{0}^{t} \|\bar{\mathbf{d}}_{n}\|^\frac{2r}{4-3r} \,ds \bigg] \bigg\}^{\frac{4-3r}{4}} \bigg\{ \bar{\mathbb{E}} \bigg[ \int_{0}^{t} |\Delta \bar{\mathbf{d}}_{n}|^2 \,ds \bigg] \bigg\}^{\frac{3r}{4}} \nonumber
\\ &\leq c \ \bigg\{ \bar{\mathbb{E}} \bigg[ \sup_{s \in [0, T]} \|\bar{\mathbf{d}}_{n}\|^\frac{2r}{4-3r} \bigg] \bigg\}^{\frac{4-3r}{4}} \bigg\{ \bar{\mathbb{E}} \bigg[ \int_{0}^{T} |\Delta \bar{\mathbf{d}}_{n}|^2 \,ds \bigg] \bigg\}^{\frac{3r}{4}} \leq C(r, T).
\end{align}

So from \eqref{dlim4}, \eqref{5uni}, \eqref{unin2}, \eqref{unin3} and using Vitali's Theorem we obtain for all $t \in [0, T],$
\begin{align}\label{mnvi}
\lim_{n \to \infty} \bar{\mathbb{E}} \bigg[ \bigg| \int_{0}^{t} \big\langle M_n \big( \bar{\mathbf{d}}_{n}(s) \big) - M \big( \mathbf{d}_{\ast}(s) \big), v \big\rangle \,ds \bigg|^2 \bigg] =0.
\end{align}

Now following the calculations \eqref{gdld}, \eqref{gdld3} and from \eqref{mnvi}, using Dominated Convergence Theorem we obtain,
\begin{align}\label{mnvi1}
\lim_{n \to \infty} \int_{0}^{T} \bar{\mathbb{E}} \bigg[ \bigg| \int_{0}^{t} \big\langle M_n \big( \bar{\mathbf{d}}_{n}(s) \big) - M \big( \mathbf{d}_{\ast}(s) \big), v \big\rangle \,ds \bigg|^2 \bigg] =0.
\end{align}

\item
Let us proceed to the noise terms. Assume that $v \in \mathbb{H}.$ Using Lipschitz property of $F$, for all $t \in [0, T]$ we have,
\begin{align}
&\int_{0}^{t} \int_{Y} \big| \big( F(s, \bar{\mathbf{u}}_{n}(s), y) - F(s, \mathbf{u}_{\ast}(s), y), v \big)_{\mathbb{H}}\big|^2 \, d\nu(y) \,ds \nonumber
\\ &\leq \int_{0}^{t} \int_{Y} \big| F(s, \bar{\mathbf{u}}_{n}(s), y) - F(s, \mathbf{u}_{\ast}(s), y) \big|^2_{\mathbb{H}} \cdot \big| v\big|^2_{\mathbb{H}} \,d\nu(y) \,ds \nonumber
\\ &\leq C \int_{0}^{t} \big| \bar{\mathbf{u}}_{n}(s) - \mathbf{u}_{\ast}(s) \big|^2_{\mathbb{H}} \,ds \leq C \int_{0}^{T} \big| \bar{\mathbf{u}}_{n}(s) - \mathbf{u}_{\ast}(s) \big|^2_{\mathbb{H}} \,ds.
\end{align}

From \eqref{zt1c} we have, $\bar{\mathbf{u}}_{n} \to \mathbf{u}_{\ast}$ in $L^2(0, T; \mathbb{H}),$ $\bar{\mathbb{P}}$-a.s. Then we obtain for all $t \in [0, T],$ 
\begin{align}\label{lim4}
\lim_{n \to \infty} \int_{0}^{t} \int_{Y} \big| \big( F(s, \bar{\mathbf{u}}_{n}(s), y) - F(s, \mathbf{u}_{\ast}(s), y), v \big)_{\mathbb{H}}\big|^2 \, d\nu(y) \,ds = 0.
\end{align}

Moreover, from \eqref{lgf}, \eqref{unb2pg} and Proposition \ref{ush2p}, for every $t \in [0, T]$, every $r \geq 1$ and every $n \in \mathbb{N},$
\begin{align}\label{4uni}
&\bar{\mathbb{E}} \bigg[ \bigg| \int_{0}^{t} \int_{Y} \big| \big( F(s, \bar{\mathbf{u}}_{n}(s), y) - F(s, \mathbf{u}_{\ast}(s), y), v \big)_{\mathbb{H}}\big|^2 \, d\nu(y) \,ds \bigg|^{r} \bigg] \nonumber
\\ &\leq C \ |v|^{2r}_{\mathbb{H}} \ \bar{\mathbb{E}} \bigg[ \bigg| \int_{0}^{t} \int_{Y} \big\{ |F(s, \bar{\mathbf{u}}_{n}(s), y)|^{2}_{\mathbb{H}} + |F(s, \mathbf{u}_{\ast}(s), y)|^{2}_{\mathbb{H}} \big\} \, d\nu(y) \,ds \bigg|^{r} \bigg] \nonumber
\\ &\leq C \ \bar{\mathbb{E}} \bigg[ \bigg| \int_{0}^{t} \big\{ 2 + |\bar{\mathbf{u}}_{n}(s)|^{2}_{\mathbb{H}} + |\mathbf{u}_{\ast}(s)|^{2}_{\mathbb{H}} \big\} \,ds \bigg|^{r} \bigg] \nonumber
\\ &\leq C(r, T) \bigg( 1+ \bar{\mathbb{E}} \bigg[ \sup_{s \in [0, T]}  |\bar{\mathbf{u}}_{n}(s)|^{2r}_{\mathbb{H}} \bigg] +\bar{\mathbb{E}} \bigg[ \sup_{s \in [0, T]}  |\mathbf{u}_{\ast}(s)|^{2r}_{\mathbb{H}} \bigg]\bigg) \leq C.
\end{align}
Where $C > 0$ is a constant. Then by \eqref{lim4}, \eqref{4uni} and by Vitali's Theorem, for all $t \in [0, T],$ 
\begin{align}
\lim_{n \to \infty} \bar{\mathbb{E}} \bigg[ \int_{0}^{t} \int_{Y} \big| \big( F(s, \bar{\mathbf{u}}_{n}(s), y) - F(s, \mathbf{u}_{\ast}(s), y), v \big)_{\mathbb{H}}\big|^2 \, d\nu(y) \,ds \bigg] = 0, \quad v\in \mathbb{H}.
\end{align}

Since the restriction of $P_n$ to the space $\mathbb{H}$ is the $(\cdot, \cdot)_{\mathbb{H}}$-projection onto $\mathbb{H}_n,$ we obtain
\begin{align}\label{pnF}
\lim_{n \to \infty} \bar{\mathbb{E}} \bigg[ \int_{0}^{t} \int_{Y} \big| \big( P_n F(s, \bar{\mathbf{u}}_{n}(s), y) - F(s, \mathbf{u}_{\ast}(s), y), v \big)_{\mathbb{H}}\big|^2 \, d\nu(y) \,ds \bigg] = 0, \quad v\in \mathbb{H}.
\end{align}

Since $\mathbb{V} \subset \mathbb{H},$ \eqref{pnF} holds for all $v \in \mathbb{V}.$ As $\bar{\eta}_n = \eta_{\ast},$ for all $n \in \mathbb{N}.$ From \eqref{pnF} and \eqref{isof} we have,
 \begin{align}\label{pnF1}
\lim_{n \to \infty} \bar{\mathbb{E}} \bigg[ \bigg| \int_{0}^{t} \int_{Y} \big\langle P_n F(s, \bar{\mathbf{u}}_{n}(s), y) - F(s, \mathbf{u}_{\ast}(s), y), v \big\rangle \,\tilde{\eta}_{\ast}(ds, dy) \bigg|^2 \bigg] = 0.
\end{align}

Moreover, from \eqref{4uni} and \eqref{isof}, with $r=1,$ we obtain, 
\begin{align}\label{dctF}
&\bar{\mathbb{E}} \bigg[ \bigg| \int_{0}^{t} \int_{Y} \big\langle P_n F(s, \bar{\mathbf{u}}_{n}(s), y) - F(s, \mathbf{u}_{\ast}(s), y), v \big\rangle \,\tilde{\eta}_{\ast}(ds, dy) \bigg|^2 \bigg] \nonumber
\\ &= \bar{\mathbb{E}} \bigg[ \int_{0}^{t} \int_{Y} \big| \big( F(s, \bar{\mathbf{u}}_{n}(s), y) - F(s, \mathbf{u}_{\ast}(s), y), v \big)_{\mathbb{H}}\big|^2 \, d\nu(y) \,ds \bigg] \leq C.
\end{align}

Finally, from \eqref{pnF1}, \eqref{dctF} and using Dominated Convergence Theorem we obtain,
\begin{align}
\lim_{n \to \infty} \int_{0}^{T} \bar{\mathbb{E}} \bigg[  \bigg| \int_{0}^{t} \int_{Y} \big\langle P_n F(s, \bar{\mathbf{u}}_{n}(s), y) - F(s, \mathbf{u}_{\ast}(s), y), v \big\rangle \,\tilde{\eta}_{\ast}(ds, dy) \bigg|^2 \,dt \bigg] = 0.
\end{align}
\end{enumerate}
\end{proof}

Now we will give the proof of Lemma \ref{lnterm}.
\begin{proof}
\begin{enumerate}[label=(\alph*)]
\item
 Let us consider
\begin{align}
&\| ( \bar{\mathbf{d}}_{n}(\cdot), v )_{L^2} - ( \mathbf{d}_{\ast}(\cdot), v)_{L^2}\|^{2}_{L^2([0, T] \times \bar{\Omega})} = \int_{\bar{\Omega}} \int_{0}^{T} \big| (\bar{\mathbf{d}}_{n}(t) - \mathbf{d}_{\ast}(t), v)_{L^2}\big|^{2} \, dt \ \bar{\mathbb{P}}(d\omega) \nonumber
\\ &= \bar{\mathbb{E}} \bigg[ \int_{0}^{T} \big| (\bar{\mathbf{d}}_{n}(t) - \mathbf{d}_{\ast}(t), v)_{L^2}\big|^{2} \, dt \bigg]
\end{align}

Moreover,
\begin{align}\label{dddas}
&\int_{0}^{T} \big| (\bar{\mathbf{d}}_{n}(t) - \mathbf{d}_{\ast}(t), v)_{L^2}\big|^{2} \, dt = \int_{0}^{T} \big| _{(H^2)'}\big\langle \bar{\mathbf{d}}_{n}(t) - \mathbf{d}_{\ast}(t), v \big\rangle_{H^2}\big|^{2} \, dt \nonumber
\\ & \ \leq \|v\|_{H^2}^{2} \int_{0}^{T} \big| \bar{\mathbf{d}}_{n}(t) - \mathbf{d}_{\ast}(t) \big|_{(H^2)'}^{2} \,dt
\end{align}

By \eqref{zt2c}, $\bar{\mathbf{d}}_{n} \to \mathbf{d}_{\ast}$ in $\mathbb{C}([0, T]; (H^2)')$ and from \eqref{dnbp}, $\sup_{t \in [0, T]} |\bar{\mathbf{d}}_{n}(t)|_{L^2}^{2} < \infty$, $\bar{\mathbb{P}}$-a.s.. The embedding $L^2 \hookrightarrow (H^2)'$ is continuous. Then by Dominated Convergence Theorem we observe that $\bar{\mathbf{d}}_{n} \to \mathbf{d}_{\ast}$ in $L^2(0, T; (H^2)').$ So from \eqref{dddas},
\begin{align}\label{d1lim}
\lim_{n \to \infty} \int_{0}^{T} \big| (\bar{\mathbf{d}}_{n}(t) - \mathbf{d}_{\ast}(t), v)_{L^2}\big|^{2} \, dt = 0.
\end{align}

Moreover, from \eqref{dsp1}, Proposition \ref{dsL2p} and using H\"older's inequality, for every $n \in \mathbb{N}$ and every $r > 1$ we obtain
\begin{align}\label{d1uni}
&\bar{\mathbb{E}} \bigg[ \bigg| \int_{0}^{T} \big| \bar{\mathbf{d}}_{n}(t) - \mathbf{d}_{\ast}(t) \big|^{2}_{L^2} \,dt \bigg|^{r} \bigg] \leq c \ \bar{\mathbb{E}} \bigg[ \int_{0}^{T} \big( \big| \bar{\mathbf{d}}_{n}(t) \big|^{2r} + \big| \mathbf{d}_{\ast}(t) \big|^{2r} \big) \,dt  \bigg] \nonumber
\\ & \ \leq c \ \bar{\mathbb{E}} \bigg[ \sup_{t \in [0, T]} \big| \bar{\mathbf{d}}_{n}(t) \big|^{2r} \bigg] \leq c \cdot \tilde{C}_{2r, T} < \infty.
\end{align}
for some constant $c > 0.$ Then by \eqref{d1lim}, \eqref{d1uni} and Vitali's Theorem we obtain
\begin{align*}
\lim_{n \to \infty} \bar{\mathbb{E}} \bigg[\int_{0}^{T} \big| (\bar{\mathbf{d}}_{n}(t) - \mathbf{d}_{\ast}(t), v)_{L^2}\big|^{2} \, dt \bigg] = 0, \quad \text{which proves (a)}.
\end{align*}

\item
From \eqref{zt2c}, $\bar{\mathbf{d}}_n \to \mathbf{d}_{\ast}$ in $\mathbb{C}([0, T]; (H^2)') \ \bar{\mathbb{P}}$-a.s. and $\mathbf{d}_{\ast}$ is continuous at $t =0.$ So we obtain $_{(H^2)'}\big\langle \bar{\mathbf{d}}_{n}(0), v \big\rangle_{H^2} \to$  $_{(H^2)'}\big\langle \mathbf{d}_{\ast}(0), v \big\rangle_{H^2} \  \bar{\mathbb{P}}$-a.s. As $L^2 \hookrightarrow (H^2)'$ is continuous, we have
\[\sup_{n \in \mathbb{N}} \bar{\mathbb{E}}\big| \bar{\mathbf{d}}_{n}(0) \big|^{2}_{(H^2)'} \leq \sup_{n \in \mathbb{N}} \bar{\mathbb{E}}\big| \bar{\mathbf{d}}_{n}(0) \big|^{2}_{L^2} \leq C.\]

From \eqref{dnbp} and applying Vitali's Theorem, we get
\begin{align*}
\lim_{n \to \infty} \bar{\mathbb{E}} \bigg[ \big| _{(H^2)'}\big\langle\bar{\mathbf{d}}_{n}(0) - \mathbf{d}_{\ast}(0), v \big\rangle_{H^2} \big|^2 \bigg] = 0
\end{align*}
Hence,
\begin{align*}
\lim_{n \to \infty} \big\| _{(H^2)'}\big\langle \bar{\mathbf{d}}_{n}(0) - \mathbf{d}_{\ast}(0), v \big\rangle_{H^2} \big\|^{2}_{L^2([0, T] \times \bar{\Omega})} = 0.
\end{align*}

which implies
\begin{align}
\lim_{n \to \infty} \big\| \big( \bar{\mathbf{d}}_{n}(0) - \mathbf{d}_{\ast}(0), v \big)_{L^2} \big\|^{2}_{L^2([0, T] \times \bar{\Omega})} = 0.
\end{align}

\item
Now from \eqref{zt2c},  $\bar{\mathbf{d}}_n \to \mathbf{d}_{\ast}$ in $L^{2}(0, T; H^1), \bar{\mathbb{P}}$-a.s., then from \eqref{cala} and for all $v \in H^1 \hookrightarrow H^2$ we obtain $\bar{\mathbb{P}}$-a.s.,
\begin{align}\label{d2lim}
&\lim_{n \to \infty} \int_{0}^{t} \langle \mathcal{A} \bar{\mathbf{d}}_{n}(s), v \rangle \,ds =  \lim_{n \to \infty} \int_{0}^{t} ((\bar{\mathbf{d}}_{n}(s), v )) \,ds \nonumber
\\ &= \int_{0}^{t} (( \mathbf{d}_{\ast}(s), v )) \,ds =  \int_{0}^{t} \langle \mathcal{A}\mathbf{d}_{\ast}(s), v \rangle \,ds
\end{align}

By \eqref{2prop} and using H\"older's inequality we obtain for all $t \in [0, T], r > 2$  and $n \in \mathbb{N},$
\begin{align}\label{d2unii}
&\bar{\mathbb{E}} \bigg[ \bigg| \int_{0}^{t} \langle \mathcal{A}\bar{\mathbf{d}}_{n}(s), v \rangle \,ds \bigg|^{2 + r} \bigg] = \bar{\mathbb{E}} \bigg[ \bigg| \int_{0}^{t} ((\bar{\mathbf{d}}_{n}(s), v )) \,ds \bigg|^{2+r} \bigg] \nonumber
\\ &\leq \bar{\mathbb{E}} \bigg[ \bigg( \int_{0}^{t} \|\bar{\mathbf{d}}_{n}(s)\|_{H^1} \| v \|_{H^1} \,ds \bigg)^{2+r} \bigg] \leq c \ \| v \|^{2+r}_{H^2} \ \bar{\mathbb{E}} \bigg[ \bigg( \int_{0}^{T} \|\bar{\mathbf{d}}_{n}(s)\|_{H^1} \,ds \bigg)^{2+r} \bigg] \nonumber
\\ &\leq \tilde{c} \ \bar{\mathbb{E}} \bigg[ \sup_{s \in [0, T]} \big\|\bar{\mathbf{d}}_{n}(s) \big\|_{H^1}^{2+r} \,ds \bigg]  \leq C
\end{align}

for some constant $C > 0.$ Then by \eqref{d2lim}, \eqref{d2unii} and using Vitali's Theorem we obtain for all $t \in [0, T],$
\begin{align}\label{aunbc}
\lim_{n \to \infty} \bar{\mathbb{E}} \bigg[\bigg| \int_{0}^{t} \langle \mathcal{A} \bar{\mathbf{d}}_{n}(s) - \mathcal{A}\mathbf{d}_{\ast}(s), v \rangle \,ds \bigg|^{2} \bigg] = 0
\end{align}

Since $H^1 \hookrightarrow (H^2)'$ is continuous, from \eqref{2prop}, using Dominated Convergence Theorem, for all $t \in [0, T]$ and all $n \in \mathbb{N}$ we get,
\begin{align}
\lim_{n \to \infty} \int_{0}^{T} \bar{\mathbb{E}} \bigg[ \bigg| \int_{0}^{t} \langle \mathcal{A} \bar{\mathbf{d}}_{n}(s) - \mathcal{A}\mathbf{d}_{\ast}(s), v \rangle \,ds \bigg|^{2} \bigg] \,dt = 0.
\end{align}

\noindent Now we jump to the nonlinear term.

\item
From Lemma \ref{Pnu} and Lemma \ref{bntlim} we have,
\begin{align}\label{dlim3}
&\lim_{n \to \infty} \int_{0}^{t} \big\langle \tilde{B}_n \big( \bar{\mathbf{u}}_{n}(s), \bar{\mathbf{d}}_{n}(s) \big) - \tilde{B} \big( \mathbf{u}_{\ast}(s), \mathbf{d}_{\ast}(s) \big), v \big\rangle \,ds \nonumber
\\ &= \lim_{n \to \infty} \int_{0}^{t} \big\langle \tilde{B} \big( \bar{\mathbf{u}}_{n}(s), \bar{\mathbf{d}}_{n}(s) \big) - \tilde{B} \big( \mathbf{u}_{\ast}(s), \mathbf{d}_{\ast}(s) \big), \tilde{P}_{n} v \big\rangle \,ds = 0 \quad \bar{\mathbb{P}}\text{-a.s.}
\end{align}

\noindent Now from \eqref{btil} and H\"older's inequality, we obtain for all $t \in [0, T], n \in \mathbb{N}$ and $r>1,$
\begin{align}\label{duni3}
&\bar{\mathbb{E}} \bigg[ \bigg| \int_{0}^{t} \big\langle \tilde{B}_n \big( \bar{\mathbf{u}}_{n}(s), \bar{\mathbf{d}}_{n}(s) \big), v \big\rangle \,ds \bigg|^{r} \bigg] \leq \bar{\mathbb{E}} \bigg[ \bigg( \int_{0}^{t} \| \tilde{B}_n \big( \bar{\mathbf{u}}_{n}(s), \bar{\mathbf{d}}_{n}(s) \big)\|_{(H^1)'} \|v\|_{H^1} \,ds \bigg)^{r} \bigg] \nonumber
\\ &\leq \|v\|_{H^2}^{r} \ t^{r-1} \ \bar{\mathbb{E}} \bigg[ \int_{0}^{t} \| \tilde{B}_n \big( \bar{\mathbf{u}}_{n}(s), \bar{\mathbf{d}}_{n}(s) \big)\|^{r}_{(H^1)'} \,ds \bigg] \leq c \ \bar{\mathbb{E}} \bigg[ \int_{0}^{t} | \tilde{B}_n \big( \bar{\mathbf{u}}_{n}(s), \bar{\mathbf{d}}_{n}(s) \big)|^{r}_{L^2} \,ds \bigg] \nonumber
\\ &\leq c \ \bar{\mathbb{E}} \bigg[ \int_{0}^{t} |\bar{\mathbf{u}}_{n}|_{\mathbb{H}}^{r-\frac{r \mathbf{n}}{4}} \|\bar{\mathbf{u}}_{n}\|^{\frac{r \mathbf{n}}{4}} \|\bar{\mathbf{d}}_{n}\|^{r-\frac{r \mathbf{n}}{4}}  |\Delta \bar{\mathbf{d}}_{n}|_{L^2}^{\frac{r \mathbf{n}}{4}} \,ds \bigg],
\end{align}
for $\mathbf{n} = 2, 3.$

First we consider the case $\mathbf{n} = 2.$ Using Young's inequality we have,
\begin{align}\label{d2uni}
&\bar{\mathbb{E}} \bigg[ \int_{0}^{t} (|\bar{\mathbf{u}}_{n}|^{\frac{r}{2}} \|\bar{\mathbf{u}}_{n}\|^{\frac{r}{2}} ) (\|\bar{\mathbf{d}}_{n}\|^{\frac{r}{2}}  |\Delta \bar{\mathbf{d}}_{n}|^{\frac{r}{2}}) \,ds \bigg] \leq \bar{\mathbb{E}} \bigg[ \int_{0}^{t} |\bar{\mathbf{u}}_{n}|^r \|\bar{\mathbf{u}}_{n}\|^r \,ds \bigg] + \bar{\mathbb{E}} \bigg[ \int_{0}^{t} \|\bar{\mathbf{d}}_{n}\|^r  |\Delta \bar{\mathbf{d}}_{n}|^r \,ds \bigg]
\end{align}

Now the estimate for the second term of Right Hand Side follows from \eqref{unin2}. So let us estimate the first term. From \eqref{unb2pg} and \eqref{unbv}, using H\"older's inequality, for $r \in [1, 2]$ and for all $t \in [0, T], n \in \mathbb{N}$ we obtain,
\begin{align}\label{unbrr}
&\bar{\mathbb{E}} \bigg[ \int_{0}^{t} |\bar{\mathbf{u}}_{n}|^r \|\bar{\mathbf{u}}_{n}\|^r \,ds \bigg] \leq \bar{\mathbb{E}} \bigg[ \sup_{s \in [0, T]} |\bar{\mathbf{u}}_{n}|^r \int_{0}^{t} \|\bar{\mathbf{u}}_{n}\|^r \,ds \bigg] \nonumber
\\ &\leq \bigg\{ \bar{\mathbb{E}} \bigg[ \sup_{s \in [0, T]}  |\bar{\mathbf{u}}_{n}|^{2r} \bigg] \bigg\}^{\frac{1}{2}} \bigg\{ \bar{\mathbb{E}} \bigg[ \int_{0}^{t} \|\bar{\mathbf{u}}_{n}\|^r \,ds \bigg]^2 \bigg\}^{\frac{1}{2}} \leq C(r, T).
\end{align}

Similarly, from \eqref{duni3} for $\mathbf{n} = 3,$ using Young's inequality we obtain,
\begin{align}\label{d3uni}
&\bar{\mathbb{E}} \bigg[ \int_{0}^{t} (|\bar{\mathbf{u}}_{n}|^{\frac{r}{4}} \|\bar{\mathbf{u}}_{n}\|^{\frac{3r}{4}} ) (\|\bar{\mathbf{d}}_{n}\|^{\frac{r}{4}}  |\Delta \bar{\mathbf{d}}_{n}|^{\frac{3r}{4}}) \,ds \bigg] \nonumber
\\ &\leq \bar{\mathbb{E}} \bigg[ \int_{0}^{t} |\bar{\mathbf{u}}_{n}|^{\frac{r}{2}} \|\bar{\mathbf{u}}_{n}\|^{\frac{3r}{2}} \,ds \bigg] + \bar{\mathbb{E}} \bigg[ \int_{0}^{t} \|\bar{\mathbf{d}}_{n}\|^{\frac{r}{2}}  |\Delta \bar{\mathbf{d}}_{n}|^{\frac{3r}{2}} \,ds \bigg]
\end{align}

Similarly, the estimate for the second term of Right Hand Side follows from \eqref{unin3}. So we handle only the first term. From \eqref{unb2p} and \eqref{unbv} for $r \in [1, \frac{4}{3})$ and for all $t \in [0, T], n \in \mathbb{N},$ using H\"older's inequality we obtain,
\begin{align}\label{unbr2}
&\bar{\mathbb{E}} \bigg[ \int_{0}^{t} |\bar{\mathbf{u}}_{n}|^{\frac{r}{2}} \|\bar{\mathbf{u}}_{n}\|^{\frac{3r}{2}} \,ds \bigg] \leq \bar{\mathbb{E}} \bigg[  \bigg(\int_{0}^{t} (|\bar{\mathbf{u}}_{n}|^\frac{r}{2})^{\frac{4}{4-3r}} \,ds \bigg)^{\frac{4-3r}{4}} \bigg( \int_{0}^{t} (\|\bar{\mathbf{u}}_{n}\|^{\frac{3r}{2}})^{\frac{4}{3r}} \,ds \bigg)^{\frac{3r}{4}} \bigg] \nonumber
\\ &\leq \bar{\mathbb{E}} \bigg[  \bigg(\int_{0}^{t} |\bar{\mathbf{u}}_{n}|^\frac{2r}{4-3r} \,ds \bigg)^{\frac{4-3r}{4}} \bigg( \int_{0}^{t} \|\bar{\mathbf{u}}_{n}\|^2 \,ds \bigg)^{\frac{3r}{4}} \bigg] \nonumber
\\ &\leq \bigg\{ \bar{\mathbb{E}} \bigg[ \int_{0}^{t} |\bar{\mathbf{u}}_{n}|^\frac{2r}{4-3r} \,ds \bigg] \bigg\}^{\frac{4-3r}{4}} \cdot \bigg\{ \bar{\mathbb{E}} \bigg[ \int_{0}^{t} \|\bar{\mathbf{u}}_{n}\|^2 \,ds \bigg] \bigg\}^{\frac{3r}{4}} \nonumber
\\ &\leq c \ \bigg\{ \bar{\mathbb{E}} \bigg[ \sup_{s \in [0, T]} |\bar{\mathbf{u}}_{n}|^\frac{2r}{4-3r} \bigg] \bigg\}^{\frac{4-3r}{4}} \bigg\{ \bar{\mathbb{E}} \bigg[ \int_{0}^{T} \| \bar{\mathbf{u}}_{n}\|^2 \,ds \bigg] \bigg\}^{\frac{3r}{4}} \leq C(r, T).
\end{align}

So from \eqref{dlim3}, \eqref{duni3}, \eqref{d2uni}, \eqref{unbrr}, \eqref{d3uni}, \eqref{unbr2} and using Vitali's Theorem we obtain for all $t \in [0, T],$
\begin{align}\label{d3vit}
\lim_{n \to \infty} \bar{\mathbb{E}} \bigg[ \bigg| \int_{0}^{t} \big\langle \tilde{B}_n \big( \bar{\mathbf{u}}_{n}(s), \bar{\mathbf{d}}_{n}(s) \big) - \tilde{B} \big( \mathbf{u}_{\ast}(s), \mathbf{d}_{\ast}(s) \big), v \big\rangle \,ds \bigg|^2 \bigg] =0.
\end{align}

From \eqref{unb2p}, \eqref{unbv}, \eqref{dnbv}, \eqref{dnbla} and \eqref{d3vit}, using Dominated Convergence Theorem we obtain,
\begin{align}
\lim_{n \to \infty} \int_{0}^{T} \bar{\mathbb{E}} \bigg[ \bigg| \int_{0}^{t} \big\langle \tilde{B}_n \big( \bar{\mathbf{u}}_{n}(s), \bar{\mathbf{d}}_{n}(s) \big) - \tilde{B} \big( \mathbf{u}_{\ast}(s), \mathbf{d}_{\ast}(s) \big), v \big\rangle \,ds \bigg|^2 \bigg] \,dt =0.
\end{align}
\item
$\bar{\mathbf{d}}_{n} \to \mathbf{d}_{\ast}$ in $\mathcal{Z}_{T, 2}.$ Since $f$ is a polynomial function of order $2N+1,$ 
from Lemma \ref{Pnu} we have $\bar{\mathbb{P}}$-a.s.,
\begin{align}\label{dlimf}
&\lim_{n \to \infty} \int_{0}^{t} \big\langle f_n \big( \bar{\mathbf{d}}_{n}(s) \big) - f \big(\mathbf{d}_{\ast}(s) \big), v \big\rangle \,ds = \lim_{n \to \infty} \int_{0}^{t} \big\langle f \big(\bar{\mathbf{d}}_{n}(s) \big) - f \big(\mathbf{d}_{\ast}(s) \big), \tilde{P}_{n} v \big\rangle \,ds = 0
\end{align}

Since $H^1 \hookrightarrow L^{\bar{q}},$ for $\bar{q}=4N+2,$ where $N \in I_{\mathbf{n}}.$ From Remark \ref{fes} and \eqref{2prop}, we obtain for all $t \in [0, T], r >1$ and $n \in \mathbb{N},$
\begin{align}\label{5duni}
&\bar{\mathbb{E}} \bigg[ \bigg| \int_{0}^{t} \big\langle f_n \big( \bar{\mathbf{d}}_{n}(s) \big), v \big\rangle \,ds \bigg|^{r} \bigg] \leq  \|v\|_{L^2}^{r} \ t^{r-1} \ \bar{\mathbb{E}} \bigg[ \int_{0}^{t} \big| f_n \big( \bar{\mathbf{d}}_{n}(s) \big) \big|_{L^2}^{r} \,ds \bigg] \nonumber
\\ &\leq C + \|v\|_{H^2}^{r} \ t^{r-1} \ \bar{\mathbb{E}} \bigg[ \int_{0}^{t} \big\| \bar{\mathbf{d}}_{n}(s) \big\|_{L^{\bar{q}}}^{\frac{r\bar{q}}{2}} \,ds \bigg] \leq C + C_{r, t} \ \bar{\mathbb{E}} \bigg[ \int_{0}^{t} \big\|\bar{\mathbf{d}}_{n}(s) \big\|_{H^1}^{\frac{r\bar{q}}{2}} \,ds \bigg] \nonumber
\\ &\leq C + \tilde{C}_{r, t} \ \bar{\mathbb{E}} \bigg[ \sup_{s \in [0, T]} \big\|\bar{\mathbf{d}}_{n}(s) \big\|_{H^1}^{\frac{r\bar{q}}{2}}  \bigg] \leq C(r, N, T).
\end{align}

So from \eqref{dlimf}, \eqref{5duni} and using Vitali's Theorem we obtain for all $t \in [0, T],$
\begin{align}\label{dctc}
\lim_{n \to \infty} \bar{\mathbb{E}} \bigg[ \bigg| \int_{0}^{t} \big\langle f_n \big( \bar{\mathbf{d}}_{n}(s) \big) - f \big( \mathbf{d}_{\ast}(s) \big), v \big\rangle \,ds \bigg|^2 \bigg] =0.
\end{align}

Again from \eqref{2prop}, Remark \ref{fes} and using Dominated Convergence Theorem we obtain,
\begin{align}
\lim_{n \to \infty} \int_{0}^{T} \bar{\mathbb{E}} \bigg[ \bigg| \int_{0}^{t} \big\langle f_n \big( \bar{\mathbf{d}}_{n}(s) \big) - f \big( \mathbf{d}_{\ast}(s) \big), v \big\rangle \,ds \bigg|^2 \bigg] \,dt =0.
\end{align}
\end{enumerate}

\end{proof}

\section{Proof of the Main Result}

\begin{theorem}\label{exma}
There exists a martingale solution $\big( \bar{\Omega}, \bar{\mathcal{F}}, \bar{\mathbb{F}}, \bar{\mathbb{P}}, \bar{\mathbf{u}}, \bar{\mathbf{d}}, \bar{\eta} \big)$ of the problem \eqref{r1st}-\eqref{r2nd} provided the assumptions in Subsection \ref{assu} hold.
\end{theorem}

\begin{proof}
From Lemma \ref{knterm}, we have
\begin{align}\label{unbuscon}
\lim_{n \to \infty} \big\| (\bar{\mathbf{u}}_{n}(\cdot), v)_{\mathbb{H}} - (\mathbf{u}_{\ast}(\cdot), v)_{\mathbb{H}} \big\|_{L^2 ([0, T] \times \bar{\Omega})}= 0 
\end{align}
and
\begin{align}\label{knkcon}
\lim_{n \to \infty} \big\| \mathscr{K}_n(\bar{\mathbf{u}}_n, \bar{\mathbf{d}}_n, \bar{\eta}_n, v) - \mathscr{K}(\mathbf{u}_{\ast}, \mathbf{d}_{\ast}, \eta_{\ast}, v)\big\|_{L^2([0, T] \times \bar{\Omega})}= 0.
\end{align}

From Lemma \ref{lnterm}, we obtain
\begin{align}\label{dnbdscon}
\lim_{n \to \infty} \big\| (\bar{\mathbf{d}}_{n}(\cdot), v)_{L^2} - (\mathbf{d}_{\ast}(\cdot), v)_{L^2} \big\|_{L^2([0, T] \times \bar{\Omega})}= 0 
\end{align}
and
\begin{align}\label{lnlcon}
\lim_{n \to \infty} \big\| \Lambda_n(\bar{\mathbf{u}}_n, \bar{\mathbf{d}}_n, v) - \Lambda(\mathbf{u}_{\ast}, \mathbf{d}_{\ast}, v)\big\|_{L^2([0, T] \times \bar{\Omega})}= 0
\end{align} 

Since $(\mathbf{u}_n, \mathbf{d}_n)$ is a solution of the Galerkin approximation equations \eqref{g1st-}-\eqref{g2nd} for all $t \in [0, T]$, we have for $\mathbb{P}$-a.s.
\begin{align*}
(\mathbf{u}_{n}(t), v)_{\mathbb{H}} = \mathscr{K}_n(\mathbf{u}_n, \mathbf{d}_n, \eta_n, v)(t) 
\end{align*}
and
\begin{align*}
(\mathbf{d}_{n}(t), v)_{L^2} = \Lambda_n(\mathbf{u}_n, \mathbf{d}_n, v)(t). 
\end{align*}

In particular,
\begin{align*}
\int_{0}^{T} \mathbb{E} \big[ \big|(\mathbf{u}_{n}(t), v)_{\mathbb{H}} - \mathscr{K}_n(\mathbf{u}_n, \mathbf{d}_n, \eta_n, v)(t) \big|^2 \big] \,dt = 0
\end{align*}
and
\begin{align*}
\int_{0}^{T} \mathbb{E} \big[ \big|(\mathbf{d}_{n}(t), v)_{L^2} - \Lambda_n(\mathbf{u}_n, \mathbf{d}_n, v)(t) \big|^2 \big] \,dt = 0.
\end{align*}

Since $\mathscr{L}(\mathbf{u}_n, \mathbf{d}_n, \eta_n)= \mathscr{L}(\bar{\mathbf{u}}_n, \bar{\mathbf{d}}_n, \bar{\eta}_n),$ we conclude
\begin{align*}
\int_{0}^{T} \bar{\mathbb{E}} \big[ \big|(\bar{\mathbf{u}}_{n}(t), v)_{\mathbb{H}} - \mathscr{K}_n(\bar{\mathbf{u}}_n, \bar{\mathbf{d}}_n, \bar{\eta}_n, v)(t) \big|^2 \big] \,dt =0
\end{align*}
and
\begin{align*}
\int_{0}^{T} \bar{\mathbb{E}} \big[ \big|(\bar{\mathbf{d}}_{n}(t), v)_{L^2} - \Lambda_n(\bar{\mathbf{u}}_n, \bar{\mathbf{d}}_n, v)(t) \big|^2 \big] \,dt =0.
\end{align*}

From \eqref{unbuscon}, \eqref{knkcon}, \eqref{dnbdscon} and \eqref{lnlcon}, we have
\begin{align*}
\int_{0}^{T} \bar{\mathbb{E}} \big[ \big|(\mathbf{u}_{\ast}(t), v)_{\mathbb{H}} - \mathscr{K}(\mathbf{u}_{\ast}, \mathbf{d}_{\ast}, \eta_{\ast}, v)(t) \big|^2 \big] \,dt =0
\end{align*}
and
\begin{align*}
\int_{0}^{T} \bar{\mathbb{E}} \big[ \big|(\mathbf{d}_{\ast}(t), v)_{L^2} - \Lambda(\mathbf{u}_{\ast}, \mathbf{d}_{\ast}, v)(t) \big|^2 \big] \,dt =0.
\end{align*}

Hence for $l$-almost all $t \in [0, T]$ and $\bar{\mathbb{P}}$-almost all $\omega \in \bar{\Omega},$ we obtain
\begin{align*}
(\mathbf{u}_{\ast}(t), v)_{\mathbb{H}} - \mathscr{K}(\mathbf{u}_{\ast}, \mathbf{d}_{\ast}, \eta_{\ast}, v)(t) = 0
\end{align*}
and
\begin{align*}
(\mathbf{d}_{\ast}(t), v)_{L^2} - \Lambda(\mathbf{u}_{\ast}, \mathbf{d}_{\ast}, v)(t) = 0.
\end{align*}

In particular,
\begin{align}\label{eqnus}
(\mathbf{u}_{\ast}(t), v)_{\mathbb{H}}  &+ \int_{0}^{t} \big\langle \mathscr{A} \mathbf{u}_{\ast}(s), v \big\rangle \, ds + \int_{0}^{t} \big\langle B(\mathbf{u}_{\ast}(s)), v \big\rangle \,ds + \int_{0}^{t} \big\langle M(\mathbf{d}_{\ast}(s)), v \big\rangle \,ds \nonumber
\\ &= \big( \mathbf{u}_{\ast}(0), v \big)_{\mathbb{H}} + \int_{0}^{t} \int_{Y} \big( F(s, \mathbf{u}_{\ast}(s); y), v\big)_{\mathbb{H}} \, \tilde{\eta}_{\ast}(ds, dy)
\end{align}
and
\begin{align}\label{eqnds}
(\mathbf{d}_{\ast}(t), v)_{L^2}  &+ \int_{0}^{t} \big\langle \mathcal{A} \mathbf{d}_{\ast}(s), v \big\rangle \, ds + \int_{0}^{t} \big\langle \tilde{B}(\mathbf{u}_{\ast}(s), \mathbf{d}_{\ast}(s)), v \big\rangle \,ds \nonumber
\\ &= \big( \mathbf{d}_{\ast}(0), v \big)_{L^2} - \int_{0}^{t} \big\langle f(\mathbf{d}_{\ast}(s)), v \big\rangle \,ds
\end{align}

Since $(\mathbf{u}_{\ast}, \mathbf{d}_{\ast})$ is $\mathcal{Z}_{T, 1} \times \mathcal{Z}_{T, 2}$-valued random variable, and $\mathbf{u}_{\ast}, \mathbf{d}_{\ast}$ are weakly c\`adl\`ag, and weakly continuous respectively, we obtain that the equalities \eqref{eqnus}-\eqref{eqnds} hold for all $t \in [0, T]$ and all $v \in \mathbb{V}$ and $v \in H^2$ respectively. Putting $\bar{\mathbf{u}} := \mathbf{u}_{\ast},$ $\bar{\mathbf{d}} := \mathbf{d}_{\ast}$ and $\bar{\eta} := \eta_{\ast},$ we infer that the system $\big( \bar{\Omega}, \bar{\mathcal{F}}, \bar{\mathbb{F}}, \bar{\mathbb{P}}, \bar{\mathbf{u}}, \bar{\mathbf{d}}, \bar{\eta} \big)$ is a martingale solution of \eqref{r1st}-\eqref{r2nd}.  

\end{proof}

\section{Pathwise Uniqueness and Existence of Strong Solution in 2-D}\label{pu2}

In this section we prove the pathwise uniqueness of the weak solutions of \eqref{r1st}-\eqref{r2nd}. Then we use results from \cite{Ond}, for existence of strong solution of \eqref{r1st}-\eqref{r2nd} as well. We consider these cases only in two dimensions.

In the following Lemma we will show that almost all trajectories of the solution $(\mathbf{u}, \mathbf{d})$ are almost everywhere equal to a $\mathbb{V} \times H^2$-valued function defined on [$0, T$].

\begin{lemma}\label{val}
Let the assumptions from Subsection \ref{assu} holds. Let $\big( \bar{\Omega}, \bar{\mathcal{F}}, \bar{\mathbb{F}}, \bar{\mathbb{P}}, \bar{\mathbf{u}}, \bar{\mathbf{d}}, \bar{\eta} \big)$ be a martingale solution of \eqref{r1st}-\eqref{r2nd}. Let $(\mathbf{u}_0, \mathbf{d}_0) \in \mathbb{H} \times H^1.$ Then for $\bar{\mathbb{P}}$-almost all $\omega \in \bar{\Omega},$ the trajectories $\bar{\mathbf{u}}(\cdot, \omega)$ is almost everywhere equal to a c\`adl\`ag  $\mathbb{V}$-valued function and $\bar{\mathbf{d}}(\cdot, \omega)$ is almost everywhere equal to a continuous $L^2$-valued function defined on $[0, T]$.
\end{lemma}

\begin{proof}
From previous results we have for $t \in [0, T],$
\begin{align}\label{firu}
\bar{\mathbf{u}}(t) &= \bar{\mathbf{u}}_{0} - \int_{0}^{t} \mathscr{A} \bar{\mathbf{u}}(s) \,ds - \int_{0}^{t} B(\bar{\mathbf{u}}(s)) \,ds - \int_{0}^{t} M(\bar{\mathbf{d}}(s)) \,ds + \int_{0}^{t} \int_{Y} F(s, \bar{\mathbf{u}}(s), y) \,\tilde{\eta}(ds, dy) 
\end{align} 
and
\begin{align}\label{secd}
\bar{\mathbf{d}}(t) = \bar{\mathbf{d}}_{0} - \int_{0}^{t} \mathcal{A} \bar{\mathbf{d}}(s) \,ds - \int_{0}^{t} \tilde{B}(\bar{\mathbf{u}}(s), \bar{\mathbf{d}}(s)) \,ds - \int_{0}^{t} f(\bar{\mathbf{d}}(s)) \,ds
\end{align}

We need to verify the drift terms in the RHS of \eqref{firu} are $\mathbb{V}'$-valued. Then from \ref{Ah1} and Proposition \ref{ugrdn} we get,
\begin{align}
\bar{\mathbb{E}} \int_{0}^{T} \big|\mathscr{A} \bar{\mathbf{u}}(s)\big|^2_{\mathbb{V}'} \,ds \leq \bar{\mathbb{E}}  \int_{0}^{T} \| \bar{\mathbf{u}}(s)\|^2 \, ds < \infty.
\end{align}

From \ref{BHH} and Proposition \ref{ugrdn} we have,
\begin{align}
\bar{\mathbb{E}} \int_{0}^{T} \big| B (\bar{\mathbf{u}}(s)) \big|^2_{\mathbb{V}'} \,ds \leq \bar{\mathbb{E}} \int_{0}^{T} | \bar{\mathbf{u}}(s)|^{2}_{\mathbb{H}} \, ds \leq \bar{\mathbb{E}} \big[\sup_{s \in [0, T]} | \bar{\mathbf{u}}(s)|^{2}_{\mathbb{H}}\big] < \infty.
\end{align}

From \eqref{md1d2}, Proposition \ref{Ldn} and Corollary \ref{codnh} we obtain,
\begin{align}
\bar{\mathbb{E}} \int_{0}^{T} \big| M (\bar{\mathbf{d}}(s)) \big|^2_{\mathbb{V}'} \,ds \leq \bar{\mathbb{E}} \int_{0}^{T} \|\bar{\mathbf{d}}(s)\|^{4-\mathbf{n}} | \Delta \bar{\mathbf{d}}(s)|^{\mathbf{n}} \, ds 
\end{align}

For $\mathbf{n} = 2,$ applying Young's inequality we get,
\begin{align}
&\bar{\mathbb{E}} \int_{0}^{T} \|\bar{\mathbf{d}}(s)\|^{2} | \Delta \bar{\mathbf{d}}(s)|^{2} \, ds \leq \bar{\mathbb{E}} \bigg[ \sup_{s \in [0, T]} \|\bar{\mathbf{d}}(s)\|^{2} \int_{0}^{T} | \Delta \bar{\mathbf{d}}(s)|^{2} \, ds \bigg] \nonumber
\\ &\leq c \ \bar{\mathbb{E}} \bigg[ \sup_{s \in [0, T]} \|\bar{\mathbf{d}}(s)\|^{2} \bigg]^2 + c \ \bar{\mathbb{E}} \bigg[ \int_{0}^{T}  | \Delta \bar{\mathbf{d}}(s)|^{2} \, ds \bigg]^2 < \infty.
\end{align}

Now we need to show $F$ is $\mathbb{H}$-valued. So using It\^o isometry, \eqref{lgf} and Proposition \ref{ugrdn} we obtain,
\begin{align}
&\bar{\mathbb{E}} \int_{0}^{T} \bigg| \int_{Y} F(s, \bar{\mathbf{u}}(s), y) \,\tilde{\eta}(ds, dy) \bigg|^2_{\mathbb{H}} \leq \bar{\mathbb{E}} \int_{0}^{T} \int_{Y} \big| F(s, \bar{\mathbf{u}}(s), y) \big|^2_{\mathbb{H}} \, \nu(dy)ds \nonumber
\\ & \leq \bar{\mathbb{E}} \int_{0}^{T} \big( 1+ |\bar{\mathbf{u}}(s)|^2_{\mathbb{H}}\big) \,ds \leq C_T + \bar{\mathbb{E}} \bigg[\sup_{s \in [0, T]} | \bar{\mathbf{u}}(s)|^{2}_{\mathbb{H}}\bigg] < \infty.
\end{align}

Now we consider the second equation \eqref{secd}. We need to show the RHS of \eqref{secd} is $L^2$-valued. Then from Proposition \ref{dgrdn} and Proposition \ref{Ldn} we have,
\begin{align}
\bar{\mathbb{E}} \int_{0}^{T} \big|\mathcal{A} \bar{\mathbf{d}}(s)\big|^2_{L^2} \,ds \leq \bar{\mathbb{E}} \int_{0}^{T} \| \bar{\mathbf{d}}(s) \|^2_{H^2} \,ds < \infty.
\end{align}

From \eqref{btil}, Proposition \ref{ugrdn} and Proposition \ref{Ldn} we get,
\begin{align}
\bar{\mathbb{E}} \int_{0}^{T} \big|\tilde{B} (\bar{\mathbf{u}}(s), \bar{\mathbf{d}}(s)) \big|^2_{L^2} \,ds \leq c\ \bar{\mathbb{E}} \int_{0}^{T} \big( |\bar{\mathbf{u}}(s)|^{2-\frac{\mathbf{n}}{2}} \|\bar{\mathbf{u}}(s)\|^{\frac{\mathbf{n}}{2}} \|\bar{\mathbf{d}}(s)\|^{2-\frac{\mathbf{n}}{2}} |\Delta \bar{\mathbf{d}}(s)|^{\frac{\mathbf{n}}{2}}\big) \,ds 
\end{align}

For $\mathbf{n}=2,$
\begin{align}
&\bar{\mathbb{E}} \int_{0}^{T} \big( |\bar{\mathbf{u}}(s)| \ \|\bar{\mathbf{u}}(s)\| \ \|\bar{\mathbf{d}}(s)\| \ |\Delta \bar{\mathbf{d}}(s)| \big) \,ds \leq \bar{\mathbb{E}} \int_{0}^{T} |\bar{\mathbf{u}}(s)|^2 \|\bar{\mathbf{u}}(s)\|^2 + \bar{\mathbb{E}} \int_{0}^{T} \|\bar{\mathbf{d}}(s)\|^2 |\Delta \bar{\mathbf{d}}(s)|^2 \,ds \nonumber
\\ &\leq c \ \bar{\mathbb{E}} \bigg[ \sup_{s \in [0, T]} |\bar{\mathbf{u}}(s)|^2 \int_{0}^{T} \|\bar{\mathbf{u}}(s)\|^2 \,ds \bigg] + c \ \bar{\mathbb{E}} \bigg[ \sup_{s \in [0, T]} \|\bar{\mathbf{d}}(s)\|^2 \int_{0}^{T} |\Delta \bar{\mathbf{d}}(s)|^2 \,ds \bigg] \nonumber
\\ &\leq c \ \bar{\mathbb{E}} \bigg( \sup_{s \in [0, T]} |\bar{\mathbf{u}}(s)|^2 \bigg)^2 + c \ \bar{\mathbb{E}} \bigg( \int_{0}^{T} \|\bar{\mathbf{u}}(s)\|^2 \,ds \bigg)^2  \nonumber
\\ & \qquad \quad + c \ \bar{\mathbb{E}} \bigg( \sup_{s \in [0, T]} \|\bar{\mathbf{d}}(s)\|^2 \bigg)^2 + c \ \bar{\mathbb{E}} \bigg( \int_{0}^{T} |\Delta \bar{\mathbf{d}}(s)|^2 \,ds \bigg)^2 < \infty.
\end{align}

Now from Corollary \ref{codnh} and Remark \ref{fes} we obtain for $\bar{q}= 4N+2$,
\begin{align}
\bar{\mathbb{E}} \int_{0}^{T} | f( \bar{\mathbf{d}}(s))|^2_{L^2} \,ds \leq C_T + \bar{\mathbb{E}} \int_{0}^{T} \| \bar{\mathbf{d}}(s)) \|^{\bar{q}}_{L^{\bar{q}}} \,ds \leq C_T + \bar{\mathbb{E}} \sup_{s \in [0, T]} \| \bar{\mathbf{d}}(s)) \|^{\bar{q}}_{H^1} < \infty.
\end{align}

Thus the proof is complete.

\end{proof}

In the following lemma we show that the solutions of \eqref{r1st}-\eqref{r2nd} are pathwise unique. We use  Gy\"ongy and Krylov's version of It\^o's formula (see \cite{GK}) for a suitable function in this proof.

\begin{lemma}\label{paun}
Let us assume that $(\mathbf{u}_1, \mathbf{d}_1)$ and $(\mathbf{u}_2, \mathbf{d}_2)$ are the two solutions of \eqref{r1st}-\eqref{r2nd} defined on the same stochastic system $(\Omega, \mathcal{F}, \mathcal{F}_t, \mathbb{P}, \tilde{\eta}),$ with the same initial condition $(\mathbf{u}_0, \mathbf{d}_0) \in \mathbb{H} \times H^1.$ Then 
\begin{align}\label{pu}
(\mathbf{u}_1(t), \mathbf{d}_1(t)) = (\mathbf{u}_2(t), \mathbf{d}_2(t)) \quad \mathbb{P}\text{-a.s.} \quad \text{for any} \ t \in (0, T].
\end{align}
\end{lemma}

\begin{proof}
Let us denote the norms as per section \ref{bd} in this proof. Let $\mathbf{u}(t) = \mathbf{u}_1(t) - \mathbf{u}_2(t)$ and $\mathbf{d}(t) = \mathbf{d}_1(t) - \mathbf{d}_2(t),$ with $(\mathbf{u}(0), \mathbf{d}(0)) = (0, 0).$ Let us denote $F_d(t, y) := \big(F(t, \mathbf{u}_1(t); y)-F(t, \mathbf{u}_2(t); y)\big)$. These processes satisfy
\begin{align*}
d\mathbf{u}(t) &+ \Big( \mathscr{A}\mathbf{u}(t) + B(\mathbf{u}(t), \mathbf{u}_1(t)) + B(\mathbf{u}_2(t), \mathbf{u}(t))\Big) \,dt  
\\ &= - \Big( M(\mathbf{d}(t), \mathbf{d}_1(t)) + M(\mathbf{d}_2(t), \mathbf{d}(t)) \Big) \,dt + \int_{Y} F_d(t, y) \,\tilde{\eta}(dt, dy),
\end{align*}
and
\begin{align*}
d\mathbf{d}(t) &+ \Big( \mathcal{A}\mathbf{d}(t) + \tilde{B}(\mathbf{u}(t), \mathbf{d}_1(t)) + \tilde{B}(\mathbf{u}_2(t), \mathbf{d}(t))\Big) \,dt = - \Big( f(\mathbf{d}_2(t)) - f(\mathbf{d}_1(t)) \Big) \,dt
\end{align*}

From \eqref{buvvd}, \eqref{md1d2}, \eqref{btil} and using Poincar\'e and Young's inequalities, we obtain for any $\kappa_3 > 0, \kappa_4 > 0,\kappa_5 > 0, \kappa_6 > 0, \kappa_7 > 0$ and $\kappa_8 > 0,$ there exist $C(\kappa_3)  >0, C(\kappa_4, \kappa_5) > 0, C(\kappa_6, \kappa_8) >0$ and  $C(\kappa_7) > 0$ such that 
\begin{align}\label{bmmbt}
&|\langle B(\mathbf{u}, \mathbf{u}_1), \mathbf{u}\rangle| \leq \kappa_3 \|\mathbf{u}\|^2 + C(\kappa_3) | \mathbf{u}_1|^2 \| \mathbf{u}_1\|^2 | \mathbf{u}|^2, \nonumber
\\ & |\langle M(\mathbf{d}_2, \mathbf{d}), \mathbf{u}\rangle| \leq \kappa_4 \|\mathbf{u}\|^2 + \kappa_5 | \Delta \mathbf{d}|^2 + C(\kappa_4, \kappa_5) \|\mathbf{d}_2\|^2 | \Delta \mathbf{d}_2|^2 \|\mathbf{d}\|^2, \nonumber
\\ & |\langle M(\mathbf{d}, \mathbf{d}_1), \mathbf{u}\rangle| \leq \kappa_8 \|\mathbf{u}\|^2 + \kappa_6 | \Delta \mathbf{d}|^2 + C(\kappa_6, \kappa_8) \|\mathbf{d}_1\|^2 | \Delta \mathbf{d}_1|^2 \|\mathbf{d}\|^2, \nonumber
\\ &|\langle \tilde{B}(\mathbf{u}_2, \mathbf{d}), \Delta \mathbf{d}\rangle| \leq \kappa_7 | \Delta \mathbf{d}|^2 + C(\kappa_7) | \mathbf{u}_2|^2 \|\mathbf{u}_2\|^2 \|\mathbf{d}\|^2.
\end{align}

From Gagliardo-Nirenberg inequality and from the Sobolev embedding $H^2 \subset L^{\infty},$ we obtain for any $\kappa_9 > 0$ there exists $C(\kappa_9)>0$ such that
\begin{align*}
|\langle \tilde{B}(\mathbf{u}, \mathbf{d}_1), \mathbf{d}\rangle| &\leq |\mathbf{u}| \ \|\mathbf{d}_1\| \ \|\mathbf{d}\|_{L^{\infty}} \leq \kappa_9 |\Delta \mathbf{d}|^2 + C(\kappa_9) |\mathbf{u}|^2 \|\mathbf{d}_1\|^2.
\end{align*} 

Now let us define
\[\Upsilon(t) := \exp\bigg(- 2 \int_{0}^{t} (\xi_1(s) + \xi_2(s) + \xi_3(s)) \,ds \bigg), \quad \text{for any} \ t > 0.\]
where
\begin{align*}
&\xi_1(s) := C(\kappa_3) |\mathbf{u}_1(s)|^2 \|\mathbf{u}_1(s)\|^2 + C(\kappa_9) \|\mathbf{d}_1(s)\|^2,
\\ &\xi_2(s) := C(\kappa_4, \kappa_5) \|\mathbf{d}_2(s)\|^2 | \Delta \mathbf{d}_2(s)|^2 + C(\kappa_6, \kappa_8) \|\mathbf{d}_1(s)\|^2 | \Delta \mathbf{d}_1(s)|^2 
\\ &\qquad\qquad+ C(\kappa_7) |\mathbf{u}_2(s)|^2 \|\mathbf{u}_2(s)\|^2 + C_1(\kappa_2) \beta(\mathbf{d}_1, \mathbf{d}_2),
\\ &\xi_3(s) := \big(C(\kappa_1) + C_2(\kappa_2) \big) \beta(\mathbf{d}_1, \mathbf{d}_2),
\end{align*}
where $\beta(\mathbf{d}_1, \mathbf{d}_2)$ is defined in \eqref{betdef} of Appendix.

 Now applying It\^o's formula to $\Upsilon(t) |\mathbf{d}(t)|^2,$ we obtain
\begin{align}\label{dU1}
d\big[\Upsilon(t) |\mathbf{d}(t)|^2\big] &= -2 \Upsilon(t) \Big[ \|\mathbf{d}(t)\|^2 + \big\langle \tilde{B}(\mathbf{u}(t), \mathbf{d}_1(t)), \mathbf{d}(t) \big\rangle \nonumber
\\ & \ \ \ + \big\langle f(\mathbf{d}_2(t)) - f(\mathbf{d}_1(t)), \mathbf{d}(t) \big\rangle \Big] \,dt + \Upsilon '(t) |\mathbf{d}(t)|^2 \,dt.
\end{align}

In this proof we will use It\^o's formula due to \cite{GK}. So applying It\^o's formula to $\Upsilon(t) \|\mathbf{d}(t)\|^2$ and $\Upsilon(t) |\mathbf{u}(t)|^2$ we get,
\begin{align}\label{dU2}
d\big[\Upsilon(t) \|\mathbf{d}(t)\|^2\big] &= 2 \Upsilon(t) \Big[ -|\Delta \mathbf{d}(t)|^2 + \big\langle \tilde{B}(\mathbf{u}(t), \mathbf{d}_1(t))+ \tilde{B}(\mathbf{u}_2(t), \mathbf{d}(t)), \Delta \mathbf{d}(t) \big\rangle \nonumber
\\ & \ \ + \big\langle f(\mathbf{d}_2(t)) - f(\mathbf{d}_1(t)), \Delta \mathbf{d}(t) \big\rangle \Big] \,dt + \Upsilon '(t) \|\mathbf{d}(t)\|^2 \,dt,
\end{align}
and
\begin{align}\label{dU3}
d\big[\Upsilon(t) |\mathbf{u}(t)|^2\big] &= -2 \Upsilon(t) \Big[ \|\mathbf{u}(t)\|^2 + \big\langle B(\mathbf{u}(t), \mathbf{u}_1(t)) + M(\mathbf{d}(t), \mathbf{d}_1(t)), \mathbf{u}(t) \big\rangle \nonumber
\\ & \ \ \ + \big\langle  M(\mathbf{d}_2(t), \mathbf{d}(t)), \mathbf{u}(t) \big\rangle \Big] \,dt  \nonumber
\\ & \ \ \ + 2 \Upsilon(t) \Big[ \int_{Y} \big( F_d(t, y), \mathbf{u}(t)\big)_{\mathbb{H}} \,\tilde{\eta}(dt, dy) + \int_{Y}  \big|F_{d}(t, y) \big|^2_{\mathbb{H}} \,\nu(dy)ds \Big] \nonumber
\\ & \ \ \ + \Upsilon '(t) |\mathbf{u}(t)|^2 \,dt.
\end{align}

Now adding \eqref{dU1}, \eqref{dU2} and \eqref{dU3}, using the inequalities \eqref{bmmbt}, Assumption (B) and Lemma \ref{fdflbet} we get,
\begin{align}\label{5s}
d&\Big[ \Upsilon(t) \big( |\mathbf{u}(t)|^2 + |\mathbf{d}(t)|^2 + \|\mathbf{d}(t)\|^2 \big) \Big]+ 2 \Upsilon(t) \Big[ \|\mathbf{u}(t)\|^2 + \|\mathbf{d}(t)\|^2 + |\Delta \mathbf{d}(t)|^2 \Big] \,dt \nonumber
\\ &\leq 2 \Upsilon(t) \Big[ \xi_1(t) |\mathbf{u}(t)|^2 + \xi_2(t) |\mathbf{d}(t)|^2 + \xi_3(t) \|\mathbf{d}(t)\|^2 \Big] \,dt \nonumber
\\ & \ \ \ + 2 \Upsilon(t) \Big[ \big( \kappa_2 + \kappa_9 + \sum_{i=5}^{7} \kappa_i \big) |\Delta \mathbf{d}(t)|^2 \,dt + \int_{Y} \big( F_d(t, y), \mathbf{u}(t)\big)_{\mathbb{H}} \,\tilde{\eta}(dt, dy) \Big] \nonumber
\\ & \ \ \ + 2 \Upsilon(t) \Big[ L |\mathbf{u}(t)|^2 + (\kappa_3 + \kappa_4 + \kappa_8) \|\mathbf{u}(t)\|^2 + \kappa_1 \|\mathbf{d}(t)\|^2 \Big] \,dt \nonumber
\\ & \ \ \ + \Upsilon '(t) \Big[ |\mathbf{u}(t)|^2 + |\mathbf{d}(t)|^2 + \|\mathbf{d}(t)\|^2 \Big] \,dt.
\end{align}

\noindent By the choice of $\Upsilon$ we have
\begin{align*}
2 \Upsilon(t) \Big[ \xi_1(t) |\mathbf{u}(t)|^2 &+ \xi_2(t) |\mathbf{d}(t)|^2 + \xi_3(t) \|\mathbf{d}(t)\|^2 \Big] + \Upsilon '(t) \Big[ |\mathbf{u}(t)|^2 + |\mathbf{d}(t)|^2 + \|\mathbf{d}(t)\|^2 \Big] \leq 0.
\end{align*}

So dropping the above term from the right hand side of \eqref{5s}, then choosing $\kappa_2 = \kappa_9 = \kappa_i = \frac{1}{10},$ for $i=5, 6, 7.$ $\kappa_3 = \kappa_4 = \kappa_8 = \frac{1}{6}, \kappa_1=\frac{1}{2}$ and rearranging we obtain,
\begin{align}
d&\Big[ \Upsilon(t) \big( |\mathbf{u}(t)|^2 + |\mathbf{d}(t)|^2 + \|\mathbf{d}(t)\|^2 \big) \Big]+ \Upsilon(t) \Big[ \|\mathbf{u}(t)\|^2 + \|\mathbf{d}(t)\|^2 + |\Delta \mathbf{d}(t)|^2 \Big] \,dt \nonumber
\\ &\leq 2 \Upsilon(t) \Big[ C \big( |\mathbf{u}(t)|^2 + |\mathbf{d}(t)|^2 + \|\mathbf{d}(t)\|^2 \big) \,dt + \int_{Y} \big( F_d(t, y), \mathbf{u}(t)\big)_{\mathbb{H}} \,\tilde{\eta}(dt, dy) \Big]
\end{align}

Now integrating both side and taking mathematical expectation we get
\begin{align}
&\mathbb{E} \Big[ \Upsilon(t) \big( |\mathbf{u}(t)|^2 + |\mathbf{d}(t)|^2 + \|\mathbf{d}(t)\|^2 \big) \Big] + \mathbb{E} \bigg[ \int_{0}^{t} \Upsilon(s) \big( \|\mathbf{u}(s)\|^2 + \|\mathbf{d}(s)\|^2 + |\Delta \mathbf{d}(s)|^2 \big) \,ds \bigg] \nonumber
\\ & \ \ \ \leq C \int_{0}^{t} \mathbb{E} \Big[ \Upsilon(s) \big( |\mathbf{u}(s)|^2 + |\mathbf{d}(s)|^2 + \|\mathbf{d}(s)\|^2 \big)\Big] \,ds.
\end{align}

Now applying Gronwall's inequality we obtain \eqref{pu}.

\end{proof}

\begin{definition}
Let $\big( \Omega_i, \mathcal{F}_i, \mathbb{F}_i, \mathbb{P}_i, \big\{ \mathbf{u}_{i}(t)\big\}_{t \geq 0}, \big\{ \mathbf{d}_{i}(t)\big\}_{t \geq 0}, \big\{ \tilde{\eta}_i(t, \cdot) \big\}_{t \geq 0} \big)$ for $i=1, 2$ be two martingale solutions of \eqref{r1st}-\eqref{r2nd} with $\mathbf{u}_{i}(0) = \mathbf{u}_{0}$ and $\mathbf{d}_{i}(0) = \mathbf{d}_{0}; i= 1, 2.$ Then the solutions are said to be unique in law if
\[ \mathscr{L}_{\mathbb{P}_1}(\mathbf{u}_{1}) = \mathscr{L}_{\mathbb{P}_2}(\mathbf{u}_{2}) \ \ \ \text{on}  \ \ \ L^{2}(0, T; \mathbb{V}) \cap \mathbb{D}([0, T]; \mathbb{H})\]
and
\[ \mathscr{L}_{\mathbb{P}_1}(\mathbf{d}_{1}) = \mathscr{L}_{\mathbb{P}_2}(\mathbf{d}_{2}) \ \ \ \text{on}  \ \ \ L^{2}(0, T; H^2) \cap \mathbb{C}([0, T]; H^1),\]
where $\mathscr{L}_{\mathbb{P}_i}(\mathbf{u}_{i})$ and $\mathscr{L}_{\mathbb{P}_i}(\mathbf{d}_{i})$ for $i = 1, 2$ are probability measures on $L^{2}(0, T; \mathbb{V}) \cap \mathbb{D}([0, T]; \mathbb{H})$ and $L^{2}(0, T; H^2) \cap \mathbb{C}([0, T]; H^1)$ respectively.
\end{definition}

\begin{corollary}
Let $\mathbf{n} = 2$ be the dimension. Let the assumptions from Subsection \ref{assu} hold. Then
\begin{enumerate}
\item
There exists a pathwise unique strong solution of \eqref{r1st}-\eqref{r2nd}.
\item
Moreover, if $(\Omega, \mathcal{F}, \mathcal{F}_t, \mathbb{P}, \mathbf{u}, \mathbf{d}, \tilde{\eta})$ is a strong solution of \eqref{r1st}-\eqref{r2nd} then for $\mathbb{P}$-almost all $\omega \in \Omega$ the trajectories $\mathbf{u}(\cdot, \omega)$ is almost everywhere equal to a c\`adl\`ag  $\mathbb{V}$-valued function and $\mathbf{d}(\cdot, \omega)$ is almost everywhere equal to a continuous $H^2$-valued function defined on $[0, T].$
\item
The martingale solution of \eqref{r1st}-\eqref{r2nd} is unique in law. 
\end{enumerate}
\end{corollary}

\begin{proof}
The existence of a martingale solution is shown in Theorem \ref{exma}. From lemma \ref{paun} we obtained the solutions are pathwise unique. Thus the first assertion follows from Theorem 2 of Ondrej\'at \cite{Ond}. The second assertion is a direct consequence of Lemma \ref{val}. The third assertion follows from Theorems 2, 11 of \cite{Ond}.  
\end{proof}

\begin{appendix}

\section{Time Homogeneous Poisson Random Measure}

We follow the approach due to \cite{BHs, BHa}. Also refer Ikeda and Watanabe \cite{IW}, Peszat and Zabczyk \cite{PZ} for details. Let $\bar{\mathbb{N}}$ denote the set of all extended natural numbers, i.e., $\bar{\mathbb{N}}:=\mathbb{N} \cup \{\infty\}$ and $\mathbb{R}^{+}:=[0,\infty).$ Let $(S, \mathscr{S})$ be a measurable space and $M_{\bar{\mathbb{N}}}(S)$ be the set of all $\bar{\mathbb{N}}$-valued measures on $(S, \mathscr{S}).$
On the set $M_{\bar{\mathbb{N}}}(S)$ we consider the $\sigma$-field $\mathscr{M}_{\bar{\mathbb{N}}}(S)$ defined as the smallest $\sigma$-field such that for all $B\in \mathscr{S},$ the map $$i_B:M_{\bar{\mathbb{N}}}(S) \ni \mu \rightarrow \mu(B) \in \bar{\mathbb{N}}$$ is measurable.

Let $(\Omega,\mathcal{F},\mathbb{F},\mathbb{P})$ be a filtered probability space with a filtration $\mathbb{F} = \{\mathcal{F}_t\}_{t \geq 0}$ satisfying the usual hypothesis and this probability space satisfies the usual conditions, i.e.
 \begin{enumerate}
		\item[(i)] $\mathbb{P}$ is complete on $(\Omega, \mathcal{F})$,
		\item[(ii)] for each $t\geq 0$, $\mathcal{F}_t$ contains all $(\mathcal{F},\mathbb{P})$-null sets,
		\item[(iii)] the filtration $\mathbb{F}$ is right-continuous.
\end{enumerate}

\begin{definition}
(See Appendix C of \cite{BHs})\\ Let $(Y, \mathscr{B}(Y))$ be a measurable space. A {\bf time homogeneous Poisson random measure} $\eta$ on $(Y, \mathscr{B}(Y))$ over $(\Omega,\mathcal{F},\mathcal{F}_t,\mathbb{P})$ is a measurable function
$$ \eta\,:\, (\Omega,\mathscr{F}) \rightarrow (M_{\bar{\mathbb{N}}}(\mathbb{R}^+ \times Y),\mathscr{M}_{\bar{\mathbb{N}}}(\mathbb{R}^+ \times Y))$$
such that
\begin{itemize}
\item[(a)] for each $B\in\mathscr{B}(\mathbb{R}^+) \otimes \mathscr{B}(Y), \eta(B):=i_B \circ \eta :\Omega \rightarrow \bar{\mathbb{N}}$ is a Poisson random variable with parameter $\mathbb{E}[\eta(B)];$ 
\item[(b)] $\eta$ is independently scattered, i.e., if the sets $B_1,B_2,\ldots,B_n \in \mathscr{B}(\mathbb{R}^+)\otimes \mathscr{B}(Y) $ are disjoint then the random variables $\eta(B_1), \eta(B_2),\ldots, \eta(B_n)$ are mutually independent;
\item[(c)] for all $U \in \mathscr{B}(Y)$ the $\bar{\mathbb{N}}-$valued process $(N(t,U))_{t \geq 0}$ defined by
$$ N(t,U):= \eta((0,t]\times U),\quad t \geq 0$$  is $\mathcal{F}_t$-adapted and its increments are independent of the past,i.e., if $t>s \geq 0,$ then $N(t,U) - N(s,U) = \eta((s,t]\times U)$ is independent of $\mathcal{F}_s.$
\end{itemize}
\end{definition}

If $\eta$ is a time homogeneous Poisson random measure then the formula 
$$\nu(A):=\mathbb{E}[\eta((0,1] \times A)],\quad A \in \mathscr{B}(Y)$$ defines a measure on $(Y, \mathscr{B}(Y))$ called an intensity measure of $\eta.$
Moreover, for all $T<\infty$ and all $A \in \mathscr{B}(Y)$ such that $\mathbb{E}[\eta((0,T] \times A)]< \infty,$ the $\mathbb{R}-$valued process $\{\tilde{N}(t,A)\}_{t \in [0,T]}$ defined by
$$\tilde{N}(t,A):=\eta((0,t]\times A) -t \, \nu(A),\quad t \in (0,T],$$
is an integrable martingale on $(\Omega,\mathcal{F},\mathcal{F}_t,\mathbb{P}).$ The random measure $l\otimes\nu$ on $\mathscr{B}(\mathbb{R}^+)\otimes \mathscr{B}(Y),$ where $l$ stands for the Lebesgue measure, is called a compensator of $\eta$ and the difference between a time homogeneous Poisson random measure $\eta$ and its compensator, i.e.,  
$$ \tilde{\eta}:=\eta-l\otimes\nu,$$ is called a \textbf{compensated time homogeneous Poisson random measure}.
  
 We follow the notion of the first author and E. Hausenblas \cite{BHs}. Also see \cite{IW} and \cite{PZ}, to list some of the basic properties of the stochastic integral with respect to $\tilde{\eta}.$
Let $E$ be a separable Hilbert space and let $\mathcal{P}$ be the progressively measurable $\sigma$-field on $[0,T]\times\Omega.$ Let $\mathfrak{L}^2_{\nu,T}(\mathcal{P}\otimes\mathscr{B}(Y), l \otimes \mathbb{P}\otimes \nu; E)$ be a space of all $E$-valued, $\mathcal{P}\otimes\mathscr{B}(Y)$-measurable processes such that $$\mathbb{E}\bigg[\int_0^T \int_Y \|\xi(s,\cdot,y)\|_E^2 \,d \nu (y) \,ds \bigg] <\infty. $$ 

If $\xi \in \mathfrak{L}^2_{\nu,T}(\mathcal{P}\otimes\mathscr{B}(Y), l \otimes \mathbb{P}\otimes \nu; E)$ then the integral process ${\int_0^T \int_Y} \xi(s,\cdot,y)\\\tilde{\eta}(ds,dy),$ $t \in[0,T],$ is a c\`{a}dl\`{a}g $L^2$-integrable martingale. Moreover, we have the following \textbf{isometry formula}
\begin{align}\label{isof}
\mathbb{E}\bigg[\bigg\| \int_0^T \int_Y \xi(s,\cdot,y) \,\tilde{\eta}(ds,dy)\bigg\|_E^2\bigg]=\mathbb{E}\bigg[\int_0^T \int_Y \|\xi(s,\cdot,y)\|_E^2 \,d \nu(y) \,ds \bigg],\, t\in[0,T].
\end{align}

\section{Some Important Inequalities}

In this section we recall some important inequalities which are needed in our proof of main result. 

Let $\mathbf{n} = 2, 3.$ Take $a = \frac{\mathbf{n}}{4}.$ Then for all $\mathbf{u} \in W^{1, 4},$ the following special cases of Gagliardo-Nirenberg inequalities hold,
\begin{align*}
&\|\mathbf{u}\|_{L^4} \leq \|\mathbf{u}\|^{1-a}_{L^2}\| \nabla \mathbf{u}\|^{a}_{L^2},
\end{align*}
which can be viewed in terms of the continuous embedding $H^1 \subset L^4.$ Also we have
\begin{align*}
\|\mathbf{u}\|_{L^{\infty}} \leq \|\mathbf{u}\|^{1-a}_{L^4}\| \nabla \mathbf{u}\|^{a}_{L^4}.
\end{align*}

So from above observations, for all $\mathbf{u} \in D(\mathcal{A})$ we have,
\begin{align*}
\|\mathbf{u}\|_{L^{\infty}} \leq \|\mathbf{u}\|^{1-a}_{H^1}\| \nabla \mathbf{u}\|^{a}_{H^2}.
\end{align*}

\begin{lemma}\label{burk}
{\bf (Burkholder-Davis-Gundy Inequality)} Let $M$ be a Hilbert space valued c\`{a}dl\`{a}g martingale with $M_0=0$ and let $p\geq 1$ be fixed. Then for any
$\mathcal{F}$-stopping time $\mathbf{\tau}$, there exists constants $c_p$ and $C_p$ such that
$$\mathbb{E}\left\{[M]_{\mathbf{\tau}}^{p/2}\right\}\leq c_p \, \mathbb{E}\left\{\sup_{0\leq t\leq\mathbf{\tau}}\|M_t\|_H^p\right\}\leq C_p \,\mathbb{E}\left\{[M]_{\mathbf{\tau}}^{p/2}\right\}$$ for all $\mathbf{\tau}$, $0\leq
\mathbf{\tau}\leq \infty$, where $[M]$ is the quadratic variation of process $M$. The constants are universal (independent of $M$).
\end{lemma}
\begin{proof}
For the real-valued c\`{a}dl\`{a}g martingales see Theorem 3.50 of Peszat and Zabczyk \cite{PZ}.
\end{proof}

\begin{remark}\label{sep}
We will show the existence of the countable family of real valued continuous functions which are defined on $\mathcal{Z}_{T}$ and separate points of this space.
\end{remark}
\begin{enumerate}
\item
We know $L^2(0, T; \mathbb{H})$ and $\mathbb{D}([0, T]; \mathbb{V}')$ are completely metrizable and separable spaces, we deduce that there exists a countable family of continuous real valued functions on each of these spaces which separate points. For example see \cite{Bad}, expos\'e 8. 
\item
For the space $L^{2}_{w}(0, T; \mathbb{V})$ we define
\[ g_m(\mathbf{u}) := \int_{0}^{T} (( \mathbf{u}(t), v_m(t))) \, dt \in \mathbb{R}, \qquad \mathbf{u} \in L^{2}(0, T; \mathbb{V}), \quad m \in \mathbb{N},\]
where $\{ v_m, m \in \mathbb{N}\}$ is a dense subset of $ L^{2}(0, T; \mathbb{V}).$ then $(g_m)_{m \in \mathbb{N}}$ is a sequence of continuous real valued functions separating points of the space $L^{2}_{w}(0, T; \mathbb{V}).$  
\item
Let $\mathbb{H}_0 \subset \mathbb{H}$ be a countable and dense subset of $\mathbb{H}.$ Then for each $h \in \mathbb{H}_0$ the mapping
\[\mathbb{D}([0, T]; \mathbb{H}_w) \ni \mathbf{u} \mapsto (\mathbf{u}(\cdot), h)_{\mathbb{H}} \in \mathbb{D}([0, T]; \mathbb{R})\]
is continuous. Since $\mathbb{D}([0, T]; \mathbb{R})$ is a separable complete metric space, there exists a sequence $(f_l)_{l \in \mathbb{N}}$ of real valued continuous functions defined on $\mathbb{D}([0, T]; \mathbb{R})$ separating points of this space. Then the mappings $m_{h, l},$ where $h \in H_0, l \in \mathbb{N}$ defined by
\[ m_{h, l}(\mathbf{u}) := f_l(( \mathbf{u}(\cdot), h)_{\mathbb{H}}, \qquad \mathbf{u} \in \mathbb{D}([0, T]; \mathbb{H}_w),\]
form a countable family of continuous functions on $\mathbb{D}([0, T]; \mathbb{H}_w)$ which separates points of this space. 
\end{enumerate}

Similarly we can define for the space $\mathcal{Z}_{T, 2}.$

\section{Some convergence results}

\begin{lemma}\label{bntlim}
Let $\mathbf{u}_{\ast} \in L^2(0, T; \mathbb{V})$ and $\mathbf{d}_{\ast} \in L^2(0, T; H^1).$ Let $(\bar{\mathbf{u}}_{n})_{n \in \mathbb{N}}$ and $(\bar{\mathbf{d}}_{n})_{n \in \mathbb{N}}$ are bounded sequences in $L^2(0, T; \mathbb{V})$ and $L^2(0, T; H^1)$ respecitvely such that $\bar{\mathbf{u}}_{n} \to \mathbf{u}_{\ast}$ in $L^2(0, T; \mathbb{V})$ and $\bar{\mathbf{d}}_{n} \to \mathbf{d}_{\ast}$ in $L^2(0, T; H^1).$ Then for all $v \in H^2$ and all $t \in [0, T],$
\begin{align*}
\lim_{n \to \infty} \int_{0}^{t} \big\langle \tilde{B} \big( \bar{\mathbf{u}}_{n}(s), \bar{\mathbf{d}}_{n}(s)\big), v \big\rangle \,ds =  \int_{0}^{t} \big\langle \tilde{B} \big( \mathbf{u}_{\ast}(s), \mathbf{d}_{\ast}(s) \big), v \big\rangle \,ds.
\end{align*}
\end{lemma}
\begin{proof}
Let $v \in H^1 \hookrightarrow H^2.$ Then for $\mathbf{u}, \mathbf{w} \in H^1,$ from \eqref{buw} we have
\begin{align}\label{Buwv}
\big| _{(H^1)'}\big\langle \tilde{B}(\mathbf{u}, \mathbf{w}), v \big\rangle_{H^1} \big| \leq \|\tilde{B}(\mathbf{u}, \mathbf{w})\|_{(H^1)'} \| v \|_{H^1} \leq \| \mathbf{u}\|_{H^1} \| \mathbf{w}\|_{H^1} \|v\|_{H^1}.
\end{align}
Moreover,
\begin{align}\label{Bunus}
\tilde{B} \big( \bar{\mathbf{u}}_{n}, \bar{\mathbf{d}}_{n} \big) - \tilde{B} \big( \mathbf{u}_{\ast}, \mathbf{d}_{\ast} \big) = \tilde{B} \big( \bar{\mathbf{u}}_{n}-\mathbf{u}_{\ast}, \bar{\mathbf{d}}_{n} \big) + \tilde{B} \big( \mathbf{u}_{\ast}, \bar{\mathbf{d}}_{n}-\mathbf{d}_{\ast} \big).
\end{align}
Then from \eqref{Buwv}, \eqref{Bunus} and using H\"older's inequality, we obtain
\begin{align}
&\bigg| \int_{0}^{t} \big\langle \tilde{B} \big( \bar{\mathbf{u}}_{n}(s), \bar{\mathbf{d}}_{n}(s) \big), v \big\rangle \,ds - \int_{0}^{t} \big\langle \tilde{B} \big( \mathbf{u}_{\ast}(s), \mathbf{d}_{\ast}(s) \big), v \big\rangle \,ds \bigg| \nonumber
\\ &\leq \bigg| \int_{0}^{t} \big\langle \tilde{B} \big( \bar{\mathbf{u}}_{n}(s)-\mathbf{u}_{\ast}(s), \bar{\mathbf{d}}_{n}(s) \big), v \big\rangle \,ds \bigg| + \bigg| \int_{0}^{t} \big\langle \tilde{B} \big( \mathbf{u}_{\ast}(s), \bar{\mathbf{d}}_{n}(s)-\mathbf{d}_{\ast}(s) \big), v \big\rangle \,ds \bigg| \nonumber
\\ &\leq \bigg( \int_{0}^{t} \|\bar{\mathbf{u}}_{n}(s)-\mathbf{u}_{\ast}(s)\|_{\mathbb{V}} \|\bar{\mathbf{d}}_{n}(s)\|_{H^1} \,ds + \int_{0}^{t} \|\mathbf{u}_{\ast}(s)\|_{\mathbb{V}} \|\bar{\mathbf{d}}_{n}(s)-\mathbf{d}_{\ast}(s)\|_{H^1} \,ds \bigg) \|v\|_{H^1} \nonumber
\\ &\leq c \ \big( \|\bar{\mathbf{u}}_{n}-\mathbf{u}_{\ast}\|_{L^2(0, T; \mathbb{V})} \|\bar{\mathbf{d}}_{n}\|_{L^2(0, T; H^1)} + \|\mathbf{u}_{\ast}\|_{L^2(0, T; \mathbb{V})} \|\bar{\mathbf{d}}_{n}-\mathbf{d}_{\ast}\|_{L^2(0, T; H^1)} \big) \|v\|_{H^1}.
\end{align}

Now from \eqref{unbv}, \eqref{usv2}, \eqref{dnbv}, \eqref{dsv2} and using the fact that $\bar{\mathbf{u}}_{n} \to \mathbf{u}_{\ast}$ in ${L^2(0, T; \mathbb{V})}$ and $\bar{\mathbf{d}}_{n} \to \mathbf{d}_{\ast}$ in $L^2(0, T; H^1)$ and $H^1$ is dense in $H^2,$ we conclude for all $v \in H^2$
\begin{align}\label{liBun}
\lim_{n \to \infty} \int_{0}^{t} \big\langle B \big( \bar{\mathbf{u}}_{n}(s), \bar{\mathbf{u}}_{n}(s) \big), v \big\rangle \,ds =  \int_{0}^{t} \big\langle B \big( \mathbf{u}_{\ast}(s), \mathbf{u}_{\ast}(s) \big), v \big\rangle \,ds.
\end{align}
\end{proof}

\begin{lemma}\label{bnlim}
Let $\mathbf{u}_{\ast} \in L^2(0, T; \mathbb{V}).$  Let $(\bar{\mathbf{u}}_{n})_{n \in \mathbb{N}}$ is a bounded sequence in $L^2(0, T; \mathbb{V})$  such that $\bar{\mathbf{u}}_{n} \to \mathbf{u}_{\ast}$ in $L^2(0, T; \mathbb{V}).$ Then for all $v \in \mathbb{V}$ and all $t \in [0, T],$
\begin{align*}
\lim_{n \to \infty} \int_{0}^{t} \big\langle B \big( \bar{\mathbf{u}}_{n}(s)\big), v \big\rangle \,ds =  \int_{0}^{t} \big\langle B \big( \mathbf{u}_{\ast}(s)\big), v \big\rangle \,ds.
\end{align*}
\end{lemma}
\begin{proof}
The proof is similar to the proof of previous Lemma.
\end{proof}

\begin{lemma}\label{mlimi}
Let $\mathbf{d}_{\ast} \in L^2(0, T; H^2)$ and let $(\bar{\mathbf{d}}_{n})_{n \in \mathbb{N}}$ be a bounded sequence in $L^2(0, T; H^2)$ such that $\bar{\mathbf{d}}_{n} \to \mathbf{d}_{\ast}$ in $L^2(0, T; H^2).$ Then for all $v \in \mathbb{V}$ and all $t \in [0, T],$
\begin{align*}
\lim_{n \to \infty} \int_{0}^{t} \big\langle M \big( \bar{\mathbf{d}}_{n}(s)\big), v \big\rangle \,ds =  \int_{0}^{t} \big\langle M \big( \mathbf{d}_{\ast}(s) \big), v \big\rangle \,ds.
\end{align*}
\end{lemma}
\begin{proof}
Assume that $v \in \mathbb{V}.$ Then for $\mathbf{d}_1, \mathbf{d}_2 \in H^2,$ from \eqref{md1d2} we have
\begin{align}\label{Muwv}
\big| _{\mathbb{V}'}\big\langle M(\mathbf{d}_1, \mathbf{d}_2), v \big\rangle_{\mathbb{V}} \big| \leq \|M(\mathbf{d}_1, \mathbf{d}_2)\|_{\mathbb{V}'} \| v \|_{\mathbb{V}} \leq \| \mathbf{d}_1\|_{H^2} \| \mathbf{d}_2\|_{H^2} \|v\|_{\mathbb{V}}.
\end{align}
Moreover,
\begin{align}\label{Munus}
M \big( \bar{\mathbf{d}}_{n}, \bar{\mathbf{d}}_{n} \big) - M \big( \mathbf{d}_{\ast}, \mathbf{d}_{\ast} \big) = M \big( \bar{\mathbf{d}}_{n}-\mathbf{d}_{\ast}, \bar{\mathbf{d}}_{n} \big) + M \big( \mathbf{d}_{\ast}, \bar{\mathbf{d}}_{n}-\mathbf{d}_{\ast} \big).
\end{align}
Then from \eqref{Muwv}, \eqref{Munus} and using H\"older's inequality, we obtain
\begin{align}
&\bigg| \int_{0}^{t} \big\langle M \big( \bar{\mathbf{d}}_{n}(s), \bar{\mathbf{d}}_{n}(s) \big), v \big\rangle \,ds - \int_{0}^{t} \big\langle M \big( \mathbf{d}_{\ast}(s), \mathbf{d}_{\ast}(s) \big), v \big\rangle \,ds \bigg| \nonumber
\\ &\leq \bigg| \int_{0}^{t} \big\langle M \big( \bar{\mathbf{d}}_{n}(s)-\mathbf{d}_{\ast}(s), \bar{\mathbf{d}}_{n}(s) \big), v \big\rangle \,ds \bigg| + \bigg| \int_{0}^{t} \big\langle M \big( \mathbf{d}_{\ast}(s), \bar{\mathbf{d}}_{n}(s)-\mathbf{d}_{\ast}(s) \big), v \big\rangle \,ds \bigg| \nonumber
\\ &\leq \bigg( \int_{0}^{t} \|\bar{\mathbf{d}}_{n}(s)-\mathbf{d}_{\ast}(s)\|_{H^2} \|\bar{\mathbf{d}}_{n}(s)\|_{H^2} \,ds + \int_{0}^{t} \|\mathbf{d}_{\ast}(s)\|_{H^2} \|\bar{\mathbf{d}}_{n}(s)-\mathbf{d}_{\ast}(s)\|_{H^2} \,ds \bigg) \|v\|_{\mathbb{V}} \nonumber
\\ &\leq c \  \|\bar{\mathbf{d}}_{n}-\mathbf{d}_{\ast}\|_{L^2(0, T; H^2)}\big( \|\bar{\mathbf{d}}_{n}\|_{L^2(0, T; H^2)} + \|\mathbf{d}_{\ast}\|_{L^2(0, T; H^2)} \big) \|v\|_{\mathbb{V}}.
\end{align}

Now from \eqref{c0}, \eqref{ldn} for $p=2, q=1$ and using the fact that $\bar{\mathbf{d}}_{n} \to \mathbf{d}_{\ast}$ in ${L^{2}_{w}(0, T; H^2)}$ we conclude for all $v \in \mathbb{V}$
\begin{align}\label{liBun}
\lim_{n \to \infty} \int_{0}^{t} \big\langle M \big( \bar{\mathbf{d}}_{n}(s), \bar{\mathbf{d}}_{n}(s) \big), v \big\rangle \,ds =  \int_{0}^{t} \big\langle M \big( \mathbf{d}_{\ast}(s), \mathbf{d}_{\ast}(s) \big), v \big\rangle \,ds.
\end{align}
\end{proof}

\section{Some important estimates}

\begin{proposition}\label{ush2p}
Let $\mathbf{u}_{\ast}$ be the process as defined in \eqref{zt1c}. Then for $p \geq 1,$ we have
\[ \bar{\mathbb{E}} \bigg[ \sup_{s \in [0,T]} \big|\mathbf{u}_{\ast}(s) \big|_{\mathbb{H}}^{2p} \bigg] < C_p.\]
\end{proposition}

\begin{proof}
From \eqref{unb2p} we have, $\big(\bar{\mathbf{u}}_{n}\big)_{n \in \mathbb{N}}$ is uniformly bounded in $L^{2p}(\bar{\Omega}; L^{\infty}(0,T;\mathbb{H})).$
Since the dual of $L^{2p}(\bar{\Omega}; L^{\infty}(0,T; \mathbb{H}))$ is $\big(L^\frac{2p}{2p-1}(\bar{\Omega}; L^{1}(0,T; \mathbb{H}))\big)',$ by Banach-Alaoglu Theorem, there exists a subsequence of $\bar{\mathbf{u}}_{n}$, again denoted by the same and there exists $v \in L^{2p}(\bar{\Omega}; L^{\infty}(0,T; \mathbb{H}))$ such that 
$\bar{\mathbf{u}}_{n}$ convergent weakly-star to $v$ in $L^{2p}(\bar{\Omega}; L^{\infty}(0,T; \mathbb{H})).$ In particular,
\[ \big\langle \bar{\mathbf{u}}_{n}, \phi \big\rangle \rightharpoonup \big\langle v, \phi \big\rangle, \quad \phi \in L^\frac{2p}{2p-1}(\bar{\Omega}; L^{1}(0,T; \mathbb{H})).\]
i.e.,
\begin{align}\label{etuv}
\int_{\bar{\Omega}} \int_{0}^{T} \big\langle \bar{\mathbf{u}}_{n}, \phi \big\rangle \,dt \,d\mathbb{P}(\omega) \to \int_{\bar{\Omega}} \int_{0}^{T} \big\langle v, \phi \big\rangle \,dt \,d\mathbb{P}(\omega) 
\end{align}

Again we have $\bar{\mathbf{u}}_{n}$ convergent weakly to $\mathbf{u}_{\ast}$ in $L^2(\bar{\Omega}; L^{2}(0,T; \mathbb{V})).$ So by using the compact embedding $\mathbb{V} \subset \mathbb{H},$ we have
\begin{align}\label{etup}
\bar{\mathbb{E}} \bigg[\int_{0}^{T} \big(\bar{\mathbf{u}}_{n}(t,\omega), \phi(t,\omega) \big)_{\mathbb{H}} \, dt \bigg] \rightarrow \bar{\mathbb{E}} \bigg[\int_{0}^{T} &\big(\mathbf{u}_{\ast}(t,\omega), \phi(t,\omega) \big)_{\mathbb{H}} \, dt \bigg] \nonumber
\\ &\forall\,\,\phi \in L^{2}(\bar{\Omega}; L^{2}(0,T; \mathbb{H})).
\end{align} 

For $p \geq 1, \, L^{2}(\bar{\Omega}; L^{2}(0,T; \mathbb{H})) $ is dense subspace of $L^{\frac{2p}{2p-1}}(\bar{\Omega}; L^{1}(0,T; \mathbb{H})).$ 

From \eqref{etuv} and \eqref{etup} we have
\begin{align*}
\bar{\mathbb{E}} \bigg[ \int_0^T (v(t,\omega),\phi(t,\omega))_{\mathbb{H}} \,dt \bigg] = \bar{\mathbb{E}}\bigg[ \int_0^T&(\mathbf{u}_{\ast}(t,\omega),\phi(t,\omega))_{\mathbb{H}} \,dt \bigg]
\\& \forall \, \phi \in L^{2}(\bar{\Omega}; L^{2}(0,T; \mathbb{H})).
\end{align*}
Thus we have, $\mathbf{u}_{\ast} = v$ and $\mathbf{u}_{\ast} \in L^{2p}(\bar{\Omega}; L^{2}(0,T; \mathbb{H})),$ i.e.,
\begin{align}
\bar{\mathbb{E}} \bigg[\sup_{s \in [0,T]}\big|\mathbf{u}_{\ast}(s) \big|_{\mathbb{H}}^{2p} \bigg] < C_p
\end{align} 
\end{proof}

\begin{proposition}\label{dsL2p}
Let $\mathbf{d}_{\ast}$ be the process as defined in \eqref{zt2c}. Then for $p \geq 1,$ we have
\[ \bar{\mathbb{E}} \bigg[ \sup_{s \in [0,T]} \big|\mathbf{d}_{\ast}(s) \big|_{L^2}^{2p} \bigg] < C(p).\]
\end{proposition}
\begin{proof}
Similar proof as Proposition \ref{ush2p}.
\end{proof}

\begin{lemma}
Let the Assumption (C) of section \ref{assu} holds. Then there exists a positive constant $C$ such that for any $\mathbf{d}_1, \mathbf{d}_2 \in H^2,$ we obtain
\begin{align*}
\|f(\mathbf{d}_1)-f(\mathbf{d}_2)\|_{H^1} \leq C \big( 1+ \big\|\mathbf{d}_1 \big\|^{2N}_{H^2} + \big\|\mathbf{d}_2 \big\|^{2N}_{H^2} \big) \|\mathbf{d}_1-\mathbf{d}_2\|_{H^2}.
\end{align*}
\end{lemma}
\begin{proof}
Since $H^2$ is an algebra, for any $\mathbf{d}_1, \mathbf{d}_2 \in H^2$ and $r \geq 0,$ we have
\begin{align*}
\big\||\mathbf{d}_1|^{2r} \mathbf{d}_2 \big\|_{H^2} \leq c_1 \, \big\|\mathbf{d}_1 \big\|_{H^2}^{2r} \big\|\mathbf{d}_2 \big\|_{H^2}.
\end{align*}
Further, by Young's inequality we obtain
\begin{align}\label{dH2}
\big\|\mathbf{d}_1 \big\|_{H^2}^{2r} \big\|\mathbf{d}_2 \big\|_{H^2} \leq c_2 \, \big\|\mathbf{d}_2 \big\|_{H^2} \big(1 + \big\|\mathbf{d}_1 \big\|_{H^2}^{2N} \big), \quad \text{for} \ r = 0, \cdots, N.
\end{align}

So it is sufficient to prove the lemma for the leading term $b_N |\mathbf{d}|^{2N} \mathbf{d}.$ Further we have
\begin{align*}
|\mathbf{d}_1|^{2N} \mathbf{d}_1 - |\mathbf{d}_2|^{2N} \mathbf{d}_2 = |\mathbf{d}_1|^{2N} (\mathbf{d}_1 - \mathbf{d}_2) + \mathbf{d}_2 (|\mathbf{d}_1|-|\mathbf{d}_2|) \big( \sum_{k=0}^{2N-1} |\mathbf{d}_1|^{2N-k-1} |\mathbf{d}_2|^k \big).
\end{align*}

Then by using Young's inequality several times we get,
\begin{align}\label{d2N}
&\big||\mathbf{d}_1|^{2N} \mathbf{d}_1 - |\mathbf{d}_2|^{2N} \mathbf{d}_2 \big| \leq |\mathbf{d}_1 - \mathbf{d}_2| |\mathbf{d}_1|^{2N} + |\mathbf{d}_2| \big||\mathbf{d}_1|-|\mathbf{d}_2| \big| \big( \sum_{k=0}^{2N-1} |\mathbf{d}_1|^{2N-k-1} |\mathbf{d}_2|^k \big) \nonumber
\\&\leq |\mathbf{d}_1 - \mathbf{d}_2| \bigg[ |\mathbf{d}_1|^{2N} + |\mathbf{d}_2| \big( \sum_{k=0}^{2N-1} |\mathbf{d}_1|^{2N-k-1} |\mathbf{d}_2|^k \big) \bigg] \leq C_N \, |\mathbf{d}_1 - \mathbf{d}_2| \big( 1+ |\mathbf{d}_1|^{2N} + |\mathbf{d}_2|^{2N}\big).
\end{align}

Finally from \eqref{dH2} and \eqref{d2N} we infer that the lemma holds for the leading term.
\end{proof}

\begin{lemma}\label{fdflbet}
For any $\kappa_1 >0$ and $\kappa_2 >0,$ there exist $C(\kappa_1) > 0, C_1(\kappa_2) >0$ and $C_2(\kappa_2) >0$ such that
\begin{align}\label{fdbeta}
\big| \big\langle f(\mathbf{d}_1)-f(\mathbf{d}_2), \mathbf{d}_1 - \mathbf{d}_2\big\rangle\big| \leq \kappa_1 \|\nabla \mathbf{d}_1 - \nabla \mathbf{d}_2\|^2_{L^2} + C(\kappa_1) \|\mathbf{d}_1 - \mathbf{d}_2\|^2_{L^2} \beta(\mathbf{d}_1, \mathbf{d}_2),
\end{align}
and
\begin{align}\label{flbeta}
&\big| \big\langle f(\mathbf{d}_1)-f(\mathbf{d}_2), \Delta \mathbf{d}_1 - \Delta \mathbf{d}_2 \big\rangle\big| \nonumber
\\ &\leq \kappa_2 \|\Delta \mathbf{d}_1 - \Delta \mathbf{d}_2\|^2_{L^2} + \big \{ C_1(\kappa_2) \|\nabla \mathbf{d}_1 - \nabla \mathbf{d}_2\|^2_{L^2} + C_2(\kappa_2) \|\mathbf{d}_1 - \mathbf{d}_2\|^2_{L^2} \big\} \beta(\mathbf{d}_1, \mathbf{d}_2),
\end{align}
where 
\begin{align}\label{betdef}
\beta(\mathbf{d}_1, \mathbf{d}_2) := C \big( 1 +  \|\mathbf{d}_1\|^{2N}_{L^{4N+2}} + \|\mathbf{d}_2\|^{2N}_{L^{4N+2}} \big)^2.
\end{align}
\end{lemma}

\begin{proof}
From \eqref{d2N} we obtain
\begin{align*}
\big| \big\langle f(\mathbf{d}_1)-f(\mathbf{d}_2), \mathbf{d}_1 - \mathbf{d}_2\big\rangle\big| \leq C \int_{\mathbb{O}} \big( 1+ |\mathbf{d}_1|^{2N} + |\mathbf{d}_2|^{2N}\big) |\mathbf{d}_1 - \mathbf{d}_2|^2 \,dx.
\end{align*}

Then using H\"older's, Gagliardo-Nirenberg  and Young's inequalities and using the fact that $L^{4N+2} \subset L^{4N},$ we get
\begin{align*}
&\big| \big\langle f(\mathbf{d}_1)-f(\mathbf{d}_2), \mathbf{d}_1 - \mathbf{d}_2\big\rangle\big| \leq C \bigg\{\int_{\mathbb{O}} \big( |\mathbf{d}_1 - \mathbf{d}_2|^2 \big)^2 \,dx \bigg\}^{\frac{1}{2}} \bigg\{\int_{\mathbb{O}} \big( 1+ |\mathbf{d}_1|^{2N} + |\mathbf{d}_2|^{2N}\big)^2 \,dx \bigg\}^{\frac{1}{2}}
\\ &\leq C \bigg\{\int_{\mathbb{O}} |\mathbf{d}_1 - \mathbf{d}_2|^4 \,dx \bigg\}^{\frac{1}{4} \cdot 2} \bigg\{\int_{\mathbb{O}} \big( 1+ |\mathbf{d}_1|^{4N} + |\mathbf{d}_2|^{4N} \big) \,dx \bigg\}^{\frac{1}{4N} \cdot 2N}
\\ &\leq C \|\mathbf{d}_1 - \mathbf{d}_2\|^2_{L^4}  \big( 1+ \|\mathbf{d}_1 \|_{L^{4N+2}}^{2N} + \|\mathbf{d}_2 \|_{L^{4N+2}}^{2N}\big) 
\\ &\leq  C \|\mathbf{d}_1 - \mathbf{d}_2\|_{L^2} \| \nabla \mathbf{d}_1 - \nabla \mathbf{d}_2\|_{L^2} \big( 1+ \|\mathbf{d}_1 \|_{L^{4N+2}}^{2N} + \|\mathbf{d}_2 \|_{L^{4N+2}}^{2N}\big)
\\ &\leq \kappa_1 \| \nabla \mathbf{d}_1 - \nabla \mathbf{d}_2\|^2_{L^2} + C(\kappa_1) \|\mathbf{d}_1 - \mathbf{d}_2\|^2_{L^2} \big( 1+ \|\mathbf{d}_1 \|_{L^{4N+2}}^{2N} + \|\mathbf{d}_2 \|_{L^{4N+2}}^{2N}\big)^2.
\end{align*} 

which proves \eqref{fdbeta}. Using the similar argument as in \eqref{fdbeta}, using H\:\"older's inequality with exponents $2, 4N+2, \frac{4N+2}{2N}$ and using the fact that $H^1 \subset L^{4N+2}$ for $N \in \mathbb{N}$ we obtain
\begin{align*}
&\big| \big\langle f(\mathbf{d}_1)-f(\mathbf{d}_2), \Delta \mathbf{d}_1 - \Delta \mathbf{d}_2\big\rangle\big| \leq C \int_{\mathbb{O}} \big( 1+ |\mathbf{d}_1|^{2N} + |\mathbf{d}_2|^{2N}\big) |\mathbf{d}_1 - \mathbf{d}_2| | \Delta \mathbf{d}_1 - \Delta \mathbf{d}_2| \,dx
\\ &\leq \|\mathbf{d}_1 - \mathbf{d}_2\|_{L^{4N+2}} \| \Delta \mathbf{d}_1 - \Delta \mathbf{d}_2\|_{L^2} \big( 1+ \|\mathbf{d}_1 \|_{L^{4N+2}}^{2N} + \|\mathbf{d}_2 \|_{L^{4N+2}}^{2N}\big)
\\ &\leq \|\mathbf{d}_1 - \mathbf{d}_2\|_{H^1} \| \Delta \mathbf{d}_1 - \Delta \mathbf{d}_2\|_{L^2} \big( 1+ \|\mathbf{d}_1 \|_{L^{4N+2}}^{2N} + \|\mathbf{d}_2 \|_{L^{4N+2}}^{2N}\big)
\\ &\leq \kappa_2 \| \Delta \mathbf{d}_1 - \Delta \mathbf{d}_2\|^2_{L^2} + C(\kappa_2) \|\mathbf{d}_1 - \mathbf{d}_2\|^2_{H^1}  \big( 1+ \|\mathbf{d}_1 \|_{L^{4N+2}}^{2N} + \|\mathbf{d}_2 \|_{L^{4N+2}}^{2N}\big)^2.
\end{align*}
from which we get \eqref{flbeta}.
\end{proof}
\end{appendix}

\end{document}